\documentclass[
a4paper 
,11pt 
,leqno 
,english 
]{scrartcl} 

\usepackage{my-preamble-article} 
\usepackage{my-macros} 
\usepackage{my-specific-macros} 
\usepackage{my-counters-related-preprints} 

\title{Global relaxation of bistable solutions for gradient systems in one unbounded spatial dimension}
\author{Emmanuel \textsc{Risler}}
\begin{document}
\maketitle
\begin{abstract}
This paper is concerned with parabolic gradient systems of the form 
\[
u_t=-\nabla V (u) + \mathcal{D}  u_{xx}\,,
\]
where the spatial domain is the whole real line, the state variable $u$ is multidimensional, $\mathcal{D}$ denotes a fixed diffusion matrix, and the potential $V$ is coercive at infinity. \emph{Bistable} solutions, that is solutions close at both ends of space to stable homogeneous equilibria, are considered. For a solution of this kind, it is proved that, if the equilibria approached at both ends belong to the same level set of the potential and if an appropriate (localized in space) energy remains bounded from below as time increases, then the solution approaches, as time goes to infinity, a pattern of profiles of stationary solutions homoclinic or heteroclinic to stable homogeneous equilibria, going slowly away from one another. This result provides a step towards a complete description of the global behaviour of all bistable solutions that is pursued in a companion paper. Some consequences are derived, and applications to some examples are given.
\end{abstract}
\nnfootnote{%
}
\thispagestyle{empty} 
\pagestyle{empty}
\hypersetup{pageanchor=false} 
\newpage
\tableofcontents
\newpage
\hypersetup{pageanchor=true} 
\pagestyle{plain}
\setcounter{page}{1}
\section{Introduction}
\label{sec:intro}
This paper deals with the global dynamics of nonlinear parabolic systems of the form
\begin{equation}
\label{init_syst}
u_t=-\nabla V (u) + \ddd  u_{xx} \,,
\end{equation}
where the time variable $t$ and the space variable $x$ are real, the spatial domain is the whole real line, the function $(x,t)\mapsto u(x,t)$ takes its values in $\rr^d$ with $d$ a positive integer, $\ddd$ is a fixed $d\times d$ positive definite symmetric real matrix, and the nonlinearity is the gradient of a scalar \emph{potential} function $V:\rr^d\to\rr$, which is assumed to be regular (of class $\ccc^2$) and coercive at infinity (see hypothesis \cref{hyp_coerc} in \vref{subsubsec:coerc_glob_exist}).

Let us mention at this stage that the choice of introducing the diffusion matrix $\ddd$ in system \cref{init_syst} is just for sake of generality, but that its presence will never change the core of the arguments that will be carried on along the paper. Thus the reader may very well choose to assume that $\ddd$ is actually equal to the identity matrix without missing any significant point. 

The main feature of system~\cref{init_syst} is that it can be recast, at least formally, as the gradient flow of an energy functional. If $(w,w')$ is a pair of vectors of $\rr^d$, let $w\cdot w'$ and $\abs{w} =\sqrt{w\cdot w}$ denote the usual Euclidean scalar product and the usual Euclidean norm, and let
\[
\langle w,w' \rangle_{\ddd} =w\cdot \ddd w'
\quad\text{and}\quad
\abs{w}_{\ddd} =\sqrt{\langle w,w\rangle_{\ddd}}
\]
denote the scalar product associated to $\ddd$ and the corresponding norm, respectively. For every function $v:x\mapsto v(x)$ defined on $\rr$ with values in $\rr^d$, its \emph{energy} (or \emph{Lagrangian} or \emph{action}) with respect to system \cref{init_syst} is defined (at least formally) by
\begin{equation}
\label{form_en}
\eee[v] = \int_\rr \Bigl(\frac{1}{2}\abs{v_x(x)}_{\ddd}^2+V\bigl(v(x)\bigr)\Bigr)\, dx
\,.
\end{equation}
Formally, the differential of this functional reads (skipping border terms in the integration by parts)
\[
\begin{aligned}
d\eee[v]\cdot \delta v &= \int_\rr \bigl( v_x \cdot \ddd (\delta v)_x + \nabla V(v)\cdot \delta v \bigr) \, dx \\
&= \int_\rr \bigl( - \ddd v_{xx} + \nabla V(v) \bigr) \cdot \delta v \, dx
\,.
\end{aligned}
\]
In other words, the (formal) gradient of this functional with respect to the $L^2(\rr,\rr^d)$-scalar product reads 
\[
\nabla\eee[v] = \nabla V(v) - \ddd v_{xx}
\,,
\]
and system \cref{init_syst} can formally be rewritten as
\[
u_t = - \nabla \eee[u(\cdot,t)]
\,.
\]
Accordingly, if $(x,t)\mapsto u(x,t)$ is a solution of this system, then (formally)
\begin{equation}
\label{formal_gradient_structure}
\begin{aligned}
\frac{d}{d t}\eee[u(\cdot,t)] &= d\eee[u(\cdot,t)]\cdot u_t(\cdot,t) \\
&= \bigl\langle \nabla\eee[u(\cdot,t)], u_t(\cdot,t) \bigr\rangle_{L^2(\rr,\rr^d)} \\
&= - \int_\rr \abs{u_t(x,t)}^2 \, dx \le 0
\,.
\end{aligned}
\end{equation}

If system~\cref{init_syst} is considered on a \emph{bounded} spatial domain with boundary conditions that preserve this gradient structure, then the integrals above (on this spatial domain) converge, thus the system really --- and not only formally --- is of gradient type. In this case the dynamics is (at least from a  qualitative point of view) fairly well understood, up to a fine description of the global attractor that is compact and made of the unstable manifolds of stationary solutions \cite{Hale_asymptBehavDissipSyst_1988,Temam_infiniteDimDynSyst_1997}. According to LaSalle's principle, every solution approaches the set of stationary solutions (and even a single stationary solution if the potential is analytic \cite{Simon_asymptoticsEvolEqu_1983}). 

If space is the whole real line and the solutions under consideration are only assumed to be bounded, then the gradient structure above is only formal and allows a much richer phenomenology (the full attractor is far from being fully understood in this case, see the introduction of \cite{GallaySlijepcevic_energyFlowFormallyGradient_2001} and references therein). A salient feature is the occurrence of travelling fronts, that is travelling waves connecting homogeneous equilibria at both ends of space, which are known to play a major role in the asymptotic behaviour of ``many'' solutions. A reasonably wide class of solutions of system \cref{init_syst}, large enough to capture the convergence towards travelling fronts while limiting the complexity of the dynamics, is made of solutions that are close to homogeneous equilibria at both ends of space, at least for large positive times. And among such solutions the simplest case is that of \emph{bistable} solutions, when both equilibria at the ends of space are stable. 

In the late seventies, substantial breakthroughs have been achieved by P. C. Fife and J. B. McLeod about the global behaviour of such \emph{bistable} solutions in the \emph{scalar} case ($d$ equals $1$). Their results comprise global convergence towards a bistable front \cite{FifeMcLeod_approachTravFront_1977}, global convergence towards a ``stacked family of bistable fronts'' \cite{FifeMcLeod_phasePlaneDisc_1981}, and finally, in the case of a bistable potential, a rather complete description of the global asymptotic behaviour of all solutions that are close enough, at infinity in space, to the local (non global) minimum point \cite{Fife_longTimeBistable_1979}.

This paper is part of a series \cite{GallayRisler_globStabBistableTW_2007,Risler_globCVTravFronts_2008,Risler_globalBehaviour_2016} aiming at making a step further in this program, by extending those results to the case of \emph{systems}, and by providing for such systems a complete description of the asymptotic behaviour of all bistable solutions (under generic hypotheses on the potential $V$). Concerning the nature of the arguments involved in the proofs, the main difference with respect to Fife and McLeod's approach is the fact that the maximum principle does not hold any more for systems. It turns out, though, that a purely variational approach is sufficient to recover the results obtained by these authors, thanks to the sole fact that a gradient structure similar to \cref{formal_gradient_structure} exists in every travelling referential (provided that the diffusion matrix $\ddd$ is the identity matrix, \cite{GallayRisler_globStabBistableTW_2007,Risler_globCVTravFronts_2008,Risler_globalBehaviour_2016}). A similar description was also achieved by the same approach for radially symmetric solutions of parabolic gradient systems in higher space dimension \cite{Risler_noInvasionCaseHigherSpace_2020,Risler_globalBehaviourRadiallySymmetric_2017} and for hyperbolic gradient systems in space dimension one \cite{Risler_globalBehaviourHyperbolicGradient_2017}, and might be extended to solutions invading critical points of the potential through pushed fronts at one end or both ends of space \cite{OliverBonafouxRisler_globCVPushedTravFronts_2023}.

The purpose of this paper is to treat the ``relaxation'' part of this program for parabolic systems of the form \cref{init_syst}. To be more explicit, it is to describe the asymptotic behaviour of bistable solutions connecting (local) minimum points in the same level set of the potential and having a (properly localized) energy that remains bounded from below. To this end, only the gradient structure \cref{formal_gradient_structure} in the laboratory frame will be required. The connection of this relaxation part with the full picture is, roughly speaking, as follows: when $\ddd$ is the identity matrix, this lower bound on the localized energy is equivalent to the fact that the neighbourhoods of the two ends of space where the solution remains close to homogeneous equilibria are not ``invaded'' at a positive mean speed, \cite{Risler_globalBehaviour_2016,Risler_noInvasionCaseHigherSpace_2020}; and if by contrast invasion on one or the other side occurs, it must occur via travelling fronts, \cite{Risler_globalBehaviour_2016}.

The literature about relaxation of solutions for systems like~\cref{init_syst} is abundant, and a lot has been done to obtain precise quantitative information about the approach to stationary solutions and the metastable dynamics (``dormant instability'') resulting from the long range interaction between these (spatially localized) stationary solutions. For a more complete list of references together with short historical reviews see for instance \cite{Ei_motionPulses_2002,BethuelOrlandi_slowMotion_2011,BethuelSmets_motionLawFrontsScalarRDEquEqualDepthMultWellPot_2017}. Mostly, these results concern solutions of finite energy in a potential taking only nonnegative values. 

The goal pursued in this paper, in connection with the program mentioned above, is different: the conclusions are limited to the qualitative features of the solutions (they only concern their asymptotic dynamics after an arbitrarily long interval of time for which no quantitative estimate will be given), but the hypotheses are more general: besides the fact that systems and not only scalar equations are considered, the potential may take negative values (assuming that the value taken by the potential at the minimum points approached by the solution at the ends of space is zero), and the solutions under consideration may have an infinite energy. Specific difficulties to overcome are thus to control the behaviour of the solutions at both ends of space, to set up a relaxation scheme despite the fact that the energy may be infinite, and to prove convergence towards the set of stationary solutions that are homoclinic or heteroclinic to homogeneous equilibria without any a priori information about this set.
\section{Assumptions, notation, and statement of the results}
\label{sec:assumpt}
\subsection{Semi-flow}
\subsubsection{Local semi-flow in uniformly local Sobolev space}
\label{subsubsec:funct_space}
Let us denote by $X$ the uniformly local Sobolev space $\Honeul$ (its definition is recalled in \cref{subsec:funct_fram}). This space is the most convenient with respect to the functionals (localized energy and localized $L^2$-norm of the solutions) that are used along the paper. Due to the smoothing properties of system~\cref{init_syst}, the space $X$ might be replaced with the more familiar Banach space $\cccb{1}$ of functions of class $\ccc^1$ that are uniformly bounded together with their first derivative (this more familiar framework is the one chosen in \cite{GallayRisler_globStabBistableTW_2007}). However it is within the functional framework $X=\Honeul$ that the statements are the least sensitive to regularization properties, and thus most appropriate to further generalizations to a wider class of systems, for instance hyperbolic systems (see \cref{subsubsec:damp_hyp}).

System~\cref{init_syst} defines a local semi-flow in $X$ (see for instance D. B. Henry's book \cite{Henry_geomSemilinParab_1981}). 
\subsubsection{Coercivity of the potential and global semi-flow}
\label{subsubsec:coerc_glob_exist}
\emph{Everywhere in the paper}, it will be assumed that the potential function $V:\rr^d\to\rr$ is of class $\ccc^2$ and is strictly coercive at infinity in the following sense: 
\begin{gather}
\tag{$\text{H}_\text{coerc}$}
\lim_{R\to+\infty}\quad  \inf_{\abs{u}\ge R}\ \frac{u\cdot \nabla V(u)}{\abs{u}^2} >0
\label{hyp_coerc}
\end{gather}
(or in other words there exists a positive quantity $\varepsilon$ such that the quantity $u\cdot \nabla V(u)$ is greater than or equal to $\varepsilon\abs{u}^2$ as soon as $\abs{u}$ is large enough). 

According to this hypothesis \cref{hyp_coerc}, the semi-flow of system~\cref{init_syst} is actually global, in other words solutions are defined up to $+\infty$ in time (details are given in \cref{subsec:glob_exist}). Let us denote by $(S_t)_{t\ge0}$ this semi-flow. 

Everywhere in this paper, a \emph{solution of system \cref{init_syst}} refers to a function 
\[
\rr\times[0,+\infty)\to\rr^d\,, \quad (x,t)\mapsto u(x,t)
\,,
\]
such that the function $u_0:x\mapsto u(x,t=0)$ (initial condition) is in $X$ and, for every nonnegative time $t$, the function $u(\cdot,t)$ equals $(S_t u_0)(\cdot)$ (and is therefore also in $X$). 
\subsection{Minimum points, solutions stable at one end of space, and bistable solutions}
\subsubsection{Minimum points}
Everywhere in the paper, the term ``minimum point'' denotes a point where a function --- namely the potential $V$ --- reaches a local \emph{or} global minimum value. Let $\mmm$ denote the set of \emph{nondegenerate} minimum points of $V$:
\[
\mmm=\{m\in\rr^d: \nabla V(m)=0 
\quad\text{and}\quad 
D^2V(m)\text{ is positive definite}\}
\,.
\]
\subsubsection{Solutions stable at one end of space, bistable solutions}
\label{subsubsec:def_not}
\begin{definition}[solutions stable at one end of space, bistable solution]
\label{def:solution_stable_at_one_end_of_space}
Let $(x,t)\mapsto u(x,t)$ be a solution of system~\cref{init_syst}.
\begin{itemize}
\item This solution is said to be \emph{stable at the right end of space} if there exists a point $m_+$ in $\mmm$ such that the quantity
\[
\limsup_{x\to+\infty} \ \abs{u(x,t)-m_+}
\]
goes to $0$ as time goes to $+\infty$. 
More precisely, this solution is said to be \emph{stable close to $m_+$ at the right end of space}. 
\item Similarly, this solution is said to be \emph{stable at the left end of space} if there exists a point $m_-$ in $\mmm$ such that the quantity
\[
\limsup_{x\to-\infty} \ \abs{u(x,t)-m_-}
\]
goes to $0$ as time goes to $+\infty$. 
More precisely, this solution is said to be \emph{stable close to $m_-$ at the left end of space}.
\item Finally, this solution is called a \emph{bistable solution} if it is stable at the left \emph{and} right ends of space. More precisely, if this solution is stable close to $m_-$ at the left end of space and stable close to $m_+$ at the right end of space, then it is called a \emph{bistable solution connecting $m_-$ to $m_+$} (see \cref{fig:bist_sol}). 
\end{itemize}
\begin{figure}[!htbp]
\centering
\includegraphics[width=0.6\textwidth]{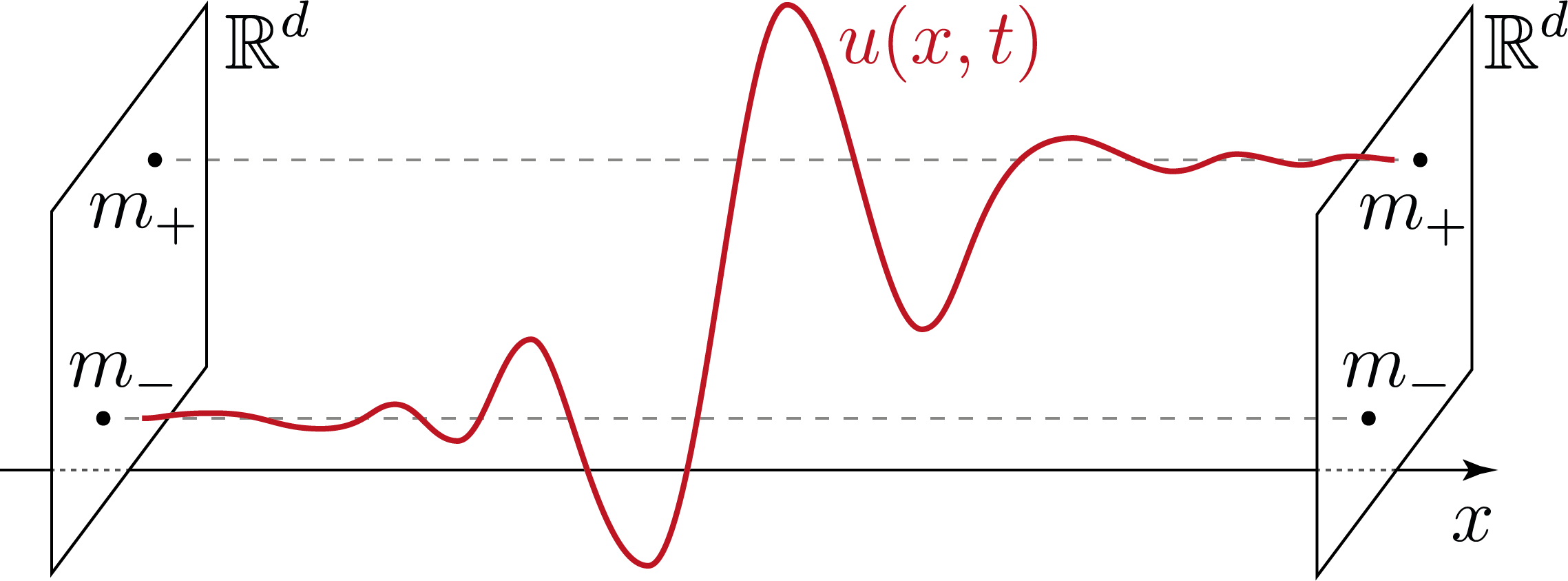}
\caption{A bistable solution connecting $m_-$ to $m_+$.}
\label{fig:bist_sol}
\end{figure}
The same definitions apply to a function (initial condition) $x\mapsto u_0(x)$ in $X$, according to the properties of the solution $(x,t)\mapsto u(x,t)=(S_t u_0)(x)$ of system~\cref{init_syst} corresponding to this initial condition: for instance, $u_0$ is called a \emph{bistable initial condition (connecting $m_-$ to $m_+$)} if $u$ is a bistable solution (connecting $m_-$ to $m_+$).
\end{definition}
\begin{notation}
For every ordered pair $(m_-,m_+)$ of points of $\mmm$, let
\[
\Xbist(m_-,m_+)
\]
denote the subset of $X$ made of bistable initial conditions connecting $m_-$ to $m_+$.
\end{notation}
By definition this set is positively invariant under the semi-flow of system~\cref{init_syst}. 
\subsubsection{Invasion speed of a solution stable at one end of space}
\begin{definition}[invasion speed of a solution stable at one end of space]
\label{def:invasion_speed}
Let $m_-$ and $m_+$ be two points in $\mmm$, and let $(x,t)\mapsto u(x,t)$ be a solution of system~\cref{init_syst}.
\begin{itemize}
\item If this solution is stable close to $m_+$ at the right end of space, then let us call \emph{set of no invasion speeds to the right} the set
\[
\SnoinvPlus[u] = \bigl\{c>0:\sup_{x\ge ct}\abs{u(x,t)-m_+}\to0\text{ when }t\to+\infty\bigr\}
\,,
\]
and let us call \emph{invasion speed to the right}, and let us denote by $\cInvPlus[u]$, the infimum of this set:
\[
\cInvPlus[u] = \inf(\SnoinvPlus[u])
\,.
\]
According to \cref{lem:upper_bound_on_invasion_speed} below, the set $\SnoinvPlus[u]$ is nonempty, so that the invasion speed $\cInvPlus[u]$ is a finite (nonnegative) quantity.
\item Similarly, if this solution is stable close to $m_-$ at the left end of space, then let us call \emph{set of no invasion speeds to the left} the set
\[
\SnoinvMinus[u] = \bigl\{c>0:\sup_{x\le -ct}\abs{u(x,t)-m_-}\to0\text{ when }t\to+\infty\bigr\}
\,,
\]
and let us call \emph{invasion speed to the left}, and let us denote by $\cInvMinus[u]$, the infimum of this set:
\[
\cInvMinus[u] = \inf(\SnoinvMinus[u])
\,.
\]
Again, according to \cref{lem:upper_bound_on_invasion_speed} below, this invasion speed $\cInvMinus[u]$ is a finite (nonnegative) quantity.
\end{itemize}
\end{definition}
\subsection{Preliminary results}
\label{subsec:prel_res}
\subsubsection{Sufficient condition for stability at one end of space, bound on invasion speed, and exponential decrease beyond invasion}
\label{subsubsec:preliminary_results_stability_at_ends_of_space}
As everywhere in the paper, let us assume that $V$ is of class $\ccc^2$ and satisfies the coercivity hypothesis \cref{hyp_coerc}.
\begin{lemma}[sufficient condition for stability at one end of space]
\label{lem:sufficient_condition_stability_right_end_of_space}
For every $m$ in $\mmm$, there exists a positive quantity $\deltaAsymptStab$ (depending on $V$ and $\ddd$ and $m$) such that every solution $(x,t)\mapsto u(x,t)$ of system \cref{init_syst} satisfying
\begin{equation}
\label{hyp_lem_sufficient_condition_stability_right_end_of_space}
\limsup_{\widebar{x}\to+\infty}\int_{\widebar{x}}^{\widebar{x}+1}\Bigl(\bigl(u(x,0)-m\bigr)^2+u_x(x,0)^2\Bigr)\, dx \le \deltaAsymptStab^2
\end{equation}
is stable close to $m$ at the right end of space. 
\end{lemma}
The square exponent of the quantity $\deltaAsymptStab$ at the right-hand side of inequality \cref{hyp_lem_sufficient_condition_stability_right_end_of_space}) is here only to ensure dimensional homogeneity with respect to the function $u_0$ and other parameters along the paper.
\begin{corollary}[to be bistable is an open condition]
\label{cor:bist}
For every ordered pair $(m_-,m_+)$ of points of $\mmm$, the set $\Xbist(m_-,m_+)$ is nonempty and open in $X$. 
\end{corollary}
\begin{lemma}[upper bound on the invasion speed of a solution stable at one end of space]
\label{lem:upper_bound_on_invasion_speed}
For every solution $u:(x,t)\mapsto u(x,t)$ of system \cref{init_syst} which is stable close to a point $m$ of $\mmm$ at the right end of space, the quantity $\cInvPlus[u]$ is bounded from above by a quantity depending on $V$ and $\ddd$ and $m$, but not on the particular solution $u$. 
\end{lemma}
\begin{lemma}[exponential decrease beyond invasion speed]
\label{lem:exponential_decrease_beyond_invasion_speed}
For every solution $u$ $(x,t)\mapsto u(x,t)$ of system \cref{init_syst} which is stable close to a point $m$ of $\mmm$ at the right end of space, and for every positive quantity $c$ larger than $\cInvPlus[u]$, there exists positive quantities $K[u]$ and $\nu$ such that, for every nonnegative time $t$, 
\begin{equation}
\label{exponential_decrease_beyond_invasion_speed}
\sup_{x\in[c t,+\infty)}\abs{u(x,t)-m} \le K[u] \exp(-\nu t)
\,.
\end{equation}
The quantity $\nu$ depends on $V$ and $\ddd$ and $m$ and the difference $c-\cInvPlus[u]$ (only), whereas $K[u]$ depends additionally on $u$. 
\end{lemma}
\subsubsection{Asymptotic energy of a bistable solution: definition and upper semi-continuity}
\label{subsubsec:asympt_en}
\begin{proposition}[asymptotic energy of a bistable solution]
\label{prop:asympt_en}
For every bistable solution $(x,t)\mapsto u(x,t)$ of system \cref{init_syst} connecting two points $m_-$ and $m_+$ of $\mmm$ in the same level set of $V$, there exists a quantity $\eeeAsympt[u]$ in $\{-\infty\}\cup\rr$ such that, for all real quantities $c_-$ and $c_+$ satisfying 
\[
\cInvMinus[u] < c_-
\quad\text{and}\quad
\cInvPlus[u] < c_+
\,,
\]
the following limit holds:
\[
\int_{-c_- t}^{c_+ t}\biggl(\frac{1}{2}\abs{u_x(x,t)}_{\ddd}^2+V\bigl(u(x,t)\bigr)-V(m_\pm)\biggr)\, dx \to \eeeAsympt[u]
\quad\text{as}\quad
t\to +\infty
\,.
\]
\end{proposition}
\begin{definition}[asymptotic energy of a bistable solution connecting two points of $\mmm$ in the same level set of $V$]
If $u:(x,t)\mapsto u(x,t)$ is a bistable solution connecting two points of $\mmm$ in the same level set of $V$, let us call \emph{asymptotic energy of $u$} the quantity $\eeeAsympt[u]$ provided by \cref{prop:asympt_en}. Similarly, if a function $u_0$ in $X$ is a bistable initial condition connecting two points of $\mmm$ in the same level set of $V$, let us call \emph{asymptotic energy of $u_0$} the asymptotic energy of the solution of \cref{init_syst} corresponding to this initial condition, and let us denote by $\eeeAsympt[u_0]$ this asymptotic energy. 
\end{definition}
This leads us to define the \emph{asymptotic energy functional} as follows (for every ordered pair $(m_-,m_+)$ of points of $\mmm$):
\begin{equation}
\label{E_infty}
\eeeAsymptmMinusmPlus : \Xbist(m_-,m_+)  \to\{-\infty\}\sqcup\rr 
\,, \quad
u_0  \mapsto \eeeAsympt[u_0]
\,.
\end{equation}
In the next proposition, statements hold with respect to the topology induced on $\Xbist(m_-,m_+)$ by the $X$-norm and the usual topology on $\{-\infty\}\sqcup\rr$. 
\begin{proposition}[upper semi-continuity of the asymptotic energy]
\label{prop:scs_asympt_en}
For every $m$ in $\mmm$, the asymptotic energy functional $\eeeAsymptmMinusmPlus$ is upper semi-continuous; equivalently, for every real quantity $E$, the set 
\[
\eeeAsymptmMinusmPlus^{-1}\bigl([E,+\infty)\bigr)
=
\left\{u_0\in \Xbist(m_-,m_+): \eeeAsympt[u_0]\ge E\right\}
\] 
is closed. 
\end{proposition}
Under an additional generic assumption on $V$, it will be proved (conclusion \cref{item:thm_main_nonnegative_asympt_energy_approach_bistable_stationary} of \cref{thm:main}) that this asymptotic energy is either nonnegative or equal to $-\infty$. In this case, the subset of $\Xbist(m_-,m_+)$ made of bistable initial conditions having a finite asymptotic energy is thus also closed.

Let us mention another result of the same nature: \cite[Theorem~2]{Risler_globCVTravFronts_2008}, stating that the speed of a travelling front invading a stable equilibrium is \emph{lower} semi-continuous with respect to initial condition. Let us also mention that the preliminary results of this \cref{subsec:prel_res} extend in higher space dimension, \cite{Risler_noInvasionCaseHigherSpace_2020}. 
\subsection{Stationary solutions and standing terraces}
\subsubsection{Hamiltonian system of stationary solutions}
\label{subsubsec:not_ham}
A stationary solution of system~\cref{init_syst} is a function $\xi\mapsto u(\xi)$ from $\rr$ to $\rr^d$ which is a solution of the second order differential system
\begin{equation}
\label{ham_ord_2}
\ddd  u''=\nabla V(u) \,,
\end{equation}
or equivalently a function $\xi\mapsto\bigl(u(\xi),v(\xi)\bigr)$ from $\rr$ to $\rr^{2d}$ of the first order differential system
\begin{equation}
\label{ham_ord_1}
\frac{d}{d\xi}
\begin{pmatrix}
u \\ v
\end{pmatrix}
=
\begin{pmatrix}
v \\ \ddd^{-1}\nabla V(u)
\end{pmatrix}
\,.
\end{equation}
Since the potential $V$ is assumed to be of class $\ccc^2$, every such solution $u$ is of class $\ccc^3$ and its derivative $v$ is of class $\ccc^2$. 

Observe that \cref{ham_ord_2} is a Hamiltonian system. Indeed, if the Hamiltonian $H$ and the nondegenerate skew-symmetric matrix $\Omega$ are defined as
\begin{equation}
\label{def_Hamiltonian_intro}
H:\rr^d\times\rr^d\to\rr,\quad
(u,v)\mapsto \frac{1}{2}\abs{v}_{\ddd}^2 - V(u)
\quad\text{and}\quad
\Omega = \begin{pmatrix}
0 & \ddd^{-1} \\ -\ddd^{-1} & 0
\end{pmatrix}
\,,
\end{equation}
then the system \cref{ham_ord_1} can be rewritten as
\[
\frac{d}{d\xi}
\begin{pmatrix}
u \\ v
\end{pmatrix}
=
\begin{pmatrix}
0 & \ddd^{-1} \\ -\ddd^{-1} & 0
\end{pmatrix}
\begin{pmatrix}
-\nabla V(u) \\ \ddd v
\end{pmatrix}
=\Omega \cdot \nabla H (u,v)
\,.
\]
and the Hamiltonian is a conserved quantity for this system: for every solution $\xi\mapsto u(\xi)$ of \cref{ham_ord_2}, 
\begin{equation}
\label{Ham_conserved_quantity}
\frac{d}{d\xi}H\bigl(u(\xi),u'(\xi)\bigr) = 0
\,.
\end{equation}
The (formal) energy defined in~\cref{form_en} is the integral over space of the Lagrangian
\begin{equation}
\label{def_Lagrangian_intro}
L:\rr^d\times\rr^d\to\rr,\quad
(u,v)\mapsto \frac{1}{2}\abs{v}_{\ddd}^2 + V(u)
\,.
\end{equation}
\subsubsection{Bistable stationary solutions}
\begin{notation}
Let $\sss$ denote the set of \emph{stationary solutions} of system~\cref{init_syst}, that is of global solutions $\xi\mapsto u(\xi)$ of system~\cref{ham_ord_2}.
If $(m_-,m_+)$ is an ordered pair of points of $\mmm$ (they might be equal or different), let 
\begin{equation}
\label{def_Phi_zero_m_minus_m_plus}
\Phi_0(m_-,m_+)
\end{equation}
denote the set of \emph{bistable stationary solutions connecting $m_-$ to $m_+$}, that is the set of functions $\xi\mapsto \phi(\xi)$ in $\sss$ satisfying
\[
\phi(\xi)\xrightarrow[\xi\to -\infty]{} m_-
\quad\text{and}\quad
\phi(\xi)\xrightarrow[\xi\to +\infty]{}  m_+
\]
(including the homogeneous solution $\phi\equiv m_\pm$ if $m_-=m_+$). This notation refers to the fact that these solutions might be viewed as ``standing fronts'' (fronts travelling at speed zero, at least if $m_-$ differs from $m_+$ --- if $m_-$ and $m_+$ are equal the denomination ``standing pulse'' suits better); the index ``$0$'' in the notation \cref{def_Phi_zero_m_minus_m_plus} refers to the vanishing speed of these solutions, by contrast with the fronts travelling at nonzero speed considered in the companion papers \cite{Risler_globalBehaviour_2016,Risler_globalBehaviourHyperbolicGradient_2017,Risler_globalBehaviourRadiallySymmetric_2017}. 
Observe that the set $\Phi_0(m_-,m_+)$ is exactly made of stationary solutions of system~\cref{init_syst} that are altogether bistable solutions connecting $m_-$ to $m_+$, in other words,
\[
\Phi_0(m_-,m_+) = \sss \,\cap\, \Xbist(m_-,m_+)
\,.
\]
Since according to equality \cref{Ham_conserved_quantity} the Hamiltonian \cref{def_Hamiltonian_intro} is constant along a stationary solution, only if $V(m_-)=V(m_+)$ can the set $\Phi_0(m_-,m_+)$ be nonempty. For every real quantity $\valueOfV$, let $\mmm_{\valueOfV}$ denote the set of nondegenerate local minimum points in the level set $V^{-1}(\{\valueOfV\})$:
\[
\mmm_{\valueOfV} = \mmm\cap V^{-1}(\{\valueOfV\}) = \{m\in\mmm : V(m)=\valueOfV\}
\,,
\]
and let $\Phi_0(\valueOfV)$ denote the union, for all ordered pairs $(m_-,m_+)$ of points of $\mmm_{\valueOfV}$, of the sets $\Phi_0(m_-,m_+)$:
\begin{equation}
\label{definition_PhiZero_of_h}
\Phi_0(\valueOfV)=\bigsqcup_{(m_-,m_+)\in \mmm_{\valueOfV}^2}\Phi_0(m_-,m_+) 
\,.
\end{equation}
For every function $\xi\mapsto u(\xi)$ in $\sss$, let
\[
I(u)=\bigcup_{\xi\in\rr} \ \bigl\{\bigl(u(\xi),u'(\xi)\bigr)\bigr\}
\]
denote the ``image'' of $u$ (its trajectory in the phase space $\rr^{2d}$ of the Hamiltonian system~\cref{ham_ord_1}), and let $I(\bigl(\Phi_0(\valueOfV)\bigr)$ denote the union of all images of bistable stationary solutions connecting points of $\mmm_{\valueOfV}$:
\[
I(\bigl(\Phi_0(\valueOfV)\bigr) = \bigcup_{\phi\in\Phi_0(\valueOfV)} \  \ I(\phi)
\,.
\]
For every $m$ in $\mmm$, let $\Ws(m,0)$ denote the stable manifold of the equilibrium $(m,0)$ for the Hamiltonian system~\cref{ham_ord_1}, and let $\Wu(m,0)$ denote its unstable manifold. It follows from this notation that for every real quantity $\valueOfV$, 
\[
I(\bigl(\Phi_0(\valueOfV)\bigr) = \left(\bigcup_{m\in\mmm_{\valueOfV}}\{(m,0)\}\right) \cup \left(\bigcup_{(m_-,m_+)\in\mmm_{\valueOfV}^2} \Wu(m_-,0)\cap \Ws(m_+,0)\right) 
\,.
\]
\end{notation}
The shapes of some examples of this set $I(\bigl(\Phi_0(\valueOfV)\bigr)$ are shown on \cref{fig:shape_pot}, for various familiar examples of potential $V$ (and for one or several values of the quantity $\valueOfV$), in the scalar case $d$ equals $1$. 
\begin{figure}[!htbp]
\centering
\includegraphics[width=\textwidth]{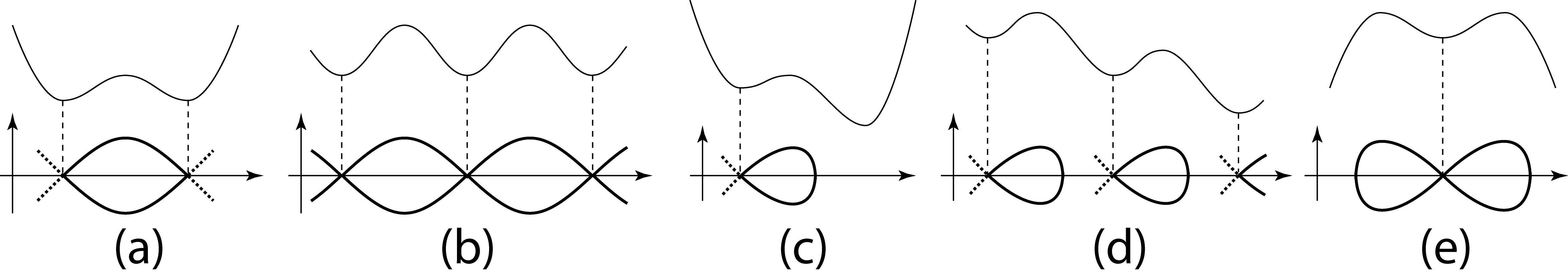}
\caption{Shapes of familiar examples of potentials and of the corresponding phase portraits of system~\cref{ham_ord_1} governing stationary solutions of system~\cref{init_syst}: (a) the Allen--Cahn equation, (b) the over-damped sine--Gordon equation, (c) the Nagumo equation, (d) the over-damped sine--Gordon equation with constant forcing, and (e) the ``subcritical'' Allen--Cahn equation. The corresponding equations are briefly discussed in \vref{sec:examples}.}
\label{fig:shape_pot}
\end{figure}
This set $I(\bigl(\Phi_0(\valueOfV)\bigr)$ will be called upon in conclusion \cref{item:thm_main_nonnegative_asympt_energy_approach_bistable_stationary} of \cref{thm:main}. 
\subsubsection{Standing terraces of bistable stationary solutions}
\label{subsubsec:def_stand_terrace}
To state conclusion \cref{item:thm_main_approach_standing_terrace_and_value_asymptotic_energy} of \cref{thm:main}, the next definitions are required. Some comments on the terminology and related references are given at the end of this \namecref{subsubsec:def_stand_terrace}.
\begin{figure}[!htbp]
\centering
\includegraphics[width=\textwidth]{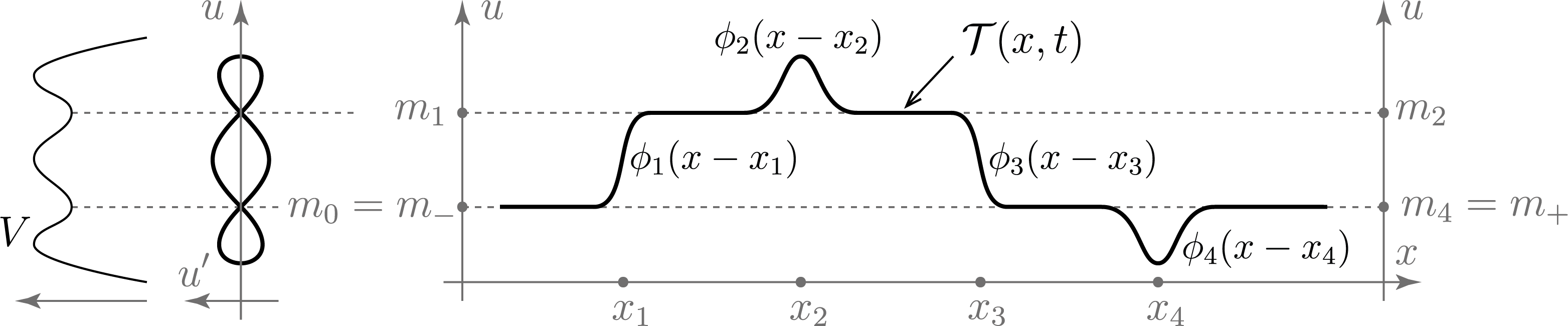}
\caption{Standing terrace (with four items, $q=4$).}
\label{fig:standing_terrace}
\end{figure}
\begin{definition}[standing terrace of bistable stationary solutions, \cref{fig:standing_terrace}]
\label{def:standing_terrace}
Let $\valueOfV$ be a real quantity and let $m_-$ and $m_+$ be two points of $\mmm_{\valueOfV}$. A function
\[
\ttt : \rr\times[0,+\infty)\to\rr^d,\quad (x,t)\mapsto \ttt(x,t)
\]
is called a \emph{standing terrace of bistable stationary solutions, connecting $m_-$ to $m_+$,} if there exists a nonnegative integer $q$ such that:
\begin{enumerate}
\item if $q$ equals $0$, then $m_-=m_+$ and, for every real quantity $x$ and every nonnegative time $t$, 
\[
\ttt(x,t)=m_-=m_+
\,;
\]
\item if $q=1$, then there exist:
\begin{itemize}
\item a function $\phi_1$ in $\Phi_0(m_-,m_+)$ (a bistable stationary solution connecting $m_-$ to $m_+$),
\item and a $\ccc^1$-function $t\mapsto x_1(t)$ defined on $[0,+\infty)$ and satisfying $x_1'(t)\to0$ as time goes to $+\infty$,
\end{itemize}
such that, for every real quantity $x$ and every nonnegative time $t$, 
\[
\ttt(x,t) = \phi_1\bigl( x-x_1(t)\bigr)
\,;
\]
\label{item:def_standing_terrace_q_equals_one}
\item if $q$ is not smaller than $2$, then there exists $q-1$ points $m_1$, …, $m_{q-1}$ of $\mmm_{\valueOfV}$ (not necessarily distinct), and if $m_-$ is denoted by $m_0$ and $m_+$ by $m_q$, then for each integer $i$ in $\{1,\dots,q\}$, there exists:
\begin{itemize}
\item a function $\phi_i$ in $\Phi_0(m_{i-1},m_i)$ (a bistable stationary solution connecting $m_{i-1}$ to $m_i$),
\item and a $\ccc^1$-function $t\mapsto x_i(t)$ defined on $[0,+\infty)$ and satisfying $x_i'(t)\to0$ as time goes to $+\infty$,
\end{itemize}
such that, for every integer $i$ in $\{1,\dots,q-1\}$,
\[
x_{i+1}(t)-x_i(t)\to +\infty 
\quad\text{as}\quad
t\to +\infty
\,,
\]
and such that, for every real quantity $x$ and every nonnegative time $t$, 
\[
\ttt(x,t) = m_- + \sum_{i=1}^q \Bigl[\phi_i\bigl(x-x_i(t)\bigr)-m_{i-1}\Bigr]
\,.
\]
\label{item:def_standing_terrace_q_larger_than_one}
\end{enumerate}
\end{definition}
\begin{remark}
Item \cref{item:def_standing_terrace_q_equals_one} may have been omitted in this definition, since it boils down to item \cref{item:def_standing_terrace_q_larger_than_one} with $q$ equals $1$.
\end{remark}
The terminology ``propagating terrace'' was introduced by A. Ducrot, T. Giletti, and H. Matano in \cite{DucrotGiletti_existenceConvergencePropagatingTerrace_2014} (and subsequently used by several other authors \cite{Polacik_propagatingTerracesAsymptOneDimSym_2017,Polacik_propTerracesProofGibbonsConj_2016,GilettiRossi_pulsatingSolMultBistMultiStab_2019,MatanoPolacik_dynNonnegSolOneDimRDII_2020,Polacik_propagatingTerracesDynFrontLikeSolRDEquationsR_2020,GilettiMatano_existenceUniquenessPropTerr_2020,PauthierRademacherU_WeakStrongInteractKinks_2021}) to denote a stacked family (a layer) of travelling fronts in a (scalar) reaction-diffusion equation. This led the author to introduce the analogous ``standing terrace'' terminology above, because this terminology is convenient to denote an object otherwise requiring a quite long description, and because it provides a convenient homogeneity in the formulation of the results of \cite{Risler_globalBehaviour_2016} describing the asymptotic behaviour of all bistable solutions of systems like \cref{init_syst}, since this behaviour involves altogether two ``propagating terraces'' (one to the left and one to the right) and a ``standing terrace'' in between. This terminology is also used in the companion papers \cite{Risler_globalBehaviourHyperbolicGradient_2017,Risler_globalBehaviourRadiallySymmetric_2017}.

The author hopes that these advantages balance some drawbacks of this terminological choice. Like the fact that the word ``terrace'' is probably more relevant in the scalar case $d$ equals $1$ (see the pictures in \cite{DucrotGiletti_existenceConvergencePropagatingTerrace_2014,Polacik_propagatingTerracesDynFrontLikeSolRDEquationsR_2020}) than in the more general case of systems considered here. Or the fact that the definitions above and in \cite{Risler_globalBehaviour_2016} are different from the original definition of \cite{DucrotGiletti_existenceConvergencePropagatingTerrace_2014} in that they involve not only the profiles of particular (standing or travelling) solutions, but also their positions (denoted above by $x_i(t)$). 

To finish, observe that in the present context:
\begin{itemize}
	\item terraces are only made of bistable solutions, by contrast with the propagating terraces introduced and used by the authors cited above;
	\item standing terraces are approached by solutions but are (in general) not solutions themselves;
	\item a standing terrace may be nothing but a single stable homogeneous equilibrium (if $q$ equals $0$).
\end{itemize}
\subsubsection{Energy of a bistable stationary solution and of a standing terrace}
\label{subsubsec:def_energy_stat_sol_stand_terrace}
\begin{definition}[energy of a bistable stationary solution]
Let $\xi\mapsto \phi(\xi)$ be a bistable stationary solution connecting two points $m_-$ and $m_+$ of $\mmm$ (which must therefore belong to the same level set of $V$). The quantity 
\[
\eee[\phi] = \int_{\rr}\Bigl(\frac{1}{2}\abs{\phi'(\xi)}_{\ddd} ^2+V\bigl(\phi(\xi)\bigr)-V(m_\pm)\Bigr)\, d\xi
\]
is called the \emph{energy of the (bistable) stationary solution $\phi$}. Observe that this integral converges: since $m_-$ and $m_+$ are in $\mmm$ they are nondegenerate local minimum points, thus $\phi(\xi)$ approaches its limits at both ends of space at an exponential rate. 
\end{definition}
%
%
\begin{definition}[energy of a standing terrace]
\label{def:energy_standing_terrace}
Let $\valueOfV$ denote a real quantity and let $\ttt$ denote a standing terrace of bistable stationary solutions. With the notation of the two definitions above, the quantity $\eee[\ttt]$ defined as
\begin{enumerate}
\item if $q$ equals $0$, then $\eee[\ttt]=0$,
\item if $q$ equals $1$, then $\eee[\ttt]=\eee[\phi_1]$,
\item if $q$ is not smaller than $2$, then $\eee[\ttt]=\sum_{i=1}^q\eee[\phi_i]$,
\end{enumerate}
is called the \emph{energy of the standing terrace $\ttt$}. 
\end{definition}
\subsection{Generic hypotheses on the potential}
\label{subsec:generic_assupmt_pot}
The goal of this \namecref{subsec:generic_assupmt_pot} is to state two generic hypotheses on $V$ (\cref{subsubsec:genericl_hyp_V}), which will be called upon in the two versions of the main result of this paper. The additional material (notation and definitions) required to state these hypotheses is provided in the next two \namecref{subsubsec:Escape_dist}. 
\subsubsection{Escape distance of a minimum point}
\label{subsubsec:Escape_dist}
\begin{notation}
For every $u$ in $\rr^d$, let $\sigma\bigl(D^2V(u)\bigr)$ denote the spectrum (the set of eigenvalues) of the Hessian matrix of $V$ at $u$, and let $\eigVmin(u)$ denote the minimum of this spectrum:
\begin{equation}
\label{def_eigVmin_of_u}
\eigVmin(u) = \min \Bigl(\sigma\bigl(D^2V(u)\bigr)\Bigr)
\,.
\end{equation}
\end{notation}
\begin{definition}[Escape distance of a nondegenerate local minimum point]
For every $m$ in $\mmm$, let us call \emph{Escape distance of $m$}, and let us denote by $\dEsc(m)$, the supremum of the set
\begin{equation}
\label{set_for_definition_Escape_distance}
\Bigl\{\delta \in[0,1]: \text{ for all } u \text{ in } \rr^d \text{ satisfying } \abs{u-m}_{\ddd}\le \delta, \ \eigVmin(u) \ge\frac{1}{2} \eigVmin(m) \Bigr\}
\,.
\end{equation}
\end{definition}
Since the quantity $\eigVmin(u)$ varies continuously with $u$, this Escape distance $\dEsc(m)$ is positive (thus in $(0,1]$). In addition, for all $u$ in $\rr^d$ such that $\abs{u-m}_{\ddd}$ is not larger than $\dEsc(m)$, the following inequality holds:
\begin{equation}
\label{property_dEsc}
\eigVmin(u) \ge\frac{1}{2} \eigVmin(m)
\,.
\end{equation}
This ``Escape'' distance will be used in two different ways. 
\begin{enumerate}
\item To ``track'' the position in space where a solution ``escapes'' a neighbourhood of $m$ (this position is called ``leading edge'' by Muratov in a framework including monostable invasion \cite{Muratov_globVarStructPropagation_2004,MuratovNovaga_globExpConvTW_2012,MuratovZhong_thresholdSymSol_2013}). The reason for the upper-case letter ``E'' in ``Esc'' is to make a difference with another escape distance ``$\desc(m)$'' that will be required later (see \cref{subsubsec:cont_fire}).
\item To normalize the bistable stationary solutions with respect to translation invariance (in the next \namecref{subsubsec:norm_bist}).
\end{enumerate}
\begin{remark}
There is nothing profound behind the choice of using the $\abs{\cdot}_{\ddd}$ rather than the usual Euclidean norm of $\rr^d$ in the definition \cref{set_for_definition_Escape_distance} of the set defining the Escape distance. The sole reason is that \vref{lem:ham_equ_hyp} fits better with this choice. 
\end{remark}
\begin{remark}
If the set $\mmm$ was assumed to be finite (this would hold for instance if all critical points of $V$ were assumed to be nondegenerate, which is generically true), then a ``uniform'' positive quantity $\dEsc$ could be picked, small enough so that inequality \cref{property_dEsc} holds with $\dEsc$ instead of $\dEsc(m)$ for every point $m$ in $\mmm$. 
\end{remark}
\subsubsection{Normalization of bistable stationary solutions with respect to translation invariance}
\label{subsubsec:norm_bist}
According to assertion \cref{item:escape_spatial_asymptotics_sw} of \vref{lem:ham_equ_hyp}, for every ordered pair $(m_-,m_+)$ of points of $\mmm$ and for every \emph{nonconstant} stationary solution $\xi\mapsto \phi(\xi)$ connecting $m_-$ to $m_+$, 
\[
\sup_{\xi\in\rr}\abs{\phi(\xi)-m_-}_{\ddd}>\dEsc(m_-)
\quad\text{and}\quad
\sup_{\xi\in\rr}\abs{\phi(\xi)-m_+}_{\ddd}>\dEsc(m_+)
\,,
\]
see \cref{fig:esc_distance}. 
\begin{figure}[!htbp]
\centering
\includegraphics[width=0.6\textwidth]{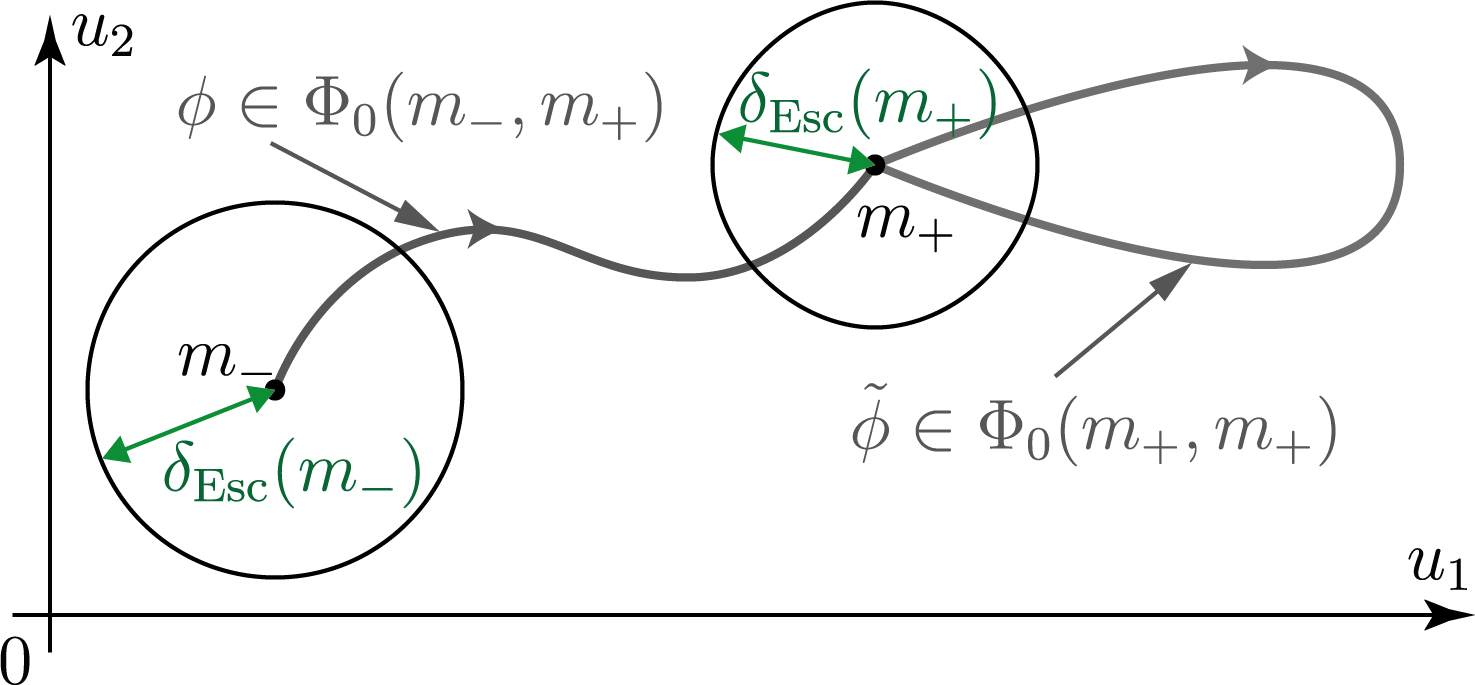}
\caption{Nonconstant bistable stationary solutions escape at least at a $\abs{\cdot}_{\ddd}$-distance $\dEsc(m_\pm)$ from their limits at $\pm\infty$.}
\label{fig:esc_distance}
\end{figure}
As a consequence, a unique ``normalized'' translate of this solution can be picked up by demanding that, say, the translate be exactly at a $\abs{\cdot}_{\ddd}$-distance $\dEsc(m_+)$ of his right-end limit $m_+$ at $\xi=0$, and closer for every positive $\xi$ (see \cref{fig:norm_stat}). Here is a more formal definition. 
\begin{figure}[!htbp]
\centering
\includegraphics[width=0.75\textwidth]{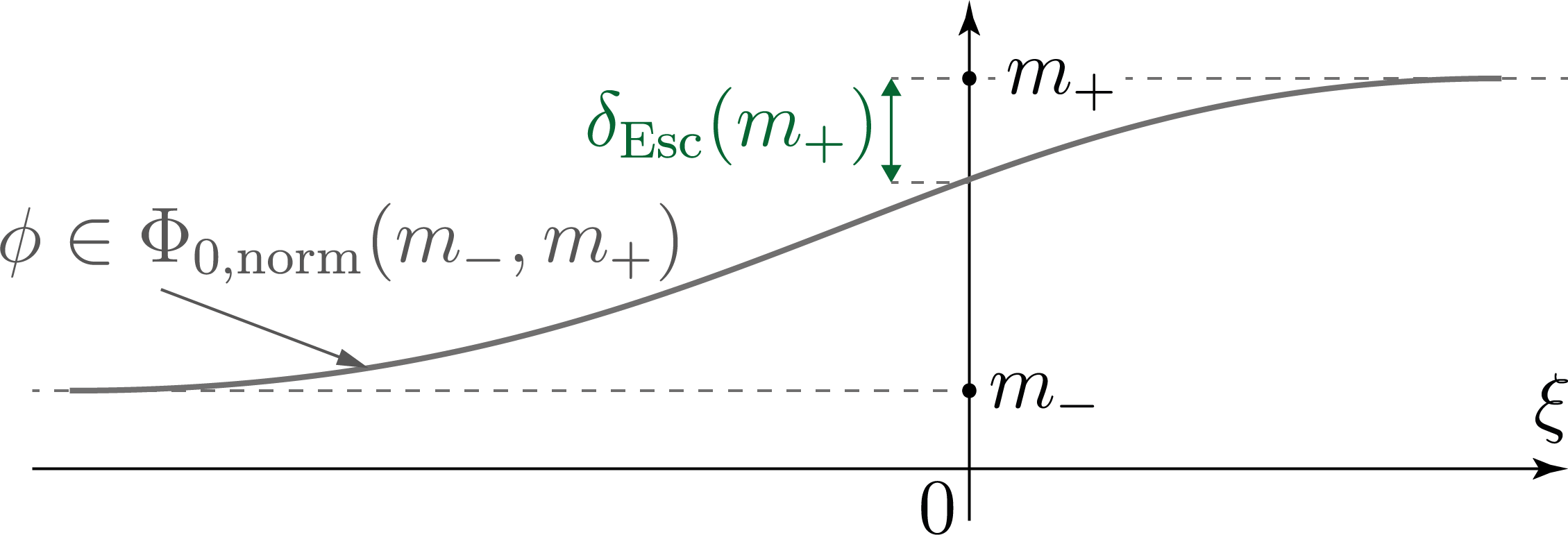}
\caption{Normalized bistable stationary solution.}
\label{fig:norm_stat}
\end{figure}
\begin{definition}[normalized bistable stationary solution]
For $(m_-, m_+)$ in $\mmm^2$, a bistable stationary solution $\phi$ connecting $m_-$ to $m_+$ is said to be \emph{normalized} if 
\begin{equation}
\label{normalization_condition}
\abs{\phi(0)-m_+}_{\ddd} =\dEsc(m_+) 
\quad\text{and}\quad
\abs{\phi(\xi)-m_+}_{\ddd} <\dEsc(m_+)\quad\text{for all}\quad \xi>0
\,.
\end{equation}
Let us denote by $\PhiZeroNorm(m_-,m_+)$ the subset of $\Phi_0(m_-,m_+)$ made of function $\phi$ satisfying the normalization condition \cref{normalization_condition}.
\end{definition}
\subsubsection{Statement of the generic hypotheses}
\label{subsubsec:genericl_hyp_V}
Let $\valueOfV$ denote a real quantity. The next two hypotheses will be called upon in the statement of \cref{thm:main}. 
\begin{description}
\item[\hypOnlyMinLabel{\valueOfV}]\hypertarget{hypOnlyMin} All critical points of $V$ in the level set $V^{-1}(\{\valueOfV\})$ are nondegenerate minimum points. In other words, for every $u$ in $\rr^d$, 
\[
V(u)=\valueOfV
\quad\text{and}\quad
\nabla V(u)=0
\implies
D^2V(u)
\text{ is positive definite.}
\]
\item[\hypDiscStatLabel{\valueOfV}]\hypertarget{hypDiscStat} For every $m_-$ in $\mmm_{\valueOfV}$, the set
\[
\bigsqcup_{m_+\in\mmm_{\valueOfV}}\bigl\{ \bigl( \phi(0), \phi'(0) \bigr) : \phi \in \PhiZeroNorm(m_-,m_+) \bigr\}
\]
is totally disconnected in $\rr^{2d}$ (that is, its connected components are singletons). Equivalently, the set 
\begin{equation}
\label{def_phiZeronorm_of_mmm_h}
\PhiZeroNorm(\valueOfV) = \bigcup_{(m_-,m_+)\in \mmm_{\valueOfV}^2}\PhiZeroNorm(m_-,m_+)
\end{equation}
is totally disconnected for the topology of compact convergence (uniform convergence on compact subsets of $\rr$). 
\end{description}
A formal proof of the genericity (with respect to the potential $V$) of these two hypotheses is provided in \cite{JolyRisler_genericTransversalityTravStandFrontsPulses_2023}.
\subsection{Main result}
\label{subsec:main_res}
Let us recall the definition of the distance between a point $z_0$ and a subset $\Sigma$ of $\rr^{2d}$:
\[
\dist(z_0,\Sigma) = \inf_{z\in \Sigma}\abs{z-z_0}
\]
where $\abs{\cdot}$ denotes (say) the usual euclidean norm on $\rr^{2d}$. Here is the main result of this paper. 
\begin{theorem}
\label{thm:main}
Let $V$ denote a function in $\ccc^2(\rr^d,\rr)$ satisfying the coercivity hypothesis \cref{hyp_coerc}. Then, for every real quantity $\valueOfV$ and for every bistable solution $(x,t)\mapsto u(x,t)$ of system~\cref{init_syst} connecting two (possibly equal) points of $\mmm_{\valueOfV}$, if the asymptotic energy of this solution is not equal to $-\infty$, then the following conclusions hold.
\begin{enumerate}
\item The time derivative $u_t(x,t)$ goes to $0$ as $t\to+\infty$, uniformly with respect to $x$ in $\rr$.
\label{item:thm_main_time_derivative_goes_to_zero}
\item Both invasion speeds of the solution (to the left and to the right) vanish.
\label{item:thm_main_invasion_speeds_vanish}
\item If hypothesis \textup{(\hyperlink{hypOnlyMin}{\hypOnlyMinRef{\valueOfV}})} holds, then the asymptotic energy of the solution is nonnegative and the quantity 
\[
\sup_{x\in\rr} \ \dist\biggl(\Bigl(u(x,t),u_x(x,t)\Bigr) \,, \,I\bigl(\PhiZero(\valueOfV)\bigr)\biggr)
\]
goes to $0$ as time goes to $+\infty$.
\label{item:thm_main_nonnegative_asympt_energy_approach_bistable_stationary}
\item If both hypotheses \textup{(\hyperlink{hypOnlyMin}{\hypOnlyMinRef{\valueOfV}})} and \textup{(\hyperlink{hypDiscStat}{\hypDiscStatRef{\valueOfV}})} hold, then: 
\label{item:thm_main_approach_standing_terrace_and_value_asymptotic_energy}
\begin{enumerate} 
\item the solution approaches (uniformly in space, as time goes to $+\infty$) a standing terrace of bistable stationary solutions, 
\label{item:thm_main_approach_standing_terrace}
\item and the asymptotic energy of the solution equals the energy of this standing terrace. 
\label{item:thm_main_value_asymptotic_energy}
\end{enumerate}
\end{enumerate}
\end{theorem}
With symbols, conclusions \cref{item:thm_main_approach_standing_terrace,item:thm_main_value_asymptotic_energy} of this theorem can be stated as follows: for every ordered pair $(m_-,m_+)$ of points of $\mmm_{\valueOfV}$, and for every $u_0$ in $\Xbist(m_-,m_+)$, if 
\[
\eeeAsympt[u_0]>-\infty
\text{ and hypotheses \textup{(\hyperlink{hypOnlyMin}{\hypOnlyMinRef{\valueOfV}})} and \textup{(\hyperlink{hypDiscStat}{\hypDiscStatRef{\valueOfV}})} hold,} 
\]
then there exists a standing terrace $\ttt$ of bistable stationary solutions, connecting $m_-$ to $m_+$, such that, if $(x,t)\mapsto u(x,t)$ denotes the solution of system \cref{init_syst} corresponding to the initial condition $u_0$, then
\[
\sup_{x\in\rr} \abs{u(x,t) - \ttt(x,t)} \to 0
\quad\text{as}\quad
t\to+\infty
\quad\text{and}\quad
\eeeAsympt[u_0]=\eee[\ttt]
\,.
\]
If conversely the asymptotic energy of the solution equals $-\infty$, then the corresponding solution certainly takes values where the potential is negative as time increases, but no precise information on its behaviour will be given in this paper. In the companion paper \cite{Risler_globalBehaviour_2016} (following \cite{Risler_globCVTravFronts_2008}), it is proved (only if the diffusion matrix $\ddd$ is equal to identity) that in this case the solution displays travelling fronts invading the stable equilibria at both ends of space. Results of the same kind have been obtained (in a different setting limited to the scalar case $d$ equals $1$) by Muratov and X. Zhong in \cite{MuratovZhong_thresholdSymSol_2013}.

A series of standard results can be recovered as direct consequences of \cref{thm:main,prop:scs_asympt_en}. Those results deal with:
\begin{itemize}
\item existence of homoclinic or heteroclinic orbits of the Hamiltonian systems governing stationary solutions;
\item the basin of attraction of the homogeneous stationary solution induced by a point of $\mmm$ (or the border of this basin of attraction).
\end{itemize}
To avoid disrupting the attention of the reader from the main result, the statements of these auxiliary results (and their proofs) are postponed until \cref{sec:exist_res_bas_att}. 
\subsection{Additional remarks and comments}
\subsubsection{Examples}
Elementary examples corresponding to the potentials illustrated on \cref{fig:shape_pot} (in the scalar case $d$ equals $1$) are discussed in \cref{sec:examples}.
\subsubsection{Convergence for a stronger topology}
Due to the smoothing properties of system~\cref{init_syst} (see \cref{subsec:glob_exist}), convergence towards the standing terrace in conclusion \cref{item:thm_main_approach_standing_terrace} of \cref{thm:main} holds with respect to the $\cccb{2}$-norm. 
\subsubsection{Limit sets of profiles and quasi-convergence}
For every solution $(x,t)\mapsto u(x,t)$ of system \cref{init_syst}, we may consider, following the notation of \cite{PauthierPolacik_largeTimeBehavSolParabEquR_2018}, its $\Omega$-limit set $\Omega(u)$ defined as
\[
\begin{aligned}
\Omega(u) = \bigl\{&\varphi\in\cccb{0}:\text{$u(x_n+\cdot,t_n)\to\varphi$, in $\Linftyloc$, as $n\to+\infty$, for some}\\
&\text{real sequences $(x_n)_{n\in\nn}$ and $(t_n)_{n\in\nn}$ with $t_n\to+\infty$ as $n\to+\infty$} \bigr\}
\,.
\end{aligned}
\]
Conclusion \cref{item:thm_main_time_derivative_goes_to_zero} of \cref{thm:main} ensures that $\Omega(u)$ consists entirely of steady states (stationary solutions of system \cref{init_syst} (in particular the solution is \emph{quasi-convergent} \cite{PauthierPolacik_largeTimeBehavSolParabEquR_2018}); and conclusion \cref{item:thm_main_nonnegative_asympt_energy_approach_bistable_stationary} ensures that $\Omega(u)$ is, more precisely, included in the set $\Phi_0(\valueOfV)$ of profiles of stationary solutions that are homoclinic or heteroclinic to points of $\mmm_{\valueOfV}$. More refined results of the same flavour (obtained by completely different methods, in the scalar case $d$ equals $1$ and under weaker assumptions otherwise) have been recently obtained by A. Pauthier and P. Poláčik \cite{PauthierPolacik_largeTimeBehavSolParabEquR_2018,PauthierPolacik_largeTimeBehavIIEqualLimitsInfinity_2021,PauthierPolacik_largeTimeBehavIIIUnstablelLimitsInfinity_2021} (see also \cite{MatanoPolacik_dynNonnegSolOneDimRDI_2016,MatanoPolacik_dynNonnegSolOneDimRDII_2020}). 
\subsubsection{Long range interaction between bistable stationary solutions}
\label{subsubsec:long_range}
The next step following \cref{thm:main} would be to study more precisely the long-range interactions between the bistable stationary solutions involved in the standing terrace describing the asymptotic behaviour of the solution (in the case $q\ge2$); and to provide explicit expressions for the asymptotics (at first order) of these interactions. This long-lasting question is treated in details for gradient systems in the recent paper \cite{BethuelSmets_motionLawFrontsScalarRDEquEqualDepthMultWellPot_2017} of Béthuel and Smets, however in a more restrictive framework (scalar equation and nonnegative potential, see also the conjecture p. 59 of \cite{BethuelOrlandi_slowMotion_2011}). It has also been addressed by finite-dimensional reduction methods for more general reaction-diffusion equation and systems, see \cite{Ei_motionPulses_2002} and especially the monograph \cite{MielkeZelik_multiPulseEvolutionAndSpaceTimeChaosDissipativeSystems_2009} by A. Mielke and S. Zelik.

Following conclusion \cref{item:thm_main_approach_standing_terrace} of \cref{thm:main}, since the stationary solutions involved in the standing terrace must a priori go (slowly) away from one another, the first order interaction term between two successive stationary solutions $u_{i}$ and $u_{i+1}$, $i\in\{1,\dots, q-1\}$ should be repulsive, and this should give some restrictions on the families $(u_{1},\dots, u_{q})$ that can actually occur in such a terrace. Elementary examples are discussed in \cref{sec:examples}, but general statements and rigorous proofs are beyond the scope of this paper. 
\subsubsection{Extension to the damped hyperbolic case}
\label{subsubsec:damp_hyp}
It is likely that results similar to those of this paper hold for the damped hyperbolic system 
\begin{equation}
\label{hyp_syst}
\alpha u_{tt}+u_t=-\nabla V(u)+\ddd u_{xx}
\,,
\end{equation}
obtained by adding an inertial term $\alpha u_{tt}$ (where $\alpha$ is a positive non necessarily small quantity) to the parabolic system~\cref{init_syst} considered here. Actually, for a dissipation matrix $\ddd$ equal to identity, the global behaviour of all bistable solutions of system \cref{hyp_syst} is described in \cite{Risler_globalBehaviourHyperbolicGradient_2017}, and the ``relaxation'' part of this description, which presents strong similarities with the results of the present paper, calls upon the same methods. 
\subsubsection{Some unsolved questions}
\label{subsubsec:uns_quest}
Here are some additional (and, to the knowledge of the author, open) questions that raise naturally from the statements above. 
\begin{enumerate}
\item Do conclusions \cref{item:thm_main_nonnegative_asympt_energy_approach_bistable_stationary,item:thm_main_approach_standing_terrace} of \cref{thm:main} still hold without hypothesis \textup{(\hyperlink{hypOnlyMin}{\hypOnlyMinRef{\valueOfV}})} (stating that all critical points in the level set $V^{-1}(\{0\})$ of the potential are nondegenerate local minima)? (this question is twofold: hypothesis \textup{(\hyperlink{hypOnlyMin}{\hypOnlyMinRef{\valueOfV}})} may be relaxed assuming that those critical points are still minimum points but possibly degenerate ones, or dropping any additional hypothesis about these critical points). 
\item Does conclusion \cref{item:thm_main_approach_standing_terrace_and_value_asymptotic_energy} of \cref{thm:main} still hold without hypothesis \textup{(\hyperlink{hypDiscStat}{\hypDiscStatRef{\valueOfV}})} (stating that the set of normalized bistable stationary solutions of zero Hamiltonian is totally disconnected in $X$)? For instance, does it hold for the $O(2)$-symmetric ``real Ginzburg--Landau'' potentials (see \cref{fig:GL_potentials}):
\begin{figure}[!htbp]
\centering
\includegraphics[width=0.7\textwidth]{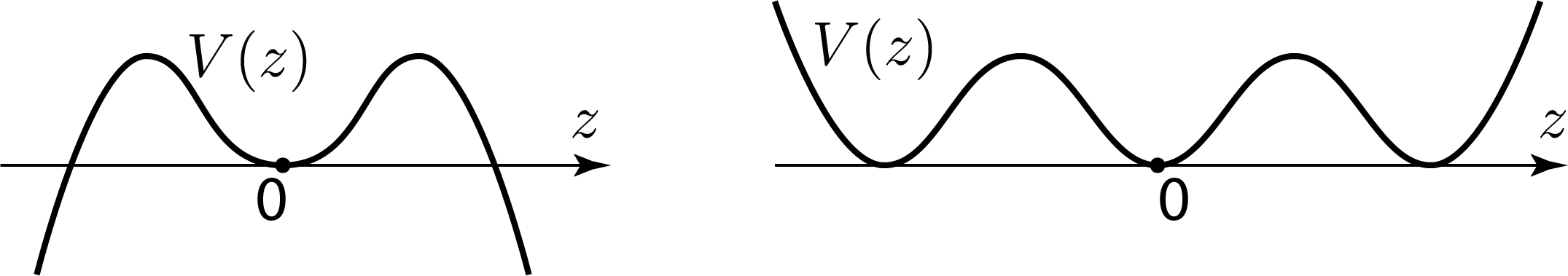}
\caption{Graphs of the restrictions to the real line of the two examples of potentials $z\mapsto V(z)$ for which hypothesis \hypDiscStatRef{\valueOfV} does not hold.}
\label{fig:GL_potentials}
\end{figure}
\[
V: \cc\simeq\rr^2 \to \rr\,, \quad
z  \mapsto \frac{\abs{z}^2}{2}-\frac{\abs{z}^4}{4}
\quad\text{or}\quad
z \mapsto \frac{\abs{z}^2}{2}-\frac{4}{\sqrt{3}}\frac{\abs{z}^4}{4}+\frac{\abs{z}^6}{6}
\quad\text{?}
\]
\item Is it possible to construct an example where conclusion \cref{item:thm_main_approach_standing_terrace} of \cref{thm:main} holds, where the number $q$ of items involved in the standing terrace equals $1$, but where the ``position'' $x_1(t)$ does not converge as time goes to $+\infty$? (note that this surely requires that the stationary solution be ``degenerated'' in the sense that it be not a hyperbolic equilibrium for the semi-flow of system \cref{init_syst}). On the other hand, does $x_1(t)$ always converge if $V$ is analytic? (see \cite{Simon_asymptoticsEvolEqu_1983}).
\end{enumerate}
\subsection{Organization of the paper} 
\begin{itemize}
\item The next \cref{sec:preliminaries} is devoted to some preliminaries (functional framework, existence of solutions, preliminary computations on spatially localized functionals, notation).
\item The preliminary results \cref{lem:sufficient_condition_stability_right_end_of_space,lem:upper_bound_on_invasion_speed,lem:exponential_decrease_beyond_invasion_speed,cor:bist} on spatial asymptotics of bistable solutions are proved in \cref{sec:spat_asympt}. 
\item \Cref{prop:asympt_en} (existence of asymptotic energy) is proved in \cref{sec:asympt_en}. 
\item \Cref{prop:scs_asympt_en} (upper semi-continuity of asymptotic energy) is proved in \cref{sec:upp_semicont_asympt_en}.
\item Conclusions \cref{item:thm_main_time_derivative_goes_to_zero,item:thm_main_invasion_speeds_vanish} of \cref{thm:main} are proved in \cref{sec:relaxation}.
\item Conclusion \cref{item:thm_main_nonnegative_asympt_energy_approach_bistable_stationary} of \cref{thm:main} is proved in \cref{sec:approach_set_bist_sol}.
\item Conclusion \cref{item:thm_main_approach_standing_terrace_and_value_asymptotic_energy} of \cref{thm:main} is proved in \cref{sec:cv_standing_terrace_value_asympt_en}. 
\end{itemize}
The remaining \namecrefs{sec:exist_res_bas_att} can be viewed as appendices. 
\begin{itemize}
\item \Cref{sec:exist_res_bas_att} is devoted to some standard results (\cref{cor:nonneg_stat,cor:neg_stat,cor:nonneg_dyn,cor:neg_dyn}) concerning existence of homoclinic or heteroclinic stationary solutions and the basin of attraction of a stable homogeneous solution, retrieved as direct consequences of \cref{thm:main,prop:scs_asympt_en}.
\item Elementary examples illustrating the results --- and the questions raised --- are discussed in \cref{sec:examples}.
\item The proof of the existence of an attracting ball for the semi-flow follows from the coercivity hypothesis \cref{hyp_coerc} and is given in \cref{sec:att_ball}. 
\item \Cref{sec:prop_ham_app} is devoted to two lemmas concerning stationary solutions of system~\cref{init_syst}, extensively used in \cref{sec:approach_set_bist_sol} to prove the approach to the set $I\bigl(\PhiZero(\valueOfV)\bigr)$.
\item Finally, a rough discussion of the map between initial conditions and the space of asymptotic patterns (and the regularity of this map) is carried out in \cref{sec:sp_as_patt}.
\end{itemize} 
\section{Preliminaries}
\label{sec:preliminaries}
As everywhere else, let us consider a function $V$ in $\ccc^2(\rr^d,\rr)$ satisfying the coercivity hypothesis \cref{hyp_coerc}. 
\subsection{Functional framework}
\label{subsec:funct_fram}
For $u$ in $H^1_\text{loc}(\rr,\rr^d)$, let 
\begin{equation}
\label{def_Honeul_norm}
\norm{u}_{\Honeul} = \sup_{\widebar{x}\in\rr} \, \left(\int_{\widebar{x}}^{\widebar{x}+1}\bigl(\abs{u(x)}^2 + \abs{u'(x)}^2 \bigr) \, dx\right)^{1/2} 
= \sup_{\widebar{x}\in\rr} \, \norm{u}_{H^1([\widebar{x},\widebar{x}+1],\rr^d)} \le \infty
\,,
\end{equation}
and let us consider the uniformly local Sobolev space $X$ defined as
\[
\begin{aligned}
X &= 
\Honeul \\
&= \left\{ u \in H^1_\text{loc}(\rr,\rr^d) : \norm{u}_{\Honeul} < \infty 
\quad\text{and}\quad 
\lim_{\widebar{x}\to 0} \norm{T_{\widebar{x}} u - u}_{\Honeul} = 0 \right\}
\,.
\end{aligned}
\]
As already mentioned in \cref{subsubsec:coerc_glob_exist}, this space is the most convenient with respect to the estimates on localized energy and $L^2$-norm that are used all along the paper. 
\subsection{Global existence of solutions and attracting ball for the semi-flow}
\label{subsec:glob_exist}
Since $V$ is assumed to be of class $\ccc^2$, the function $u\mapsto\nabla V(u)$ is of class $\ccc^1$, and therefore the nonlinearity $u(\cdot)\mapsto -\nabla V\bigl(u(\cdot)\bigr)$ in system~\cref{init_syst} is locally Lipschitz in $X$. Thus local existence of solutions in that space follows from general results (see for instance Henry's book \cite{Henry_geomSemilinParab_1981}). 

More precisely, for every $u_0$ in $X$, system~\cref{init_syst} has a unique (mild) solution $t\mapsto S_t u_0$ in $\ccc^0\bigl([0,T_{\max}),X\bigr)$ with initial condition $u_0$. This solution depends continuously on the initial condition $u_0$ and is defined up to a (unique) maximal time of existence $T_{\max}=T_{\max}[u_0]$ in $(0,+\infty]$. The following global existence result is proved in \cref{sec:att_ball}.
\begin{proposition}[global existence of solutions and attracting ball]
\label{prop:glob_exist_sol_att_ball} 
For every function $u_0$ in $X$, the solution $t\mapsto S_t u_0$ of system~\cref{init_syst} with initial condition $u_0$ is defined up to $+\infty$ in time. In addition, there exists positive quantities
\[
\RmaxInfty[u_0]
\quad\text{and}\quad
\Tatt[u_0]
\quad\text{and}\quad
\RattInfty
\]
such that
\begin{align}
\label{maximal_radius_excursion_Linfty}
\sup_{t\ge0} \norm{x\mapsto(S_t u_0)(x)}_{\Linfty} &\le \RmaxInfty[u_0] \\
\text{and}\qquad
\label{asymptotic_radius_Linfty}
\sup_{t\ge \Tatt[u_0]}\norm{x\mapsto(S_t u_0)(x)}_{\Linfty} &\le \RattInfty
\,.
\end{align}
The quantity $\RattInfty$ depends only on $V$ and $\ddd$, whereas $\RmaxInfty[u_0]$ and $\Tatt[u_0]$ depend also on the initial condition $u_0$ (more specifically, on $\norm{u_0}_{X}$). 
\end{proposition}
%
%
In addition, system~\cref{init_syst} has smoothing properties (Henry \cite{Henry_geomSemilinParab_1981}). For every nonnegative integer $k$ and every quantity $\alpha$ in $(0,1)$, let us recall that the Hölder space $\cccb{k,\alpha}$ is defined by the norm
\[
\norm{u}_{\cccb{k,\alpha}} = \sup_{x\in\rr}\abs{u(x)} + \sup_{x\in\rr}\abs{u'(x)} + \dots + \sup_{x\in\rr}\abs{u^{(k)}(x)} + \sup_{(x,y)\in\rr^2,\, x\not=y}\frac{\abs{u(x)-u(y)}}{\abs{x-y}^\alpha}
\,.
\]
Due to these smoothing properties, since $V$ is of class $\ccc^2$ (and as a consequence the nonlinearity $v\mapsto -\nabla V(v)$ is of class $\ccc^1$), for every quantity $\alpha$ in the interval $(0,1)$, every solution $t\mapsto S_t u_0$ in $\ccc^0([0,+\infty),X)$ actually belongs to
\[
\ccc^0\left((0,+\infty),\cccb{2,\alpha}\right)\cap \ccc^1\left((0,+\infty),\cccb{0,\alpha}\right)
\,,
\]
and, for every positive quantity $\varepsilon$, the quantities
\begin{equation}
\label{bound_u_ut_ck}
\sup_{t\ge\varepsilon}\norm{S_t u_0}_{\cccb{2,\alpha}}
\quad\text{and}\quad
\sup_{t\ge\varepsilon}\norm{\frac{d(S_t u_0)}{dt}(t)}_{\cccb{0,\alpha}}
\end{equation}
are finite. 
\subsection{Asymptotic compactness}
The following standard compactness statement (see for instance \cite[1963]{MatanoPolacik_entireSolutionBistableParabEquTwoCollidingPulses_2017}, from where the notation and sketch of proof below are reproduced) will be called upon several times in \cref{sec:relaxation,sec:approach_set_bist_sol,sec:cv_standing_terrace_value_asympt_en}. 
\begin{lemma}[asymptotic compactness]
\label{lem:compactness}
For every solution $(x,t)\mapsto u(x,t)$ of system \cref{init_syst}, and for every sequence $(x_n,t_n)_{n\in\nn}$ in $\rr\times[0,+\infty)$ such that $t_n\to+\infty$ as $n\to+\infty$, there exists a entire solution $\widebar{u}$ of system \cref{init_syst} in 
\[
\ccc^0\left(\rr,\cccb{2}\right)\cap \ccc^1\left(\rr,\cccb{0}\right)
\,,
\]
such that, up to replacing the sequence $(x_n,t_n)_{n\in\nn}$ by a subsequence, 
\begin{equation}
\label{compactness}
D^{2,1}u(x_n+\cdot,t_n+\cdot)\to D^{2,1}\widebar{u}
\quad\text{as}\quad
n\to+\infty
\,,
\end{equation}
uniformly on every compact subset of $\rr^2$, where the symbol $D^{2,1}v$ stands for $(v,v_x,v_{xx},v_t)$ (for $v$ equal to $u$ or $\widebar{u}$). 
\end{lemma}
\begin{proof}
Let us consider the sequence of functions 
\[
\rr\times(-t_n,+\infty)\to\rr^d\,,
\quad
(\xi,s)\mapsto u(x_n+\xi,t_n+s)\,,
\quad n\in\nn
\,.
\]
According to the Hölder estimates \cref{bound_u_ut_ck} on the solution $u$, up to replacing the sequence $(x_n,t_n)$ by a subsequence, the convergence \cref{compactness} holds on any given compact subset of $\rr^2$. The conclusion follows from a diagonal extraction procedure. 
\end{proof}
\subsection{Time derivative of (localized) energy and \texorpdfstring{$L^2$}{L2}-norm of a solution}
\label{subsec:1rst_ord}
Let $(x,t)\mapsto u(x,t)$ denote a solution of system~\cref{init_syst} and let $m$ be a point in $\mmm$. Key ingredients in the proofs rely on appropriate combinations of the two most natural functionals to consider, namely the energy (Lagrangian) and the $L^2$-norm of the distance to $m$, defined (at least formally) as
\[
\int_{\rr}\Bigl(\frac{1}{2}\abs{u_x(x,t)}_{\ddd}^2 + V\bigl(u(x,t)\bigr)-V(m)\Bigr)\, dx
\quad\text{and}\quad
\int_{\rr}\frac{1}{2}\bigl(u(x,t)-m\bigr)^2 \, dx
\,.
\]
Let $x\mapsto\psi(x)$ denote a function in the space $W^{2,1}(\rr,\rr)$ (that is a function belonging to $L^1(\rr)$, together with its first and second derivatives). To simplify the presentation, let us assume here that 
\begin{equation}
\label{assumption_m_is_the_origin_and_V_of_m_equals_0}
m=0_{\rr^d}
\quad\text{and}\quad
V(m)=V(0_{\rr^d}) = 0
\,.
\end{equation}
In order to deal with convergent integrals, let us multiply by $\psi$ the integrands of the two aforementioned functionals. Then, the time derivatives of these functionals read
\begin{equation}
\label{ddt_loc_en}
\frac{d}{d t}\int_{\rr} \psi\ \Bigl(\frac{1}{2}\abs{u_x}_{\ddd} ^2+V(u)\Bigr)\, dx = \int_{\rr} \bigl( - \psi\ u_t^2 - \psi'\ \ddd u_x \cdot u_t \bigr) \, dx
\end{equation}
and 
\begin{equation}
\label{ddt_loc_L2}
\begin{aligned}
\frac{d}{d t}\int_{\rr} \psi\frac{1}{2}u^2\, dx & = \int_{\rr} \Bigl( \psi\ \bigl( - u\cdot \nabla V(u) - \abs{u_x}_{\ddd} ^2 \bigr) - \psi'\ u \cdot \ddd u_x \Bigr) \, dx \\
& = \int_{\rr} \Bigl( \psi\ \bigl( - u\cdot \nabla V(u) - \abs{u_x}_{\ddd} ^2 \bigr) +  \frac{1}{2}\psi''\abs{u}_{\ddd} ^2 \Bigr) \, dx \,. 
\end{aligned}
\end{equation}
Here are some basic observations about these expressions.
\begin{itemize}
\item The variation of the (localized) energy is the sum of a (nonpositive) ``dissipation'' term and a additional ``flux'' term. 
\item The variation of the (localized) $L^2$-norm is similarly made of two ``main'' terms and an additional ``flux'' term. Among the two main terms, the second one is nonpositive, and so is the first one if the quantity $u\cdot\nabla V(u)$ is positive, that is:
\begin{itemize}
\item for $\abs{u}$ large (according to the coercivity hypothesis \cref{hyp_coerc} on $V$);
\item for $\abs{u}$ small according to the assumption \cref{assumption_m_is_the_origin_and_V_of_m_equals_0}. 
\end{itemize}
\item The second integration by parts that is performed on the last term of the expression \cref{ddt_loc_L2} of the time derivative of the $L^2$-functional will lead to slightly simpler calculations, but is not essential.
\item The slower the weight function $\psi$ varies, the smaller the flux terms are. More precisely, it seems relevant to choose $\psi$ as a function satisfying, for a small positive quantity $\varepsilon$,
\[
\abs{\psi'(x)}\le\varepsilon\psi(x)
\quad\text{and}\quad
\abs{\psi''(x)}\le\varepsilon\psi(x)
\quad\text{for all }x\text{ in }\rr.
\]
This way, if $\varepsilon$ is small enough, the flux terms might very well be ``dominated'' by the other terms of the right-hand sides of equalities \cref{ddt_loc_en,ddt_loc_L2}. 
\item An appropriate combination of these two functionals might display coercivity properties, again for $\abs{u}$ large (according to the coercivity hypothesis \cref{hyp_coerc} on $V$) and for $\abs{u}$ small if $0_{\rr^d}$ is in the set $\mmm_0\,$. 
\end{itemize}
These observations will be put in practice several times along the following pages: 
\begin{enumerate}
\item to prove the existence of an attracting ball for the flow (\cref{sec:att_ball});
\item to gain some control on the spatial asymptotics of bistable solutions (\cref{sec:spat_asympt,sec:upp_semicont_asympt_en});
\item to state the approximate decrease of localized energies (\cref{subsec:localized_energy,sec:upp_semicont_asympt_en}). For those localized energies the weight function that will be used (denoted by $\chi$ instead of $\psi$) will depend not only on $x$ but also on $t$, thus the right-hand side of equality \cref{ddt_loc_en} will comprise an additional ``flux'' term with weight $\chi_t\,$. 
\end{enumerate}
\subsection{Miscellanea}
\label{subsec:misc}
\subsubsection{Notation for the eigenvalues of the diffusion matrix}
Let $\eigDmin$ ($\eigDmax$) denote the smallest (respectively, largest) of the (positive) eigenvalues of the diffusion matrix $\ddd$; the following inequalities hold: 
\[
0<\eigDmin\le\eigDmax
\,.
\]
\subsubsection{Second order estimates for the potential around a minimum point}
\begin{lemma}[second order estimates for the potential around a minimum point]
\label{lem:estim_from_def_escape}
For every $m$ in $\mmm$ and every $u$ in $\rr^d$ satisfying $\abs{u-m}_{\ddd} \le\dEsc(m)$, the following estimates hold:
\begin{align}
\label{posit_pot_around_loc_min}
V(u)-V(m)&\ge \frac{\eigVmin(m)}{4} (u-m)^2 \,, \\
\text{and}\qquad
\label{v_nablaV_controls_square_around_loc_min}
(u-m)\cdot \nabla V(u) &\ge \frac{\eigVmin(m)}{2} (u-m)^2 \,, \\
\text{and}\qquad
\label{v_nablaV_controls_pot_around_loc_min}
(u-m)\cdot \nabla V(u) &\ge V(u)-V(m)
\,.
\end{align}
\end{lemma}
\begin{proof}
Take $m$ in $\mmm$ and $u$ in $\rr^d$ such that $\abs{u-m}_{\ddd}$ is not larger than $\dEsc(m)$, let us write
\[
u-m = v
\iff
u=m+v
\,,
\]
and let us introduce the $\ccc^2$-function $f$ defined on $[0,1]$ by 
\[
f(\theta)=V(m+\theta v)-V(m)
\,.
\]
Then, 
\[
f(0)=0
\qquad\text{and}\quad
f'(0)=0
\qquad\text{and}\quad
f(1)=V(u)-V(m)
\,,
\]
and for every $\theta$ in $[0,1]$,
\[
f'(\theta) = v\cdot \nabla V(m+\theta v)
\qquad\text{and}\qquad
f''(\theta) = \Bigl( D^2 V(m+\theta v)\cdot v  \Bigr) \cdot v 
\,,
\]
and thus according to inequality \vref{property_dEsc} defining $\dEsc(m)$, for every $\theta$ in $[0,1]$,
\begin{equation}
\label{lower_bound_f_prime_prime}
f''(\theta) \ge \frac{\eigVmin(m)}{2} v^2
\,.
\end{equation}
Now, according to Taylor's Theorem with Lagrange remainder, 
\[
\begin{aligned}
f(1) &= f(0) + f'(0) + \frac{1}{2} f''(\theta_1)
\quad\text{for some } \theta_1 \text{ in } (0,1) \,, \\
\text{and}\qquad
f'(1) &= f'(0) + f''(\theta_2) 
\quad\text{for some } \theta_2 \text{ in } (0,1) \,, \\
\text{and}\qquad
f(0) &= f(1) - f'(1) + \frac{1}{2} f''(\theta_3)
\quad\text{for some } \theta_3 \text{ in } (0,1) 
\,,
\end{aligned}
\]
and in view of the lower bound \cref{lower_bound_f_prime_prime} inequalities \cref{posit_pot_around_loc_min,v_nablaV_controls_square_around_loc_min,v_nablaV_controls_pot_around_loc_min} follow from these three equalities, respectively. \Cref{lem:estim_from_def_escape} is proved. 
\end{proof}
\begin{remark}
Inequality \cref{posit_pot_around_loc_min} above will actually only be used under the weaker form 
\begin{equation}
\label{nonnegative_pot_around_loc_min}
V(u) - V(m) \ge 0
\end{equation}
(for every $u$ in $\rr^d$ satisfying $\abs{u-m}_{\ddd} \le\dEsc(m)$).
\end{remark}
\subsubsection{Lower quadratic hull for the potential at minimum points}
\label{subsubsec:min_conv}
For the computations carried in the next \cref{sec:spat_asympt}, it will be convenient to introduce the quantity $\qLowHull$ defined as the minimum of the convexities of the lower quadratic hulls of $V$ at the points of $\mmm$ (see \cref{fig:low_quad_hull}). 
\begin{figure}[!htbp]
\centering
\includegraphics[width=0.5\textwidth]{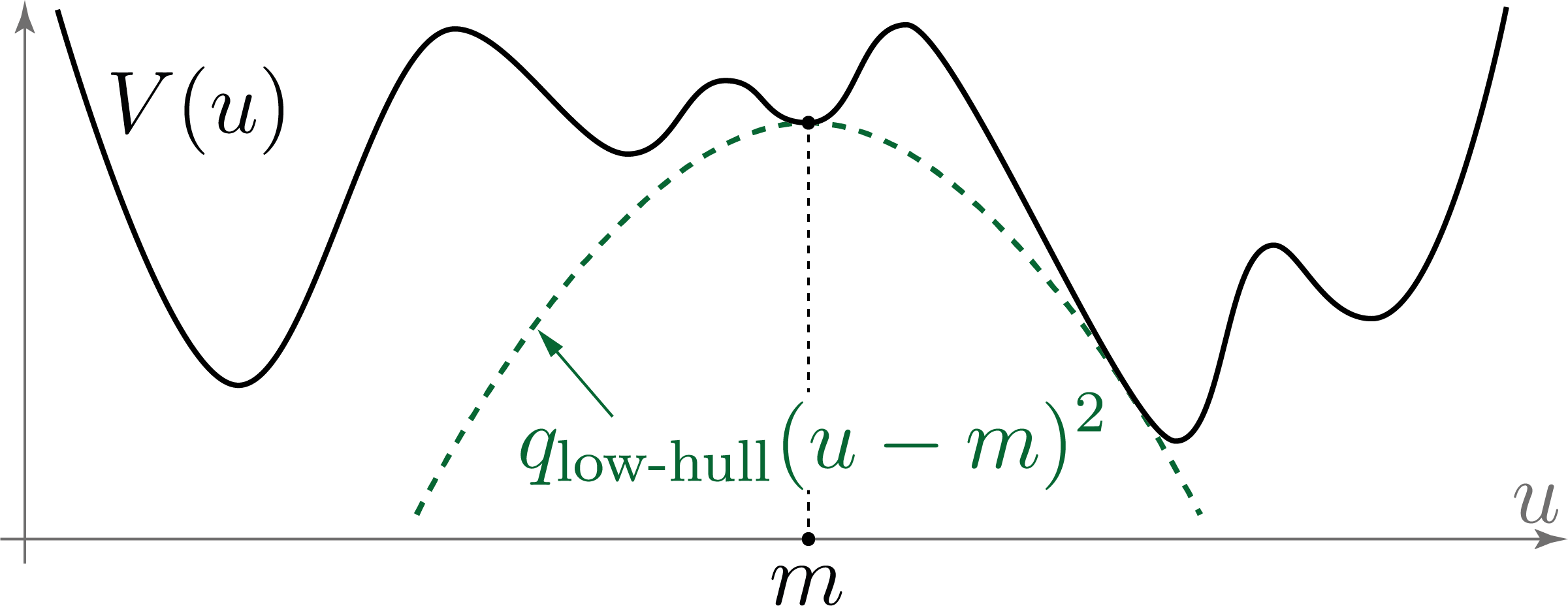}
\caption{Lower quadratic hull of the potential at a minimum point (definition of the quantity $\qLowHull$).}
\label{fig:low_quad_hull}
\end{figure}
With symbols,
\[
\qLowHull = \min_{m\in\mmm}\ \inf_{u\in\rr^d\setminus \{m\}}\ \frac{V(u)-V(m)}{(u-m)^2}
\,.
\]
This quantity $\qLowHull$ is negative as soon as $m$ is not a global minimum point of $V$ (and nonnegative otherwise), and according to hypothesis \cref{hyp_coerc} it is finite (in other words it is not equal to $-\infty$). This definition ensures that, for every $m$ in $\mmm$ and for all $u$ in $\rr^d$,
\begin{equation}
\label{ineq_neg_hull}
V(u)-V(m) - \qLowHull (u-m)^2\ge 0
\,,
\end{equation}
see \cref{fig:low_quad_hull}. Let us introduce the following quantity (it will be used as the \emph{coefficient of the energy} in the firewall function defined in \cref{subsec:firewall}):
\[
\coeffEn=\frac{1}{\max(1, - 4\, \qLowHull)}
\,.
\]
It follows from this definition that $\coeffEn$ is in $(0,1]$ and that, for every $m$ in $\mmm$ and for all $u$ in $\rr^d$,
\begin{equation}
\label{def_weight_en}
\coeffEn \bigl(V(u)-V(m)\bigr) +  \frac{1}{4}(u-m)^2\ge 0
\,.
\end{equation}
\section{Stability at one end of space}
\label{sec:spat_asympt}
The aim of this \namecref{sec:spat_asympt} is to provide preliminary results concerning the solutions stable at one end of space, and in particular to prove the results of \cref{subsubsec:preliminary_results_stability_at_ends_of_space}, namely \cref{lem:sufficient_condition_stability_right_end_of_space,lem:upper_bound_on_invasion_speed,lem:exponential_decrease_beyond_invasion_speed,cor:bist}.
\subsection{Set-up}
\label{subsec:setup_spat_as}
As everywhere else, let us consider a function $V$ in $\ccc^2(\rr^d,\rr)$ satisfying the coercivity hypothesis \cref{hyp_coerc}. Let $m$ be a point in $\mmm$, let $u_0$ be a function in $X$, and let $(x,t)\mapsto u(x,t)$ denote the solution of system \cref{init_syst} corresponding to the initial condition $u_0$. For notational convenience, let us introduce the ``normalized potential'' $V^\dag$ and the ``normalized solution'' $u^\dag$ defined as
\begin{equation}
\label{def_normalized_potential_solution}
V^\dag(v)=V(m +v)-V(m)
\quad\text{and}\quad
u^\dag(x,t) = u(x,t)-m
\,.
\end{equation}
Thus the origin $0_{\rr^d}$ of $\rr^d$ is to $V^\dag$ what $m$ is to $V$, and $u^\dag$ is a solution of system \cref{init_syst} with potential $V^\dag$ instead of $V$; and, for all $(x,t)$ in $\rr\times[0,+\infty)$, 
\[
V^\dag\bigl(u^\dag(x,t)\bigr) = V\bigl(u(x,t)\bigr)-V(m)
\,.
\]
It follows from inequality \cref{def_weight_en} satisfied by $\coeffEn$ that, for all $v$ in $\rr^d$, 
\begin{equation}
\label{def_weight_en_V_dag}
\coeffEn \, V^\dag(v) +  \frac{v^2}{4}\ge 0
\,,
\end{equation}
and it follows from inequalities \cref{v_nablaV_controls_square_around_loc_min,v_nablaV_controls_pot_around_loc_min} that, for all $v$ in $\rr^d$ satisfying $\abs{v}\le\dEsc(m)$, 
\begin{align}
\label{v_nablaV_controls_square_around_loc_min_dag}
v\cdot \nabla V^\dag(v) &\ge \frac{\eigVmin(m)}{2} v^2 \,, \\
\text{and}\qquad
\label{v_nablaV_controls_pot_around_loc_min_dag}
v\cdot \nabla V^\dag(v) &\ge V^\dag(v)
\,.
\end{align}
\subsection{Firewalls}
\label{subsec:firewall}
\subsubsection{Definition} 
\label{subsubsec:firewall_definition}
The proof relies on the definition of a functional that is an appropriate combination of the energy and the $L^2$-norm of the solution, localized by an appropriate weight function (see \cref{subsec:1rst_ord} and comments therein). As already mentioned in~\cref{subsec:1rst_ord}, the key points are:
\begin{itemize}
\item to choose the coefficients for the energy and the $L^2$-norm in such a way that the resulting function is coercive;
\item to choose a weight function that varies slowly enough to recover from expressions~\cref{ddt_loc_en,ddt_loc_L2} some decrease of the resulting function.
\end{itemize}

Concerning the first of these two points, the quantity $\coeffEn$ defined in \namecref{subsubsec:min_conv} \ref{subsubsec:min_conv} 
is a convenient coefficient for energy if the coefficient for the $L^2$-norm is chose equal to $1$, as can be seen from inequality \cref{def_weight_en_V_dag}. Concerning the second point, let $\kappa$ denote a positive quantity, small enough so that 
\begin{equation}
\label{conditions_on_kappa}
\frac{\coeffEn \, \kappa^2 \, \eigDmax}{4} \le\frac{1}{2} 
\quad\text{and}\quad
\frac{\kappa^2\, \eigDmax}{2}\le\frac{\eigVmin(m)}{8}
\end{equation}
(those conditions will be used to prove inequality \cref{dt_fire_spat_as} below); this quantity may be chosen as
\begin{equation}
\label{def_kappa}
\kappa = \min\Biggl(\sqrt{\frac{2}{\coeffEn\, \eigDmax}},\sqrt{\frac{\eigVmin(m)}{4\, \eigDmax}}\Biggr)
\,.
\end{equation}
Let us introduce the weight function $\psi$ defined as 
\begin{equation}
\label{def_psi}
\psi(x) = \exp(-\kappa\abs{x})
\,.
\end{equation}
For $\widebar{x}$ in $\rr$, let $T_{\widebar{x}}\psi$ denote the translate of $\psi$ by $\widebar{x}$, that is the function defined as
\[
T_{\widebar{x}}\psi(x) = \psi (x-\widebar{x})
\,,
\]
see \cref{fig:weight_fire}. 
\begin{figure}[!htbp]
\centering
\includegraphics[width=\textwidth]{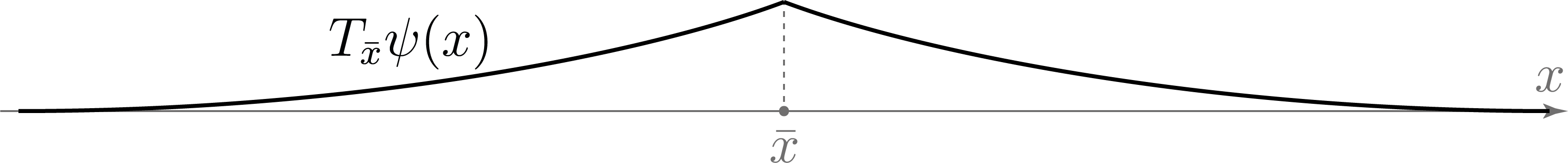}
\caption{Graph of the weight function $x\mapsto T_{\bar{x}}\psi(x)$ used to define the firewall function $\fff(\bar{x},t)$. The slope is small, according to the definition of $\kappa$.}
\label{fig:weight_fire}
\end{figure}
For $x$ in $\rr$ and $t$ in $[0,+\infty)$, let 
\[
E^\dag(x,t) = \frac{1}{2}\abs{u^\dag_x(x,t)}_{\ddd}^2 + V^\dag \bigl(u^\dag(x,t)\bigr)
\quad\text{and}\quad
F^\dag(x,t) = \coeffEn E^\dag(x,t) + \frac{1}{2}u^\dag(x,t)^2
\,,
\]
and for $\widebar{x}$ in $\rr$ and $t$ in $[0,+\infty)$, let 
\begin{equation}
\label{def_fff}
\fff(\widebar{x},t) = \int_{\rr} T_{\widebar{x}}\psi(x)\, F^\dag(x,t)\, dx
\,.
\end{equation}
\subsubsection{Coercivity} 
\begin{lemma}[coercivity of $F^\dag(x,t)$]
\label{lem:coerc_fire}
For all $t$ in $[0,+\infty)$ and $x$ in $\rr$, 
\begin{align}
\label{coerc_fire}
F^\dag(x,t) &\ge \min\Bigl(\frac{\coeffEn}{2},\frac{1}{4}\Bigr) \left( \abs{u^\dag_x(x,t)}_{\ddd} ^2 + u^\dag(x,t)^2 \right)
\,,\\
\text{and}\quad
\label{coerc_fire_hone}
F^\dag(x,t) &\ge \min\Bigl(\frac{\coeffEn\eigDmin}{2},\frac{1}{4}\Bigr) \left( \abs{u^\dag_x(x,t)}^2 + u^\dag(x,t)^2 \right)
\,,\\
\text{thus in particular}&\quad
\label{coerc_fire_nonnegative}
F^\dag(x,t)\ge 0\,,
\quad\text{and as a consequence}\quad 
\fff(x,t)\ge 0
\,.
\end{align}
\end{lemma}
\begin{proof}
Both inequalities \cref{coerc_fire,coerc_fire_hone} follow from inequality \cref{def_weight_en_V_dag}.
\end{proof}
The function $\fff(\bar x,t)$ will play the role of a ``firewall'', in the sense that its approximate decrease will enable to control the solution in the part of space where it is not too far from the minimum point $0_{\rr^d}$ (and consequently to control the flux term in the derivative of the localized energy in the next \namecref{sec:spat_asympt}). The notation $\fff$ relates to this interpretation. This approximate decrease is formalized by the next lemma. 
\subsubsection{Linear decrease up to pollution} 
For $t$ in $[0,+\infty)$, let us introduce the set (the domain of space where the normalized solution $u^\dag$ ``Escapes'' at a certain distance from $0_{\rr^d}$):
\begin{equation}
\label{def_SigmaEsc}
\SigmaEsc(t)=\left\{x\in\rr: \abs{u^\dag(x,t)}_{\ddd} >\dEsc(m)\right\} 
\,.
\end{equation}
\begin{lemma}[firewall linear decrease up to pollution]
\label{lem:approx_decrease_fire}
There exist positive quantities $\nuF$ and $\KF$ such that, for all $\widebar{x}$ in $\rr$ and all $t$ in $[0,+\infty)$,
\begin{equation}
\label{dt_fire}
\partial_t \fff(\widebar{x},t)\le-\nuF\, \fff(\widebar{x},t) + \KF\, \int_{\SigmaEsc(t)} T_{\widebar{x}}\psi(x)\, dx
\,.
\end{equation}
The quantity $\nuF$ depends on $V$ and $\ddd$ and $m$ (only), whereas $\KF$ depends additionally on the upper bound on the $L^{\infty}$-norm of the solution.
\end{lemma}
\begin{proof}
It follows from expressions \vref{ddt_loc_en,ddt_loc_L2} that, for all $\widebar{x}$ in $\rr$ and all $t$ in $[0,+\infty)$,
\[
\begin{aligned}
\partial_t &\fff (\widebar{x},t) = \\
& \int_{\rr} \biggl[ 
T_{\widebar{x}} \psi \Bigl( - \coeffEn\, (u_t^\dag)^2 - u^\dag\cdot \nabla V^\dag(u^\dag) -  \abs{u^\dag_x}_{\ddd} ^2 \Bigr) 
- T_{\widebar{x}} \psi' \bigl( \coeffEn\, \ddd u^\dag_x\cdot u^\dag_t \bigr) 
+ \frac{T_{\widebar{x}} \psi''}{2}\abs{u^\dag}_{\ddd} ^2 
\biggr] \, dx
\,.
\end{aligned}
\]
Since 
\[
\abs{\psi'(\cdot)}= \kappa \, \psi(\cdot)
\quad\text{and}\quad
\psi''(\cdot)\le \kappa^2 \psi(\cdot)
\]
(indeed $\psi''(\cdot)$ equals $\kappa^2 \psi(\cdot)$ plus a Dirac mass of negative weight), it follows that
\[
\partial_t\fff (\widebar{x},t) \le \int_{\rr} T_{\widebar{x}} \psi \Bigl( - \coeffEn\, (u_t^\dag)^2 - u^\dag\cdot \nabla V^\dag(u^\dag) -  \abs{u^\dag_x}_{\ddd} ^2 
+ \coeffEn\, \kappa \abs{\ddd u^\dag_x\cdot u^\dag_t} + \frac{\kappa^2}{2} \abs{u^\dag}_{\ddd}^2\Bigr) \, dx 
\,,
\]
thus, using the inequalities
\begin{equation}
\label{polarize_kappa_ddd_ux_ut}
\begin{aligned}
\kappa\abs{\ddd u^\dag_x\cdot u^\dag_t} &\le \kappa \abs{\sqrt{\ddd}(\sqrt{\ddd}u^\dag_x)\cdot u^\dag_t} \\
&\le \kappa \sqrt{\eigDmax} \abs{u^\dag_x}_{\ddd}\abs{u^\dag_t} \\
&\le  (u^\dag_t)^2 + \frac{\kappa^2\,\eigDmax}{4}\abs{u^\dag_x}_{\ddd}^2
\,,
\end{aligned}
\end{equation}
it follows that
\begin{equation}
\label{dt_fire_spat_as_before_applying_conditions_on_parameters}
\partial_t\fff (\widebar{x},t)\le
\int_{\rr} T_{\widebar{x}} \psi \biggl(\Bigl(\frac{\coeffEn \, \kappa^2\, \eigDmax}{4}-1\Bigr)\abs{u^\dag_x}_{\ddd} ^2 - u^\dag\cdot \nabla V^\dag(u^\dag) + \frac{\kappa^2 \, \eigDmax}{2} (u^\dag)^2 \biggr)\, dx 
\,,
\end{equation}
and according to the conditions \cref{conditions_on_kappa} satisfied by the quantity $\kappa$,
\begin{equation}
\label{dt_fire_spat_as}
\partial_t\fff(\widebar{x},t)\le
\int_{\rr} T_{\widebar{x}} \psi \Bigl(-\frac{1}{2}\abs{u^\dag_x}_{\ddd}^2 - u^\dag\cdot \nabla V^\dag(u^\dag) + \frac{\eigVmin(m)}{8} (u^\dag)^2 \Bigr)\, dx 
\,. 
\end{equation}
Let $\nuF$ be a positive quantity to be chosen below. It follows from the previous inequality and from the definition \cref{def_fff} of $\fff(\widebar{x},t)$ that
\begin{equation}
\label{dt_Fire_preliminary}
\begin{aligned}
\partial_t\fff(\widebar{x},t) + \nuF\fff(\widebar{x},t) \le \int_{\rr} T_{\widebar{x}} \psi \biggl[&-\frac{1}{2}(1-\nuF\,\coeffEn)\abs{u^\dag_x}_{\ddd}^2 - u^\dag\cdot \nabla V^\dag(u^\dag) \\
&  + \nuF\,\coeffEn\, V^\dag(u^\dag) +   \Bigl(\frac{\eigVmin(m)}{8} + \frac{\nuF}{2}\Bigr)(u^\dag)^2 \biggr]\, dx 
\,.
\end{aligned}
\end{equation}
In view of this expression and of inequalities \cref{v_nablaV_controls_square_around_loc_min_dag,v_nablaV_controls_pot_around_loc_min_dag}, let us assume that $\nuF$ is small enough so that
\begin{equation}
\label{conditions_on_nuFire}
\nuF\, \coeffEn \le 1 
\quad\text{and}\quad
\nuF\, \coeffEn \le \frac{1}{2}
\quad\text{and}\quad
\frac{\nuF}{2} \le \frac{\eigVmin(m)}{8} 
\,;
\end{equation}
the quantity $\nuF$ may be chosen as
\begin{equation}
\label{def_nuF}
\nuF = \min\Bigl(\frac{1}{2\coeffEn}, \frac{\eigVmin(m)}{4} \Bigr)
\,.
\end{equation}
Then, it follows from \cref{dt_Fire_preliminary,conditions_on_nuFire} that
\begin{equation}
\label{dt_Fire_preliminary_bis}
\partial_t\fff(\widebar{x},t) + \nuF\fff(\widebar{x},t) \le \int_{\rr} T_{\widebar{x}} \psi \Bigl[- u^\dag\cdot \nabla V^\dag(u^\dag) +\frac{1}{2}\abs{V^\dag(u^\dag)} +\frac{\eigVmin(m)}{4} (u^\dag)^2 \Bigr] \, dx 
\,.
\end{equation}
According to \cref{v_nablaV_controls_square_around_loc_min_dag,v_nablaV_controls_pot_around_loc_min_dag}, the integrand of the integral at the right-hand side of this inequality is nonpositive as long as $x$ is \emph{not} in $\SigmaEsc(t)$. Therefore this inequality still holds if the domain of integration of this integral is changed from $\rr$ to $\SigmaEsc(t)$. Besides, observe that, in terms of the ``initial'' potential $V$ and solution $u(x,t)$, the factor of $T_{\widebar{x}} \psi$ under the integral of the right-hand side of this last inequality reads
\[
- (u-m)\cdot \nabla V(u)+\frac{1}{2} \abs{V(u)-V(m)}+ \frac{\eigVmin(m)}{4} (u-m)^2
\,.
\]
Thus, if $K_{\fff}$ denotes the maximum of this expression over all possible values for $u$, that is (according to the $L^\infty$ bound \cref{maximal_radius_excursion_Linfty} on the solution) the quantity
\begin{equation}
\label{def_KF}
K_{\fff} = \max_{v\in\rr^d,\ \abs{v}\le \RmaxInfty[u_0]}\Bigl[ - (v-m)\cdot \nabla V(v)+ \frac{1}{2} \abs{V(v)-V(m)} + \frac{\eigVmin(m)}{4} (v-m)^2\Bigr]
\,,
\end{equation}
then inequality~\cref{dt_fire} follows from inequality \cref{dt_Fire_preliminary_bis} (with the domain of integration of the integral on the right-hand side restricted to $\SigmaEsc(t)$). This finishes the proof of \cref{lem:approx_decrease_fire}.
\end{proof}
\begin{remark}
By changing the definitions of the various quantities introduced from the beginning of \cref{sec:spat_asympt}, it would be possible to prove inequality \cref{dt_fire} of \cref{lem:approx_decrease_fire} with the quantity $\nuF$ replaced by a quantity hardly less than $\eigVmin(m)$. This would require to choose accordingly the following quantities: 
\begin{itemize}
\item $\dEsc(m)$ small enough so that the quantity $\eigVmin(v)$ remains hardly less than $\eigVmin(m)$ for every $v$ such that $\abs{v-m}_{\ddd}$ is not larger than $\dEsc(m)$, 
\item and $\coeffEn$ small enough so that the first two conditions of \cref{conditions_on_nuFire} be automatically satisfied as soon as $\nuF$ is less than or equal to $\eigVmin(m)$,
\item and $\kappa$ small enough so that the factor $\kappa^2 \eigDmax/2$ of $(u^\dag)^2$ in inequality \cref{dt_fire_spat_as_before_applying_conditions_on_parameters} be much smaller than $\eigVmin(m)$. 
\end{itemize}
However, this attempt to reach the ``best possible value'' for the quantity $\nuF$ would not provide any tangible benefit in what follows.
\end{remark}
\subsubsection{Control of the distance to the minimum point}
\label{subsubsec:cont_fire}
\begin{lemma}[upper bound on the distance to $m$ with the square root of the firewall]
\label{lem:dist_to_m_square_root_firewall}
For every real quantity $x$ and every nonnegative time $t$, 
\begin{equation}
\label{dist_to_m_square_root_firewall}
\abs{u^\dag(x,t)}_{\ddd}\le \sqrt{\frac{\max\left(\frac{1+\kappa\eigDmax}{2},\frac{\eigDmax}{2}\right)}{\min\left(\frac{\coeffEn}{2},\frac{1}{4}\right)}} \sqrt{\fff(x,t)}
\,.
\end{equation}
\end{lemma}
\begin{proof}
Let $v$ be a function in $X$. Then,
\[
\begin{aligned}
\abs{v(0)}_{\ddd}^2 & = \psi(0)\abs{v(0)}_{\ddd}^2 \\
& \le \frac{1}{2} \int_{\rr} \abs{\frac{d}{dx}\bigl( \psi(x) \abs{v(x)}_{\ddd}^2 \bigr)}\, dx \\
& \le \frac{1}{2} \int_{\rr} \Bigl( \abs{\psi'(x)} \abs{v(x)}_{\ddd}^2 + 2 \psi(x)\, \abs{v(x)\cdot \ddd v'(x)} \Bigr) \, dx \\
& \le \frac{1}{2} \int_{\rr} \psi(x) \Bigl( (1+\kappa\eigDmax) v(x)^2 + \eigDmax\abs{v'(x)}_{\ddd}^2\Bigr) \, dx \\
& \le \max\left(\frac{1+\kappa\eigDmax}{2},\frac{\eigDmax}{2}\right)\int_{\rr} \psi(x) \bigl(\abs{v'(x)}_{\ddd}^2 + v(x)^2\bigr) \, dx
\,,
\end{aligned}
\]
and inequality \cref{dist_to_m_square_root_firewall} thus follows from the coercivity property \cref{coerc_fire} of $\fff$. 
\end{proof}
\begin{definition}[escape distance]
Let us call \emph{escape distance of $m$}, and let us denote by $\desc(m)$ the quantity
\begin{equation}
\label{def_r_esc_fire}
\desc(m) = \dEsc(m) \sqrt{\frac{\min\left(\frac{\coeffEn}{2},\frac{1}{4}\right)}{\max\left(\frac{1+\kappa\eigDmax}{2},\frac{\eigDmax}{2}\right)}}
\,.
\end{equation}
As for the quantity $\dEsc(m)$, this quantity $\desc(m)$ depends on $V$ and $\ddd$ and $m$ (only). 
\end{definition}
The next corollary follows from \cref{lem:dist_to_m_square_root_firewall} and from the definition \cref{def_r_esc_fire} above of $\desc(m)$.
\begin{corollary}[escape/Escape]
\label{cor:esc_Esc}
For every $\widebar{x}$ in $\rr$ and every nonnegative time $t$, the following assertion holds:
\begin{equation}
\label{ineq_esc_Esc}
\fff(\widebar{x},t) \le \desc(m)^2
\implies
\abs{u^\dag(\widebar{x},t)}_{\ddd} \le \dEsc(m)
\,.
\end{equation}
\end{corollary}
\subsubsection{Control of the solution at one end of space}
\label{subsec:speed_inv}
The next three definitions ensure the validity of \cref{lem:F_under_hullnoesc} below.
\begin{itemize}
\item Let $L$ be a positive quantity, large enough so that
\begin{equation}
\label{def_L_h_noesc}
\KF\,\frac{\exp\bigl(-\kappa\, L\bigr)}{\kappa}\le \nuF \frac{\desc(m)^2}{8}
\,;
\end{equation}
the quantity $L$ may be chosen as
\begin{equation}
\label{def_L}
L = \frac{1}{\kappa}\log\Bigl(\frac{8 \,\KF}{\nuF\, \desc(m)^2\, \kappa}\Bigr)
\,.
\end{equation}
\item Let $\hullnoesc:\rr\to\rr\cup\{+\infty\}$ (``no-escape hull'') be the function defined as
\begin{equation}
\label{def_h_noesc}
\hullnoesc(\xi) = 
\left\{
\begin{aligned}
& +\infty & \quad\text{for}\quad & \xi<0 \,, \\
& \frac{\desc(m)^2}{2}\Bigl(1-\frac{\xi}{2\, L}\Bigr) & \quad\text{for}\quad & 0\le \xi\le L \,, \\
& \frac{\desc(m)^2}{4} & \quad\text{for}\quad & \xi\ge L \,,
\end{aligned}
\right.
\end{equation}
see \cref{fig:graph_hull}.
\begin{figure}[!htbp]
\centering
\includegraphics[width=\textwidth]{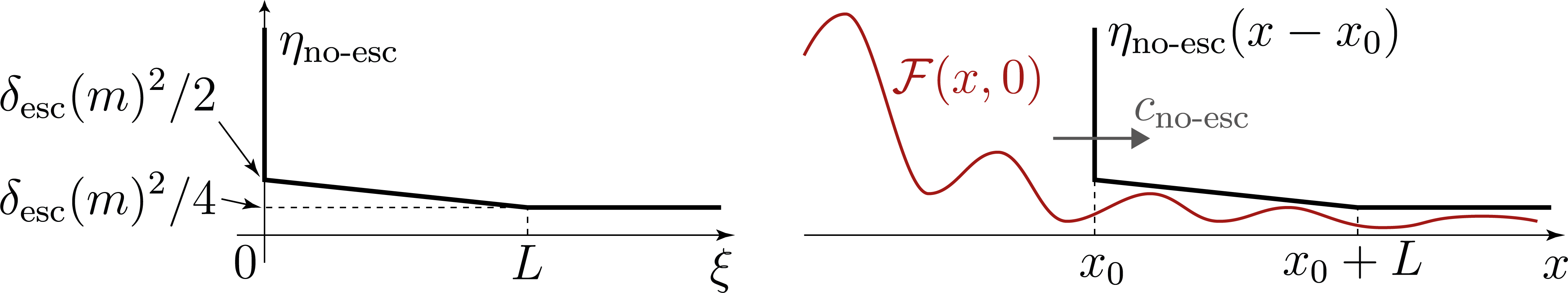}
\caption{Left: graph of the hull function $\hullnoesc$. Right: illustration of \cref{lem:F_under_hullnoesc}; if the firewall function is below a translate of the hull at time $t=0$ and if the hull travels to the right at the speed $\cnoesc$, then the firewall function will remain below the travelling hull in the future.}
\label{fig:graph_hull}
\end{figure}
\item Let $\cnoesc$ (``no-escape speed'') denote a positive quantity, large enough so that
\begin{equation}
\label{condition_on_cnoesc}
\frac{\cnoesc\,\desc(m)^2}{4\, L} \ge \frac{2\,\KF}{\kappa}
\,; 
\end{equation}
the quantity $\cnoesc$ may be chosen as
\begin{equation}
\label{def_cnoesc}
\cnoesc = \frac{8\,\KF\, L}{\kappa \, \desc(m)^2}
\,.
\end{equation}
\end{itemize}
Note that the quantities $L$ and $\cnoesc$ and the hull function $\hullnoesc$ all depend on $V$ and $\ddd$ and $m$ \emph{and} on $\norm{u_0}_X$. 

The following lemma states that if the firewall function is bounded from above by a translate of the no-escape hull at a certain time (to simplify the presentation, at the time $t=0$), then in the future it remains bounded from above by translates of this hull travelling at the no-escape speed (see \cref{fig:graph_hull}). 
\begin{lemma}[firewall remaining below travelling hull]
\label{lem:F_under_hullnoesc}
For every $x_0$ in $\rr$, if
\begin{equation}
\label{hyp_lem_F_under_hullnoesc}
\fff(x,0) \le \hullnoesc(x-x_0) 
\quad\text{for all } x \text{ in }\rr
\,,
\end{equation}
then, for every nonnegative time $t$, 
\begin{equation}
\label{F_under_hullnoesc}
\fff(x,t) \le \hullnoesc(x-x_0-\cnoesc\, t)
\quad\text{for all } x \text{ in }\rr
\,.
\end{equation}
\end{lemma}
\begin{proof}
\renewcommand{\qedsymbol}{}
Let us introduce the function
\[
\Delta: \rr\times[0,+\infty)\to\rr\cup\{+\infty\},\quad (x,t)\mapsto  \fff(x,t) - \hullnoesc(x-x_0-\cnoesc\, t)
\,
\]
and the domains
\[
\begin{aligned}
D_1 &=\bigl\{(x,t)\in\rr\times[0,+\infty): x<x_0+\cnoesc\, t\bigr\}\,, \\
\text{and}\qquad
D_2 &=\bigl\{(x,t)\in\rr\times[0,+\infty): x_0+\cnoesc\, t\le x < x_0+L+\cnoesc\, t\bigr\}\,, \\
\text{and}\qquad
\partial D_{2,3} &=\bigl\{(x,t)\in\rr\times[0,+\infty): x = x_0+L+\cnoesc\, t\bigr\}\,, \\
\text{and}\qquad
D_3 &=\bigl\{(x,t)\in\rr\times[0,+\infty): x_0+L+\cnoesc\, t<x\bigr\}
\,,
\end{aligned}
\]
see \cref{fig:areas_behav_proof_F_remains_below_etaNoesc}. 
\begin{figure}[!htbp]
\centering
\includegraphics[width=.8\textwidth]{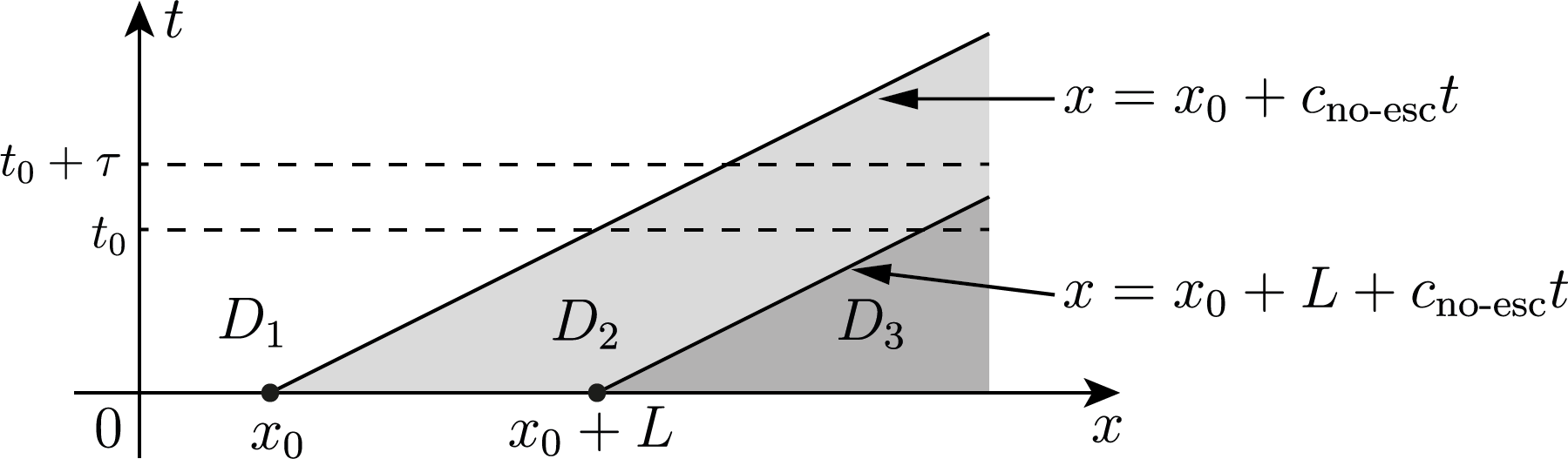}
\caption{Domains $D_1$ and $D_2$ and $D_3$.}
\label{fig:areas_behav_proof_F_remains_below_etaNoesc}
\end{figure}
Proving inequality \cref{F_under_hullnoesc} amounts to prove that $\Delta(x,t)$ is nonpositive for all $(x,t)$ in the domain
\[
\rr\times[0,+\infty) = D_1 \sqcup D_2 \sqcup \partial D_{2,3} \sqcup D_3
\,.
\]
The following observations can be made concerning the function $\Delta$:
\begin{itemize}
\item it is identically equal to $-\infty$ on $D_1$;
\item according to assumption \cref{hyp_lem_F_under_hullnoesc} it is nonpositive on $\{0\}\times\rr$;
\item it is continuous on $D_2 \sqcup \partial D_{2,3} \sqcup D_3$;
\item its partial derivative $\partial_t\Delta$ is defined on $D_2 \sqcup D_3$;
\item for every $(x,t)$ in $\rr\times[0,+\infty)$, according to inequalities \cref{coerc_fire_nonnegative,dt_fire}, 
\begin{equation}
\label{upper_bound_partial_t_F_by_quantity}
\partial_t \fff(x,t)\le \frac{2\,\KF}{\kappa}
\,,
\end{equation}
and for every $(x,t)$ in $D_2$, according to the definition \cref{def_h_noesc} of $\hullnoesc$, 
\[
\partial_t \bigl(\hullnoesc(x-x_0-\cnoesc\, t)\bigr) = \frac{\cnoesc\,\desc(m)^2}{4\,L}
\,;
\]
thus it follows from the condition \cref{condition_on_cnoesc} on $\cnoesc$ that, for every $(x,t)$ in $D_2$, 
\begin{equation}
\label{partial_t_nonpositive_in_Dtwo}
\partial_t\Delta(x,t) \le 0
\,.
\end{equation}
\item for every $(x,t)$ in $D_3$, according to inequality \cref{dt_fire}, 
\begin{equation}
\label{upper_bound_partial_t_Delta}
\partial_t\Delta(x,t)\le -\nuF\, \fff(x,t) + \KF\, \int_{\SigmaEsc(t)} T_{x}\psi(y)\, dy
\,.
\end{equation}
\end{itemize}
Let us proceed by contradiction and assume that the set
\[
\bigl\{t\in[0,+\infty): \text{ there exists $x$ in $\rr$ such that $\Delta(x,t)>0$}\bigr\} 
\]
is nonempty. The infimum of this set is a nonnegative quantity. Let us denote by $t_0$ the infimum, and let $\tau$ denote a positive quantity, small enough so that
\[
\tau\, \frac{2\,\KF}{\kappa}\le \frac{\desc(m)^2}{2}
\,;
\]
the quantity $\tau$ may be chosen as
\[
\tau = \frac{\kappa \, \desc(m)^2}{4\, \KF}
\,.
\]
The following lemma conflicts the definition of $t_0$ (and thus completes the proof).
\end{proof}
\begin{lemma}[$\Delta(\cdot,\cdot)$ cannot reach any positive value on $\rr\times{[t_0,t_0+\tau]}$] 
\label{lem:Delta_cannot_reach_positive_values_short_term}
For every $x$ in $\rr$, 
\[
\Delta(x,t_0)\le 0
\,,
\]
and for every $(x,t)$ in $\rr\times[t_0,t_0+\tau]$, 
\[
\left(\Delta(x,t)\ge -\frac{\desc(m)^2}{8} 
\quad\text{and}\quad (x,t)\not\in\partial D_{2,3}\right)
\implies
\partial_t \Delta (x,t)\le 0
\,.
\]
\end{lemma}
\begin{proof}
Since the function $\Delta$ is continuous on $D_2 \sqcup \partial D_{2,3} \sqcup D_3$, it must be nonpositive on $\rr\times[0,t_0]$ (thus in particular on $\rr\times\{t_0\}$), or else there would be a contradiction with the definition of $t_0$. This proves the first assertion. 
 
For every $x$ greater than or equal to $x_0+\cnoesc t_0$, 
\[
\Delta(x,t_0)\le 0 
\quad\text{thus}\quad
\fff(x,t_0)\le \frac{\desc(m)^2}{2}
\,,
\]
thus, according to inequality \cref{upper_bound_partial_t_F_by_quantity}, for every $t$ in $[t_0,t_0+\tau]$,
\[
\fff(x,t)\le \frac{\desc(m)^2}{2} + \tau \frac{2\,\KF}{\kappa} \le \desc(m)^2
\,,
\]
and as a consequence, according to inequality \cref{ineq_esc_Esc} of \cref{cor:esc_Esc}, 
\[
\abs{u^\dag(\widebar{x},t)}_{\ddd} \le \dEsc(m)
\,;
\]
in other words, for every $t$ in $[t_0,t_0+\tau]$, 
\begin{equation}
\label{SigmaEsc_not_larger_than_xZero_plus_cnoesc_tZero}
\SigmaEsc(t) \subset (-\infty,x_0+\cnoesc t_0]
\,.
\end{equation}
Take $(x,t)$ in $\rr\times[t_0,t_0+\tau]\setminus\partial D_{2,3}$. 
\begin{itemize}
\item If $(x,t)$ is in $D_1$ then $\Delta(x,t)$ equals $-\infty$, 
\item and if $(x,t)$ is in $D_2$, then the conclusion follows from \cref{partial_t_nonpositive_in_Dtwo}. 
\item The remaining case is when $(x,t)$ is in $D_3$, and in this case $x$ is greater than or equal to $x_0+\cnoesc t_0$. Then, according to \cref{SigmaEsc_not_larger_than_xZero_plus_cnoesc_tZero} and to inequality \cref{upper_bound_partial_t_Delta},
\[
\partial_t\Delta(x,t)\le -\nuF\, \fff(x,t) + \KF\, \frac{\exp\bigl(-\kappa\, L\bigr)}{\kappa}
\,,
\]
and as a consequence, according to the condition \cref{def_L_h_noesc} on $L$, 
\[
\partial_t\Delta(x,t)\le -\nuF\Bigl(\fff(x,t) - \frac{\desc(m)^2}{8}\Bigr) = -\nuF\Bigl(\Delta(x,t) + \frac{\desc(m)^2}{8}\Bigr)
\,,
\]
and the conclusion of the lemma follows. 
\end{itemize}
\Cref{lem:Delta_cannot_reach_positive_values_short_term} is proved.
\end{proof}
\begin{proof}[End of the proof of \cref{lem:F_under_hullnoesc}]
It follows from \cref{lem:Delta_cannot_reach_positive_values_short_term} and from the continuity of $\Delta(\cdot,\cdot)$ on $D_2 \sqcup \partial D_{2,3} \sqcup D_3$ that $\Delta(x,t)$ remains nonpositive on $\rr\times[t_0,t_0+\tau]$, a contradiction with the definition of $t_0$.
\Cref{lem:F_under_hullnoesc} is proved. 
\end{proof}
\subsubsection{Exponential decrease, first statement}
Let $c_1$ and $c_2$ denote two positive quantity with $c_1$ smaller than $c_2$. The following lemma is an intermediary result which will be called upon three times: in the proof of \cref{lem:exp_decrease_beyond_invasion_speed_uniform_version} in the next \cref{subsubsec:exp_decrease_beyond_invasion_speed_uniform_version}, in the proof of \cref{lem:sufficient_condition_stability_right_end_of_space} in \cref{subsubsec:proof_of_lem_sufficient_condition_stability_right_end_of_space}, and in the proof of \cref{lem:exponential_decrease_beyond_invasion_speed} in \cref{subsubsec:proof_of_lem_exponential_decrease_beyond_invasion_speed}. 
\begin{lemma}[exponential decrease of firewalls, first statement]
\label{lem:exponential_decrease_firewall}
Assume that there exist a real quantity $x_0$ and a nonnegative quantity $t_0$ such that, for every $t$ in $[t_0,+\infty)$, 
\begin{equation}
\label{hyp_Sigma_Esc_in_domain_growing_at_speed_cone}
\SigmaEsc(t)\subset\bigl(-\infty,x_0+c_1 (t-t_0)\bigr]
\,,
\end{equation}
and let us introduce the quantities $\nuF'$ and $\KF'$ defined as
\begin{equation}
\label{def_nu_and_K_exp_decrease_firewall}
\nuF' = \min\left(\nuF,\frac{\kappa(c_2-c_1)}{2}\right)
\quad\text{and}\quad
\KF' = \sup_{x\ge x_0}\fff(x,t_0)+\frac{2\KF}{\kappa^2(c_2-c_1)}
\,.
\end{equation}
Then, for every $t$ in $[t_0,+\infty)$, the following inequality holds:
\begin{equation}
\label{exponential_decrease_firewall}
\sup_{x\ge x_0 + c_2(t-t_0)}\fff(x,t)\le \KF' \exp\bigl(-\nuF' (t-t_0)\bigr)
\,.
\end{equation}
\end{lemma}
\begin{proof}
According to inclusion \cref{hyp_Sigma_Esc_in_domain_growing_at_speed_cone} and to inequality \cref{dt_fire} of \cref{lem:approx_decrease_fire}, for all $t$ greater than or equal to $t_0$ and for all $x$ in $\rr$, 
\[
\begin{aligned}
\partial_t \fff(x,t) + \nuF \fff(x,t) &\le \KF \int_{-\infty}^{x_0+c_1 (t-t_0)} \exp\bigl(-\kappa(x-y)\bigr)\, dy  \\
&\le \KF \exp(-\kappa x)\frac{1}{\kappa}\exp\Bigl(\kappa \bigl(x_0+c_1 (t-t_0\bigr)\Bigr)
\,,
\end{aligned}
\]
so that, if in addition $x$ is assumed to be greater than or equal to $x_0+c_2 (t-t_0)$, then 
\begin{equation}
\label{upper_bound_partial_t_F_plus_nu_F}
\partial_t \fff(x,t) + \nuF \fff(x,t) \le \frac{\KF}{\kappa} \exp\bigl(-\kappa(c_2-c_1)(t-t_0)\bigr)
\,.
\end{equation}
For every real quantity $x$ greater than or equal to $x_0$, and for every time $t$ in the interval $[t_0,t_0 + (x-x_0)/c_2]$ (see \cref{fig:time_interval_of_integration}), it thus follows from Grönwall's inequality that
\begin{figure}[!htbp]
\centering
\includegraphics[width=.85\textwidth]{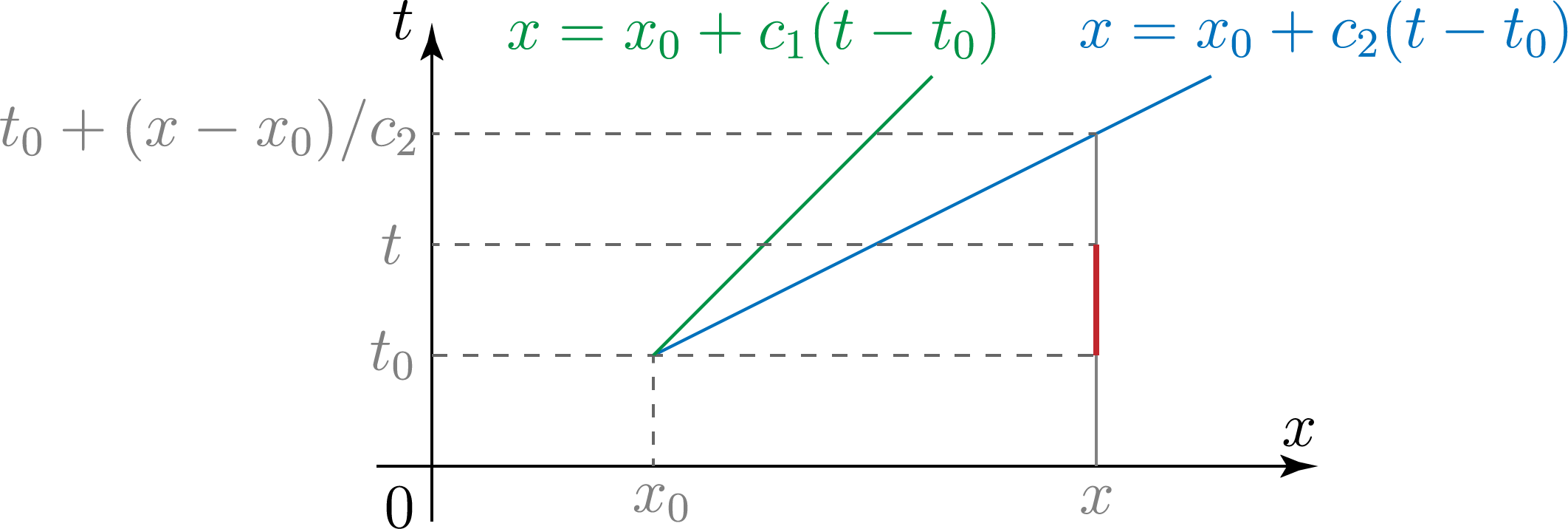}
\caption{Illustration of \cref{lem:exponential_decrease_firewall}.}
\label{fig:time_interval_of_integration}
\end{figure}
\[
\begin{aligned}
&\fff(x,t) \\
\quad&\le \fff(x,t_0)\exp\bigl(-\nuF(t-t_0)\bigr) + \frac{\KF}{\kappa}\int_{t_0}^t \exp\bigl(-\nuF(t-s)\bigr)\exp\bigl(-\kappa(c_2-c_1)(s-t_0)\bigr)\, ds \\
\quad&\le \exp\bigl(-\nuF'(t-t_0)\bigr)\left(\fff(x,t_0) + \frac{\KF}{\kappa}\int_{t_0}^t \exp\bigl(-\kappa(c_2-c_1)(s-t_0)/2\bigr)\, ds \right)
\,,
\end{aligned}
\]
and inequality \cref{exponential_decrease_firewall} follows. \Cref{lem:exponential_decrease_firewall} is proved.
\end{proof}
\subsubsection{Exponential decrease, second statement}
\label{subsubsec:exp_decrease_beyond_invasion_speed_uniform_version}
The aim of this \namecref{subsubsec:exp_decrease_beyond_invasion_speed_uniform_version} is to prove \cref{lem:exp_decrease_beyond_invasion_speed_uniform_version} below, which follows from \cref{lem:F_under_hullnoesc,lem:exponential_decrease_firewall} and concerns again the exponential decrease of the firewall functions. This second statement will be required to prove the upper semi-continuity of the asymptotic energy (\cref{sec:upp_semicont_asympt_en}), and to a lesser extent to prove \cref{lem:upper_bound_on_invasion_speed} (\cref{subsec:proof_of_preliminary_results_stab_infinity}). Stating \cref{lem:exp_decrease_beyond_invasion_speed_uniform_version} requires the following notation. 
\begin{notation}
Let us denote
\begin{itemize}
\item by $\KFatt$ the quantity $\KF$ defined in \cref{def_KF}, with $\RmaxInfty[u_0]$ replaced with $\RattInfty$, 
\item and by $\Latt$ the quantity $L$ defined in \cref{def_L}, with $\KF$ replaced with $\KFatt$, 
\item and by $\hullnoescatt$ the function $\hullnoesc$ defined in \cref{def_h_noesc}, with $L$ replaced with $\Latt$, 
\item and by $\cnoescatt$ the quantity $\cnoesc$ defined in \cref{def_cnoesc}, with $\KF$ replaced with $\KFatt$ and $L$ replaced with $\Latt$.
\end{itemize}
\end{notation}
By contrast with $\KF$ and $L$ and $\cnoesc$, these quantities depend only on $V$ and $\ddd$ and $m$, and not on the solution $u$ under consideration (indeed, by contrast with the maximal radius $\RmaxInfty[u_0]$, the radius $\RattInfty$ of the attracting ball for the $L^\infty$-norm depends only on $V$ and $\ddd$, and not on $u$, see \cref{prop:glob_exist_sol_att_ball}). 
\begin{lemma}[exponential decrease of firewalls, second statement]
\label{lem:exp_decrease_beyond_invasion_speed_uniform_version}
For every positive quantity $\delta c$, there exist positive quantities $\nuF''$ and $\KF''$, depending only on $V$ and $\ddd$ and $m$ and $\delta c$, such that, if there exists $(x_0,t_0)$ in $\rr\times\bigl[\Tatt[u_0],+\infty\bigr)$ such that 
\begin{equation}
\label{hyp_lem_exp_decrease_firwall}
\sup_{x\ge x_0} \ \fff(x,t_0) \le \frac{\desc(m)^2}{4}
\,,
\end{equation}
then, for all $t$ in $[t_0,+\infty)$, 
\begin{equation}
\label{conclusion_lem_exponential_decrease_firewall_uniform}
\sup_{x\ge x_0+(\cnoescatt + \delta c)(t-t_0)}
\fff(x,t)\le \KF''\, \exp\bigl(-\nuF'' (t-t_0)\bigr)
\,.
\end{equation}
\end{lemma}
\begin{proof}
Assume that there exist $(x_0,t_0)$ in $\rr\times\bigl[\Tatt[u_0],+\infty\bigr)$ such that assumption \cref{hyp_lem_exp_decrease_firwall} hold. Then it follows from this assumption and from the definition \cref{def_h_noesc} of $\hullnoesc$ that, for every $x$ in $\rr$, 
\[
\fff(x,t_0)\le \hullnoescatt(x-x_0)
\,.
\]
Thus, it follows from \cref{lem:F_under_hullnoesc} that, for every time $t$ greater than or equal to $t_0$ and for every real quantity $x$, 
\[
\fff(x,t) \le \hullnoescatt\bigl(x-x_0-\cnoescatt(t-t_0)\bigr)
\,,
\]
so that 
\[
\sup_{x\ge x_0+\cnoescatt(t-t_0)}\fff(x,t) \le \frac{\desc(m)^2}{2}
\,.
\]
According to \cref{cor:esc_Esc}, it follows that
\[
\SigmaEsc(t)\subset\bigl(-\infty,x_0+\cnoescatt(t-t_0)\bigr]
\,.
\]
Thus assumption \cref{hyp_Sigma_Esc_in_domain_growing_at_speed_cone} of \cref{lem:exponential_decrease_firewall} is fulfilled with $c_1$ equal to $\cnoescatt$. According to the conclusions of this lemma with $c_2$ equal to $\cnoescatt+\delta c$, introducing the quantities $\nuF''$ and $\KF''$ defined as
\begin{equation}
\label{def_nuF_second_KF_second}
\nuF'' = \min\left(\nuF,\frac{\kappa\, \delta c}{2}\right)
\quad\text{and}\quad
\KF'' = \frac{\desc(m)^2}{4}+\frac{2\KFatt}{\kappa^2\, \delta c}
\,,
\end{equation}
inequality \cref{conclusion_lem_exponential_decrease_firewall_uniform} follows from the conclusion \cref{exponential_decrease_firewall} of \cref{lem:exponential_decrease_firewall}. \Cref{lem:exp_decrease_beyond_invasion_speed_uniform_version} is proved. 
\end{proof}
\subsubsection{Sufficient condition for stability at one end of space}
\begin{lemma}[sufficient condition for small firewall at one end of space]
\label{lem:sufficient_condition_small_firewall_at_infinity}
There exist positive quantities $\deltaAsymptStab$ and $\LoneulToF$ such that the following assertion holds: for every real quantity $x_0$, if
\begin{equation}
\label{sufficient_condition_small_firewall_at_infinity}
\sup_{\widebar{x}\ge x_0} \int_{\widebar{x}}^{\widebar{x}+1} \Bigl( \bigl(u_0(x)-m\bigr)^2 + u_0'(x)^2\Bigr) \, dx \le 2\,\deltaAsymptStab^2
\,,
\end{equation}
then
\begin{equation}
\label{small_firewall_at_infinity}
\sup_{\widebar{x}\ge x_0+\LoneulToF}\fff(\widebar{x},0)\le \frac{\desc(m)^2}{4}
\,.
\end{equation}
The quantity $\deltaAsymptStab$ depends on $V$ and $\ddd$ and $m$ (only), whereas $\LoneulToF$ depends additionally on $\norm{u_0}_X$. 
\end{lemma}
\begin{remark}
The factor $2$ in the right-hand site of inequality \cref{sufficient_condition_small_firewall_at_infinity} is here to ensure that this inequality follows from assumption \cref{hyp_lem_sufficient_condition_stability_right_end_of_space} of \cref{lem:sufficient_condition_stability_right_end_of_space}. 
\end{remark}
\begin{notation}
For every $v$ in $\rr^d$, recall (see \cref{subsubsec:Escape_dist}) that $\sigma\bigl(D^2V(u)\bigr)$ denotes the spectrum (the set of eigenvalues) of the Hessian matrix of $V$ at $v$, and let $\eigVmax(v)$ denote the maximum of this spectrum:
\[
\eigVmax(v) = \max \Bigl(\sigma\bigl(D^2V(v)\bigr)\Bigr)
\,;
\]
and let 
\[
\eigVmaxBar(m) = \max_{v\in\rr^d,\,v-m\le\dEsc(m)}\eigVmax(v)
\,.
\]
Proceeding as in the proof of \cref{lem:estim_from_def_escape}, it follows that, for every $v$ in $\rr^d$ such that $\abs{v-m}$ is not larger than $\dEsc(m)$,
\begin{equation}
\label{upper_bound_V_loc}
V(v)-V(m)\le \frac{\eigVmaxBar(m)}{2} (v-m)^2
\,.
\end{equation}
To simplify the forthcoming expressions, let us introduce the two quantities
\begin{equation}
\label{def_Kone_Ktwo}
K_1 = \max\Bigl(\frac{\coeffEn\, \eigDmax}{2}, 1\Bigr)
\quad\text{and}\quad
K_2 = \max\Bigl(\frac{\coeffEn\, \eigDmax}{2}, \coeffEn\, \eigVmaxBar(m)+ \frac{1}{2}\Bigr)
\,.
\end{equation}
According to inequality \cref{Linfty_norm_bounded_from_above_by_Honeul} of \vref{cor:Linfty_norm_bounded_from_above_by_Honeul}, 
\[
\norm{u_0}_{\Linfty} \le \sqrt{2}\norm{u_0}_X
\,.
\]
Let us introduce the quantity
\begin{equation}
\label{def_Vmax}
\Vmax = \max_{v\in\rr^d,\, \abs{v}\le \sqrt{2}\norm{u_0}_X} V(v)-V(m)
\,.
\end{equation}
Let 
\begin{align}
\label{def_deltaAsymptStab}
&\deltaAsymptStab = \dEsc(m)\sqrt{\frac{1-e^{-\kappa}}{32\, K_2}}
\quad\text{(thus }\deltaAsymptStab\le\frac{\dEsc(m)}{2} \text{ ),}
\\
\text{and}\qquad
\label{def_Lprime}
&\LoneulToF = \frac{1}{\kappa}\log\Biggl(\frac{8}{\dEsc(m)^2 }\biggl(\frac{\coeffEn\, \Vmax +m^2}{\kappa} + \frac{K_1}{1-e^{-\kappa}}\norm{u_0}_X^2\biggr)\Biggr)
\,.
\end{align}
These quantities will enable us to derive the conclusions of \cref{lem:sufficient_condition_small_firewall_at_infinity}. Observe that $\deltaAsymptStab$ depends on $V$ and $\ddd$ and $m$ (only), whereas $\LoneulToF$ depends additionally on $\norm{u_0}_X$, which fits with the conclusions of \cref{lem:sufficient_condition_small_firewall_at_infinity}.
\end{notation}
\begin{proof}[Proof of \cref{lem:sufficient_condition_small_firewall_at_infinity}]
Let us assume that there exists a real quantity $x_0$ such that the assumption \cref{sufficient_condition_small_firewall_at_infinity} of \cref{lem:sufficient_condition_small_firewall_at_infinity} holds, and let $\widebar{x}$ denote a real quantity satisfying
\[
\widebar{x} \ge x_0 + \LoneulToF
\,.
\]
According to the definitions \cref{def_normalized_potential_solution} of the normalized potential and solution and \cref{def_fff} of $\fff$, 
\[
\fff(\widebar{x},0) = \int_{\rr} e^{-\kappa\abs{x-\widebar{x}}}\biggl(\coeffEn\Bigl(\frac{1}{2}\abs{u_0'(x)}_{\ddd}^2 + V\bigl(u_0(x)\bigr)-V(m)\Bigr) + \frac{1}{2}\bigl(u_0(x)-m\bigr)^2\biggr)\, dx 
\,,
\]
so that $\fff(\widebar{x},0)$ can be written as the sum of two quantities $\iiiLeft(\widebar{x})$ and $\iiiMain(\widebar{x})$ defined as
where 
\[
\begin{aligned}
\iiiLeft(\widebar{x}) &= \int_{-\infty}^{x_0} e^{-\kappa\abs{x-\widebar{x}}}\biggl(\coeffEn\Bigl(\frac{1}{2}\abs{u_0'(x)}_{\ddd}^2 + V\bigl(u_0(x)\bigr)-V(m)\Bigr) + \frac{1}{2}\bigl(u_0(x)-m\bigr)^2\biggr)\, dx \,, \\
\iiiMain(\widebar{x}) &= \int_{x_0}^{+\infty} e^{-\kappa\abs{x-\widebar{x}}}\biggl(\coeffEn\Bigl(\frac{1}{2}\abs{u_0'(x)}_{\ddd}^2 + V\bigl(u_0(x)\bigr)-V(m)\Bigr) + \frac{1}{2}\bigl(u_0(x)-m\bigr)^2\biggr)\, dx
\,.
\end{aligned}
\]
Let us consider the first integral $\iiiLeft(\widebar{x})$. According to the definition \cref{def_Vmax} of $\Vmax$, 
\[
\iiiLeft(\widebar{x}) \le \int_{-\infty}^{x_0} e^{-\kappa(\widebar{x}-x)}\biggl(\frac{\coeffEn\, \eigDmax}{2}u_0'(x)^2 + \coeffEn\, \Vmax + u_0(x)^2 + m^2\biggr)\, dx
\,,
\]
thus, according to the definition \cref{def_Kone_Ktwo} of $K_1$,
\[
\iiiLeft(\widebar{x}) \le \int_{-\infty}^{x_0} e^{-\kappa(\widebar{x}-x)}\Bigl(\coeffEn\, \Vmax + m^2 + K_1\bigl(u_0'(x)^2 + u_0(x)^2\bigr) \Bigr)\, dx
\,,
\]
 and since $\widebar{x}$ is assumed to be greater than or equal to $x_0+\LoneulToF$, 
\[
\iiiLeft(\widebar{x}) \le \exp(-\kappa \LoneulToF)\biggl(\frac{\coeffEn\, \Vmax +m^2}{\kappa} + \frac{K_1}{1-e^{-\kappa}}\norm{u_0}_X^2\biggr) 
\,,
\]
and according to the definition \cref{def_Lprime} of $\LoneulToF$, it follows that
\begin{equation}
\label{upp_bound_Ileft}
\iiiLeft(\widebar{x})\le \frac{\desc(m)^2}{8}
\,.
\end{equation}
Let us now consider the second integral $\iiiMain(\widebar{x})$. It follows from assumption \cref{sufficient_condition_small_firewall_at_infinity}, from inequality \cref{embedding_HoneLoc_into_Linfty} of \cref{lem:embedding_HoneLoc_into_Linfty}, and from the definition \cref{def_deltaAsymptStab} of $\deltaAsymptStab$ that, for every $x$ greater than or equal to $x_0$, 
\[
\abs{u_0(x)}\le 2\,\deltaAsymptStab \le \dEsc(m)
\,.
\]
As a consequence, it follows from inequality \cref{upper_bound_V_loc} that
\[
\iiiMain(\widebar{x}) \le \int_{x_0}^{+\infty} e^{-\kappa\abs{x-\widebar{x}}}\biggl(\coeffEn\Bigl(\frac{1}{2}\abs{u_0'(x)}_{\ddd}^2 + \Bigl(\coeffEn\eigVmaxBar(m)+ \frac{1}{2}\Bigr)\bigl(u_0(x)-m\bigr)^2  \biggr)\, dx
\,,
\]
thus it follows from the definition \cref{def_Kone_Ktwo} of the quantity $K_2$ that 
\[
\iiiMain(\widebar{x}) \le K_2 \int_{x_0}^{+\infty} e^{-\kappa\abs{x-\widebar{x}}} \bigl(u_0'(x)^2 + \bigl(u_0(x)-m\bigr)^2\bigr)\, dx
\,,
\]
and it follows from assumption \cref{sufficient_condition_small_firewall_at_infinity} that
\[
\iiiMain(\widebar{x}) \le \frac{4K_2}{1-e^{-\kappa}} \deltaAsymptStab^2
\,,
\]
and it finally follows from the definition \cref{def_deltaAsymptStab} of $\deltaAsymptStab$ that
\begin{equation}
\label{upp_bound_Imain}
\iiiMain(\widebar{x}) \le \frac{\desc(m)^2}{8}
\,.
\end{equation}
In view of \cref{upp_bound_Ileft,upp_bound_Imain}, \cref{lem:sufficient_condition_small_firewall_at_infinity} is proved. 
\end{proof}
\subsection{Proofs of the results of \texorpdfstring{\cref{subsubsec:preliminary_results_stability_at_ends_of_space}}{sub-subsection \ref{subsubsec:preliminary_results_stability_at_ends_of_space}}}
\label{subsec:proof_of_preliminary_results_stab_infinity}
\subsubsection{Proof of \texorpdfstring{\cref{lem:sufficient_condition_stability_right_end_of_space}}{Lemma \ref{lem:sufficient_condition_stability_right_end_of_space}}}
\label{subsubsec:proof_of_lem_sufficient_condition_stability_right_end_of_space}
\begin{proof}[Proof of \cref{lem:sufficient_condition_stability_right_end_of_space}]
It follows from \cref{lem:sufficient_condition_small_firewall_at_infinity} that, if hypothesis \cref{hyp_lem_sufficient_condition_stability_right_end_of_space} of \cref{lem:sufficient_condition_stability_right_end_of_space} holds, then there exists a real quantity $x_0$ such that
\[
\sup_{x\ge x_0}\fff(x,0) \le \frac{\desc(m)^2}{4}
\,, 
\]
so that, according to the definition \cref{def_h_noesc} of the hull function $\hullnoesc$, 
\[
\fff(x,0)\le \hullnoesc (x-x_0)
\quad\text{for all $x$ in $\rr$.}
\]
In other words, assumption \cref{hyp_lem_F_under_hullnoesc} of \cref{lem:F_under_hullnoesc} holds. According to the conclusion \cref{F_under_hullnoesc} of this lemma, for every nonnegative time $t$
\[
\fff(x,t)\le \hullnoesc(x-x_0-\cnoesc\, t)
\quad\text{for all $x$ in $\rr$,}
\]
so that 
\[
\sup_{x\ge x_0 + \cnoesc\, t} \fff(x,t)\le \frac{\desc(m)^2}{2}
\,.
\]
According to \cref{cor:esc_Esc}, it follows that
\[
\SigmaEsc(t)\subset\bigl(-\infty,x_0+\cnoesc (t-t_0)\bigr]
\,.
\]
Thus assumption \cref{hyp_Sigma_Esc_in_domain_growing_at_speed_cone} of \cref{lem:exponential_decrease_firewall} holds with $c_1$ equal to $\cnoesc$ and $t_0$ equal to $0$. It follows from the conclusion \cref{exponential_decrease_firewall} of this lemma that
\[
\sup_{x\ge x_0 + (\cnoesc+1) t} \fff(x,t)\to0 
\quad\text{as}\quad
t\to+\infty
\,.
\]
According to the coercivity \cref{coerc_fire_hone} of $\fff(x,t)$ and to inequality \cref{Linfty_norm_bounded_from_above_by_Honeul} of \vref{cor:Linfty_norm_bounded_from_above_by_Honeul}, the solution $u$ is stable close to $m$ at the right hand of space (\cref{def:solution_stable_at_one_end_of_space}). \Cref{lem:sufficient_condition_stability_right_end_of_space} is proved.
\end{proof}
\subsubsection{Proof of \texorpdfstring{\cref{lem:upper_bound_on_invasion_speed}}{Lemma \ref{lem:upper_bound_on_invasion_speed}}}
\label{subsubsec:proof_of_lem_upper_bound_on_invasion_speed}
\begin{proof}[Proof of \cref{lem:upper_bound_on_invasion_speed}]
If the solution $u$ is stable close to $m$ at the right end of space, then it follows from the upper bound \cref{bound_u_ut_ck} on $u_{xx}$ that there exists a positive time $t_0$ (arbitrarily large) and a real quantity $x_0$ such that 
\[
\sup_{\widebar{x}\ge x_0} \int_{\widebar{x}}^{\widebar{x}+1} \bigl( u^\dag(x,t_0)^2 + u_x^\dag(x,t_0)^2\bigr) \, dx \le 2\,\deltaAsymptStab^2
\,,
\]
so that, according to \cref{lem:sufficient_condition_small_firewall_at_infinity}, 
\[
\sup_{\widebar{x}\ge x_0+\LoneulToF}\fff(\widebar{x},t_0)\le \frac{\desc(m)^2}{4}
\,.
\]
Since $t_0$ can be chosen arbitrarily large, let us assume that $t_0$ is greater than or equal to $\Tatt[u_0]$. Then it follows from \cref{lem:exp_decrease_beyond_invasion_speed_uniform_version} that, for every positive quantity $\delta c$ and for every time $t$ greater than or equal to $t_0$, 
\[
\sup_{x\ge x_0+\LoneulToF+(\cnoescatt+\delta c)(t-t_0)}\fff(x,t)\le \KF'' \exp\bigl(-\nuF''\, (t-t_0)\bigr)
\,,
\]
where the quantities $\nuF''$ and $\KF''$ are defined as in \cref{def_nuF_second_KF_second}. Thus it follows from the coercivity \cref{coerc_fire_hone} of $\fff(x,t)$ that
\[
\sup_{\widebar{x}\ge x_0+\LoneulToF+(\cnoescatt+\delta c)(t-t_0)}\int_{\widebar{x}}^{\widebar{x}+1} \bigl(u^\dag(x,t)^2 + u_x^\dag(x,t)^2\bigr) \, dx \to0
\quad\text{as}\quad
t\to+\infty
\,,
\]
and thus it follows from inequality \cref{Linfty_norm_bounded_from_above_by_Honeul} of \vref{cor:Linfty_norm_bounded_from_above_by_Honeul} that 
\[
\sup_{\widebar{x}\ge x_0+\LoneulToF+(\cnoescatt+\delta c)(t-t_0)}\abs{u(x,t)} \to0
\quad\text{as}\quad
t\to+\infty
\,.
\]
It follows that
\begin{equation}
\label{cInvPlus_not_larger_than_cnoescatt}
\cInvPlus[u]\le \cnoescatt
\,,
\end{equation}
and according to its definition \cref{subsubsec:exp_decrease_beyond_invasion_speed_uniform_version}, the quantity $\cnoescatt$ depends only on $V$ and $\ddd$ and $m$ (but not on $u$). 
\Cref{lem:upper_bound_on_invasion_speed} is proved. 
\end{proof}
\subsubsection{Proof of \texorpdfstring{\cref{lem:exponential_decrease_beyond_invasion_speed}}{Lemma \ref{lem:exponential_decrease_beyond_invasion_speed}}}
\label{subsubsec:proof_of_lem_exponential_decrease_beyond_invasion_speed}
\begin{proof}[Proof of \cref{lem:exponential_decrease_beyond_invasion_speed}]
Let us assume that the solution $u$ is stable close to $m$ at the right end of space, and let $c$ denote a (positive) quantity larger than the invasion speed $\cInvPlus[u]$. Let us introduce the quantity
\[
c' = \frac{1}{2}\bigl(\cInvPlus[u]+c\bigr)\,,
\quad\text{so that}\quad
\cInvPlus[u] < c' < c
\,.
\]
According to \cref{def:invasion_speed} of the invasion speed, 
\[
\sup_{x\ge c't}\abs{u(x,t)}\to 0 
\quad\text{as}\quad
t\to+\infty
\,,
\]
so that, according to the the upper bound \cref{bound_u_ut_ck} on $u_{xx}$
\[
\sup_{\widebar{x}\ge c't}\int_{\widebar{x}}^{\widebar{x}+1} \bigl(u^\dag(x,t)^2 + u_x^\dag(x,t)^2\bigr) \, dx\to 0 
\quad\text{as}\quad
t\to+\infty
\,.
\]
It follows that there exists a positive time $t_1$ such that, for every time $t$ greater than or equal to $t_1$, 
\[
\sup_{\widebar{x}\ge c't}\int_{\widebar{x}}^{\widebar{x}+1} \bigl(u^\dag(x,t)^2 + u_x^\dag(x,t)^2\bigr) \, dx \le 2 \deltaAsymptStab^2
\,,
\]
so that, according to \cref{lem:sufficient_condition_small_firewall_at_infinity}, 
\[
\sup_{x\ge c't + \LoneulToF}\fff(x,0)\le \frac{\desc(m)^2}{4}
\,.
\]
Since $c'$ is smaller than $c$, there exists a time $t_0$ greater than or equal to $t_1$ such that
\begin{equation}
\label{c_prime_t_zero_plus_L_one_smaller_than_c_t_zero}
c't_0 + \LoneulToF \le ct_0
\,.
\end{equation}
Thus, as a consequence of \cref{lem:exponential_decrease_firewall} for $t_0$ and $x_0$ equal to $c't_0+\LoneulToF$ and $c_1$ equal to $c'$ and $c_2$ equal to $c$, it follows that, for every time $t$ greater than or equal to $t_0$,
\[
\sup_{x\ge c't_0+\LoneulToF + c(t-t_0)}\fff(x,t)\le \KF' \exp\bigl(-\nuF' (t-t_0)\bigr) 
\,,
\]
with quantities $\nuF'$ and $\KF'$ given by expressions \cref{def_nu_and_K_exp_decrease_firewall} with $c_1 = c'$ and $c_2 = c$. Thus it follows from \cref{c_prime_t_zero_plus_L_one_smaller_than_c_t_zero} that, 
\begin{equation}
\label{upper_bound_fff_proof_lemma_exponential_decrease_beyond_invasion_speed}
\sup_{x\ge ct}\fff(x,t)\le \KF' \exp\bigl(-\nuF' (t-t_0)\bigr) = \KF' \exp(\nuF' t_0)\exp(-\nuF' t)
\,.
\end{equation}
Besides, for every $\widebar{x}$ in $\rr$,
\[
\begin{aligned}
\int_{\widebar{x}}^{\widebar{x}+1}\bigl( u^\dag(x,t)^2 + u_x^\dag(x,t)^2\bigr) \, dx &\le \exp(\kappa)\int_{\widebar{x}}^{\widebar{x}+1}T_{\widebar{x}}\psi(x)\bigl( u^\dag(x,t)^2 + u_x^\dag(x,t)^2\bigr) \, dx \\
&\le\exp(\kappa)\int_{\rr}T_{\widebar{x}}\psi(x)\bigl( u^\dag(x,t)^2 + u_x^\dag(x,t)^2\bigr) \, dx
\,,
\end{aligned}
\]
and thus, according to inequality \cref{coerc_fire_hone} ensuring the coercivity of $\fff(x,t)$,
\[
\int_{\widebar{x}}^{\widebar{x}+1}\bigl( u^\dag(x,t)^2 + u_x^\dag(x,t)^2\bigr) \, dx \le \frac{\exp(\kappa)}{\min\Bigl(\frac{\coeffEn\eigDmin}{2},\frac{1}{4}\Bigr)}\fff(\widebar{x},t)
\,.
\]
Thus, introducing the quantities
\[
\Koneul[u] = \frac{\KF'\exp(\nuF' t_0 + \kappa)}{\min\Bigl(\frac{\coeffEn\eigDmin}{2},\frac{1}{4}\Bigr)}
\quad\text{and}\quad
K_\infty[u] = \sqrt{2\Koneul[u]}
\quad\text{and}\quad
\nu_\infty = \frac{\nuF'}{2}
\,,
\]
it follows from \cref{upper_bound_fff_proof_lemma_exponential_decrease_beyond_invasion_speed} that
\begin{equation}
\label{exp_decrease_honeul_norm_beyond_invasion_speed}
\sup_{\widebar{x}\in[ct,+\infty)}\int_{\widebar{x}}^{\widebar{x}+1}\bigl( u^\dag(x,t)^2 + u_x^\dag(x,t)^2\bigr) \, dx \le \Koneul[u] \exp(-\nuF' t)
\,,
\end{equation}
and, according to inequality \cref{Linfty_norm_bounded_from_above_by_Honeul} of \vref{cor:Linfty_norm_bounded_from_above_by_Honeul},
\[
\sup_{x\in[ct,+\infty)}\abs{u(x,t)-m}\le K_\infty[u]\exp(-\nu_\infty t)
\,.
\]
In other words, the conclusion \cref{exponential_decrease_beyond_invasion_speed} of \cref{lem:exponential_decrease_beyond_invasion_speed} holds for $t$ greater than or equal to $t_0$ (with the quantities $K_\infty[u]$ and $\nu_\infty$). According to the upper bound \cref{maximal_radius_excursion_Linfty} on $\abs{u(x,t)}$, up to increasing the quantity $K_\infty[u]$, the same conclusion holds for all $t$ in $[0,+\infty)$. 
\Cref{lem:exponential_decrease_beyond_invasion_speed} is proved. 
\end{proof}
\begin{proof}[Proof of \cref{cor:bist}]
According to the conclusions of \cref{lem:exponential_decrease_beyond_invasion_speed,lem:upper_bound_on_invasion_speed}, a small enough $\HoneulAlone$-perturbation of a bistable solution connecting two points $m_-$ and $m_+$ of $\mmm$ still satisfies, for large positive times, the condition \cref{hyp_lem_sufficient_condition_stability_right_end_of_space} of \cref{lem:sufficient_condition_stability_right_end_of_space} (both for $m_-$ at the left end of space and for $m_+$ at the right end of space) It follows that this perturbation is still a bistable solution connecting $m_-$ to $m_+$. The set of such bistable solutions is thus open in $X$. The fact that this set is nonempty follows from \cref{lem:sufficient_condition_stability_right_end_of_space}. 
\end{proof}
\section{Asymptotic energy}
\label{sec:asympt_en}
The aim of this \namecref{sec:asympt_en} is to prove \vref{prop:asympt_en} about the the asymptotic energy of a bistable solution. 
\subsection{Set-up}
%
As everywhere else, let us consider a function $V$ in $\ccc^2(\rr^d,\rr)$ satisfying hypothesis \cref{hyp_coerc}. Let $(m_-,m_+)$ denote an ordered pair of points of $\mmm$ in the same level set of $V$, let us write 
\[
\valueOfV=V(m_-)=V(m_+)
\,,
\]
let $u_0$ in $X$ be a bistable initial condition connecting $m_-$ to $m_+$, and let $(x,t)\mapsto u(x,t)$ denote the solution of system~\cref{init_syst} corresponding to this initial condition. Let us introduce the ``normalized'' potential $V^\ddag:\rr^d\to\rr$, $v\mapsto V^\ddag(v)$, defined as
\begin{equation}
\label{def_V_ddag_asymptotic_energy}
V^\ddag(v) = V(v) - \valueOfV = V(v) - V(m_\pm)
\,.
\end{equation}
The function $(x,t)\mapsto u(x,t)$ is still a solution of system \cref{init_syst} with $V^\ddag$ instead of $V$. 
\subsection{Localized energy}
\label{subsec:localized_energy}
\subsubsection{Definition}
There are several ways to define the localized energy of the solution. The advantages of the following definition are:
\begin{itemize}
\item it leads to natural estimates in terms of the firewall functionals defined in the previous \namecref{sec:spat_asympt}, 
\item it does not rely on the regularizing properties of system~\cref{init_syst} --- it is thus easier to extend to other classes of systems like the damped hyperbolic system~\cref{hyp_syst},
\item it provides the same explicit estimates as those that will be used for the proof of the upper semi-continuity of the asymptotic energy in \cref{sec:upp_semicont_asympt_en}. 
\end{itemize}
Let us denote:
\begin{itemize}
\item by $\kappa_-$ (by $\kappa_+$) the quantity defined in \cref{def_kappa} and denoted by $\kappa$ in \cref{sec:spat_asympt}, for the minimum point $m_-$ (for the minimum point $m_+$);
\item and by $\cnoescattMinus$ (by $\cnoescattPlus$) the quantity defined and denoted by $\cnoescatt$ in \cref{subsubsec:exp_decrease_beyond_invasion_speed_uniform_version}, for the minimum point $m_-$ (for the minimum point $m_+$);
\end{itemize}
and let
\begin{equation}
\label{def_kappa_cUpp_asymptotic_energy}
\kappa = \min(\kappa_-,\kappa_+)
\quad\text{and}\quad
\cUpp = \max(\cnoescattMinus,\cnoescattPlus)+1
\,,
\end{equation}
so that, according to inequality \cref{cInvPlus_not_larger_than_cnoescatt}, 
\begin{equation}
\label{cUpp_larger_than_both_invasion_speeds}
\cUpp > \max(\cInvMinus[u],\cInvPlus[u])
\,.
\end{equation}
For every time $t$, let us introduce the three intervals:
\[
\begin{aligned}
\iLeft(t) &= (-\infty, -\cUpp\, t] \,, \\
\text{and}\qquad
\iMain(t) &= [-\cUpp\,  t , \cUpp\,  t] \,, \\
\text{and}\qquad
\iRight(t) &= [ \cUpp\,  t , +\infty) \,.
\end{aligned}
\]
Let us introduce the weight function $\chi$ defined as
\begin{equation}
\label{def_weight_fct_en}
\chi(x,t) = 
\left\{
\begin{aligned}
&\exp\bigl(-\kappa(\cUpp\,  t-x)\bigr) = T_{-\cUpp\,  t}\ \psi(x) 
&& \quad\text{if}\quad x\in\iLeft(t) \,, \\
&1 
&& \quad\text{if}\quad x\in\iMain(t) \,, \\
&\exp\bigl(-\kappa(x-\cUpp\,  t)\bigr) = T_{\cUpp\,  t}\ \psi(x) 
&& \quad\text{if}\quad x\in\iRight(t)
\,,
\end{aligned}
\right.
\end{equation}
see \cref{fig:weight_en}.
\begin{figure}[!htbp]
\centering
\includegraphics[width=\textwidth]{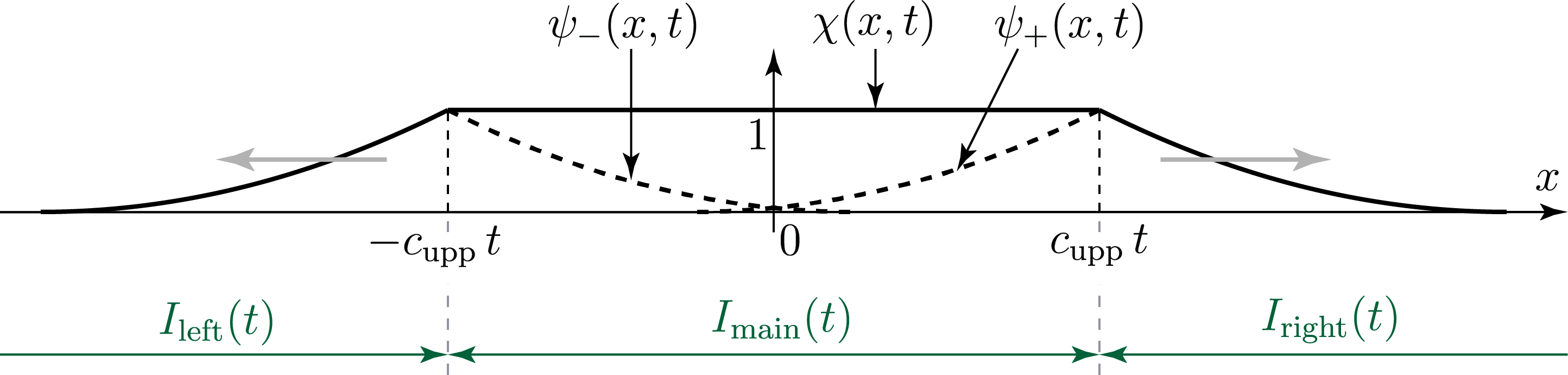}
\caption{Graph of the weight function $x\mapsto\chi(x,t)$ defining the localized energy $\eee(t)$.}
\label{fig:weight_en}
\end{figure}
For $x$ in $\rr$ and $t$ in $[0,+\infty)$, let
\begin{equation}
\label{def_Eddag}
E^\ddag(x,t) = \frac{1}{2}\abs{u_x(x,t)}_{\ddd} ^2 + V^\ddag\bigl(u(x,t)\bigr)
\,,
\end{equation}
and for $t$ in $[0,+\infty)$, let us define the ``localized energy'' $\eee(t)$ by
\begin{equation}
\label{def_energy}
\eee(t) = \int_{\rr} \chi(x,t)\,  E^\ddag(x,t) \, dx 
\,.
\end{equation}
\subsubsection{Time derivative}
For all nonnegative time $t$, let us introduce the following quantity (``localized dissipation''):
\[
\Delta(t) = \int_{\rr} \chi(x,t)\, u_t(x,t)^2 \, dx
\,.
\]
\begin{lemma}[time derivative of localized energy]
\label{lem:dt_en_loc_preliminary}
For every nonnegative time $t$,
\begin{equation}
\label{dt_en_loc_preliminary}
\eee'(t) \le -\frac{1}{2} \Delta(t) + \kappa\int_{\iLeft(t)\sqcup\iRight(t)} \chi\biggl[\frac{\cUpp+\kappa\, \eigDmax}{2}\abs{u_x}_{\ddd}^2 + \cUpp\, V^\ddag(u)\biggr]\, dx
\,.
\end{equation}
\end{lemma}
\begin{proof}
It follows from expression \vref{ddt_loc_en} for the derivative of a localized energy that
\[
\eee'(t) = - \Delta (t) + \int_{\rr} \biggl[\chi_t \, \Bigl( \frac{1}{2}\abs{u_x}_{\ddd} ^2 + V^\ddag(u) \Bigr) - \chi_x \, \ddd u_x\cdot u_t \biggr]\, dx
\,.
\]
It follows from the definition of $\chi$ that
\[
\chi_t(x,t)=
\left\{
\begin{aligned}
&0  
&&\text{if}\quad x\in\iMain(t) \,, \\
& \kappa\, \cUpp\,   \chi(x,t)
&&\text{if}\quad x\not\in\iMain(t) \,, \\
\end{aligned}
\right.
\]
and 
\[
\chi_x(x,t)=
\left\{
\begin{aligned}
&0  
&&\text{if}\quad x\in\iMain(t) \,, \\
& -\sgn(x)\kappa\, \chi(x,t)
&&\text{if}\quad x\not\in\iMain(t) \,. \\
\end{aligned}
\right.
\]
Thus it follows from these expressions that
\[
\eee'(t) \le - \Delta (t)+ \kappa \int_{\iLeft(t)\sqcup\iRight(t)} \chi\biggl[ \cUpp \Bigl( \frac{1}{2}\abs{u_x}_{\ddd} ^2 + V^\ddag(u) \Bigr) + \abs{\ddd u_x\cdot u_t} \biggr]\, dx
\,.
\]
Thus, using the inequality (compare with inequality \vref{polarize_kappa_ddd_ux_ut})
\[
\kappa\abs{\ddd u_x\cdot u_t} \le \frac{1}{2}u_t^2 + \frac{\kappa^2\eigDmax}{2} \abs{u_x}_\ddd^2
\,,
\]
 inequality \cref{dt_en_loc_preliminary} follows. \Cref{lem:dt_en_loc_preliminary} is proved. 
\end{proof}
\subsubsection{Firewalls}
For $x$ in $\rr$ and $t$ in $[0,+\infty)$, let
\[
F^\ddag_\pm(x,t) = \coeffEn E^\ddag(x,t) + \frac{1}{2}(u\bigl(x,t)-m_\pm\bigr)^2
\,,
\]
and for $\widebar{x}$ in $\rr$ and $t$ in $[0,+\infty)$, let us define the ``firewall'' functions $\fff_-(\widebar{x},t)$ and $\fff_+(\widebar{x},t)$ by
\begin{equation}
\label{def_fire_of_bar_x_and_t}
\fff_\pm(\widebar{x},t) = \int_{\rr} T_{\bar x}\psi(x)\, F^\ddag_\pm(x,t)\, dx 
\,,
\end{equation}
where $x\mapsto\psi(x)$ is the weight function defined in \vref{def_psi}. These two functions $\fff_\pm(\widebar{x},t)$ are to $m_-$ and $m_+$ what the firewall function $\fff(\widebar{x},t)$ of \cref{subsubsec:firewall_definition} was to the minimum point $0_{\rr^d}$ of $V^\dag$. To simplify the notation of the next calculations, let us introduce the more specific weight functions $(x,t)\mapsto\psi_\pm(x,t)$ defined as
\[
\begin{aligned}
\psi_-(x,t) &= T_{-\cUpp\,   t}\ \psi (x,t) = \exp\bigl(-\kappa\abs{x+\cUpp\,   t}\bigr) \,, \\
\text{and}\qquad
\psi_+(x,t) &= T_{\cUpp\,   t}\ \psi (x,t) = \exp\bigl(-\kappa\abs{x-\cUpp\,   t}\bigr) \,,
\end{aligned}
\]
see \cref{fig:weight_en}, and the more specific firewall functions $\fff_-$ and $\fff_+$ defined as
\begin{equation}
\label{def_fire}
\fff_\pm(t) = \fff_\pm(\pm \cUpp\,  t,t) = \int_{\rr} \psi_\pm(x,t) \, F^\ddag_\pm(x,t) \, dx 
\,.
\end{equation}
\subsubsection{Energy decrease up to firewalls}
\begin{lemma}[energy decrease up to firewalls]
\label{lem:energy_decrease_up_to_firewall}
There exists a positive quantity $\KEF$, depending on $V$ and $\ddd$ and $m_-$ and $m_+$, such that for every nonnegative time $t$, 
\begin{equation}
\label{equ_energy_decrease_up_to_firewall}
\eee'(t) \le -\frac{1}{2} \Delta(t) + \KEF \bigl(\fff_-(t) + \fff_+(t)\bigr)
\,.
\end{equation}
\end{lemma}
\begin{proof}
It follows from inequality \cref{dt_en_loc_preliminary} that, for every nonnegative time $t$ (observe the substitution of $\chi$ by $\psi_-$ and $\psi_+$), 
\[
\begin{aligned}
\eee'(t) \le -\frac{1}{2} \Delta(t) 
& + \kappa\int_{\iLeft(t)} \psi_- \biggl[\frac{\cUpp+\kappa\, \eigDmax}{2}\abs{u_x}_{\ddd}^2 + \cUpp\,   V^\ddag(u)\biggr]\, dx \\
& + \kappa\int_{\iRight(t)} \psi_+ \biggl[\frac{\cUpp+\kappa\, \eigDmax}{2}\abs{u_x}_{\ddd}^2 + \cUpp\,   V^\ddag(u)\biggr]\, dx
\,.
\end{aligned}
\]
Thus,
\[
\begin{aligned}
&\eee'(t) \le -\frac{1}{2} \Delta(t) + \\
& \kappa\int_{\iLeft(t)} \psi_- \biggl[\frac{\cUpp+\kappa\, \eigDmax}{2}\abs{u_x}_{\ddd}^2 + \cUpp\Bigl( V^\ddag(u) + \frac{1}{2\coeffEn} (u-m_-)^2\Bigr) \biggr]\, dx \\
& + \kappa\int_{\iRight(t)} \psi_+ \biggl[\frac{\cUpp+\kappa\, \eigDmax}{2}\abs{u_x}_{\ddd}^2 + \cUpp\Bigl( V^\ddag(u) + \frac{1}{2\coeffEn} (u-m_+)^2\Bigr) \biggr]\, dx
\,.
\end{aligned}
\]
According to inequality \vref{def_weight_en} the quantities
\[
 V^\ddag(u) + \frac{1}{2\coeffEn} (u-m_\pm)^2 = \frac{1}{\coeffEn}\Bigl(\coeffEn\, V^\ddag(u) + \frac{1}{2} (u-m_\pm)^2\Bigr)
\]
are nonnegative. As a consequence the previous inequality still holds if the factor $\cUpp$ of these quantities is replaced by the larger factor $\cUpp+\kappa\, \eigDmax$ and if the domains of integration of the two integrals are extended to the whole real line. After these changes the inequality reads
\[
\eee'(t) \le -\frac{1}{2} \Delta(t) + \frac{\kappa(\cUpp+\kappa\, \eigDmax)}{\coeffEn}\bigl(\fff_-(t) + \fff_+(t)\bigr)
\,.
\]
Thus, if $\KEF$ denotes the quantity $\kappa(\cUpp+\kappa\, \eigDmax)/\coeffEn$, then inequality \cref{equ_energy_decrease_up_to_firewall} and \cref{lem:energy_decrease_up_to_firewall} are proved. 
\end{proof}
\subsubsection{Energy decrease up to pollution}
\label{subsubsec:localized_energy_decrease}
\begin{lemma}[energy decrease up to pollution]
\label{lem:loc_en_almost_decr}
There exists positive quantities $\KE$ and $\nuE$ and a nonnegative time $t_0$ such that, for every time $t$ greater than or equal to $t_0$, 
\begin{equation}
\label{dt_en_loc_real}
\eee'(t) \le -\frac{1}{2} \Delta(t) + \KE \, \exp\bigl(-\nuE(t-t_0)\bigr)
\,.
\end{equation}
The quantities $\KE$ and $\nuE$ depend on $V$ and $\ddd$ and $m_-$ and $m_+$ (only), whereas the time $t_0$ depends additionally on the solution under consideration. 
\end{lemma}
\begin{proof}
Since $u$ is a bistable solution connecting $m_-$ to $m_+$, and according to inequality \cref{cUpp_larger_than_both_invasion_speeds} satisfied by the quantity $\cUpp$, it follows from \cref{lem:sufficient_condition_small_firewall_at_infinity} that there exists a time $t_0$ in $\bigl[\Tatt[u_0],+\infty\bigr)$ such that
\[
\sup_{x\in\iLeft(t_0)} \fff_-(x,t_0) \le \frac{\desc(m)^2}{4}
\quad\text{and}\quad
\sup_{x\in\iRight(t_0)} \fff_+(x,t_0) \le \frac{\desc(m)^2}{4}
\,.
\]
Then, according to \cref{lem:exp_decrease_beyond_invasion_speed_uniform_version}, there exist positive quantities $\nuUnifMinus$ and $\nuUnifPlus$ and $\KunifMinus$ and $\KunifPlus$ (defined as in \cref{def_nuF_second_KF_second} with $\delta c$ equal to $1$) such that, for every time $t$ greater than or equal to $t_0$, 
\begin{equation}
\label{exp_decrease_fff_of_x_and_t_plus_minus}
\begin{aligned}
\sup_{x\in\iLeft(t)}\fff_-(x,t) &\le \KunifMinus \exp\bigl(-\nuUnifMinus(t-t_0)\bigr) \,, \\
\text{and}\qquad
\sup_{x\in\iRight(t)}\fff_+(x,t) &\le \KunifPlus \exp\bigl(-\nuUnifPlus(t-t_0)\bigr)
\,,
\end{aligned}
\end{equation}
thus, in particular, 
\begin{equation}
\label{exp_decrease_fff_plus_minus}
\fff_\pm(t)\le \KunifPlusMinus \exp\bigl(-\nuUnifPlusMinus(t-t_0)\bigr)
\,.
\end{equation}
As a consequence, introducing the positive quantities
\[
\nuE = \min(\nuUnifMinus,\nuUnifPlus)
\quad\text{and}\quad
\KE = \KEF (\KunifMinus+\KunifPlus)
\,,
\]
inequality \cref{dt_en_loc_real} follows from inequality \cref{equ_energy_decrease_up_to_firewall} of \cref{lem:energy_decrease_up_to_firewall}. \Cref{lem:loc_en_almost_decr} is proved.
\end{proof}
\subsection{Asymptotic energy}
\label{subsubsec:asympt_en_incomplete}
\begin{corollary}[asymptotic energy]
\label{cor:asympt_en_as_limit_of_eee}
There exists a quantity $\eeeAsympt[u]$ (``asymptotic energy'' of the solution) in $\{-\infty\}\cup\rr$ such that
\begin{equation}
\label{asympt_en_as_limit_of_eee}
\eee(t) \to \eeeAsympt[u] \quad\text{as}\quad t\to+\infty
\,.
\end{equation}
\end{corollary}
\begin{proof}
Since the dissipation $\Delta(t)$ is nonnegative, the limit \cref{asympt_en_as_limit_of_eee} follows from inequality \cref{dt_en_loc_real} of \cref{lem:loc_en_almost_decr}. \Cref{cor:asympt_en_as_limit_of_eee} is proved. 
\end{proof}
\begin{lemma}[convergence towards asymptotic energy for various speeds of bounds of spatial domain]
\label{lem:convergence_to_asympt_enery_for_various_speeds_of_bounds}
For all positive quantities $c_-$ and $c_+$ satisfying 
\begin{equation}
\label{hyp_convergence_to_asympt_enery_for_various_speeds_of_bounds}
\cInvMinus[u] < c_-
\quad\text{and}\quad
\cInvPlus[u] < c_+
\,,
\end{equation}
the following limit holds:
\begin{equation}
\label{convergence_to_asympt_enery_for_various_speeds_of_bounds}
\int_{- c_- t}^{c_+ t}  E^\ddag(x,t) \, dx  
\to \eeeAsympt[u]
\quad\text{as}\quad t\to+\infty
\,.
\end{equation}
\end{lemma}
\begin{proof}
Let $c_-$ and $c_+$ be two positive quantities satisfying inequalities \cref{hyp_convergence_to_asympt_enery_for_various_speeds_of_bounds}. 
For every nonnegative time $t$, let us introduce the quantity $\delta\eee(t)$ defined as
\[
\begin{aligned}
\delta\eee(t) &= \eee(t) - \int_{- c_- t}^{c_+ t}  E^\ddag(x,t)\, dx \\
&= \int_{\iLeft(t)\sqcup\iRight(t)} \chi(x,t) E^\ddag(x,t)\, dx + \int_{\iMain(t)}E^\ddag(x,t)\, dx - \int_{- c_- t}^{c_+ t}  E^\ddag(x,t)\, dx
\,.
\end{aligned}
\]
All what remains to be proved is that this quantity $\delta\eee(t)$ goes to $0$ as $t$ goes to $+\infty$. For every nonnegative time $t$, let us introduce the two intervals
\[
\begin{aligned}
\jLeft(t) &= [-\max(c_-,\cUpp)t,-\min(c_-,\cUpp)t] \,,\\
\text{and}\qquad
\jRight(t) &= [\min(c_+,\cUpp)t,\max(c_+,\cUpp)t]
\,,
\end{aligned}
\]
and the integrals
\[
\iii(t) = \int_{\iLeft(t)\sqcup\iRight(t)} \chi(x,t) \abs{E^\ddag(x,t)}\, dx
\quad\text{and}\quad
\jjj(t) = \int_{\jLeft(t)\sqcup\jRight(t)}\abs{E^\ddag(x,t)}\, dx
\,.
\]
According to this notation,
\begin{equation}
\label{eee_of_t_smaller_than_iii_plus_jjj}
\abs{\delta\eee(t)}\le \iii(t) + \jjj(t)
\,.
\end{equation}
It follows from inequalities \cref{cUpp_larger_than_both_invasion_speeds,hyp_convergence_to_asympt_enery_for_various_speeds_of_bounds} that, if $t$ is large enough, 
\[
\begin{aligned}
\sup_{x\in\iLeft(t)\cup\jLeft(t)}\abs{u(x,t)-m_-}_{\ddd} &\le \dEsc(m_-) \,,\\
\text{and}\quad
\sup_{x\in\iRight(t)\cup\jRight(t)}\abs{u(x,t)-m_+}_{\ddd} &\le \dEsc(m_+)
\,,
\end{aligned}
\]
so that, according to inequality \cref{nonnegative_pot_around_loc_min} of \vref{lem:estim_from_def_escape}, 
\[
E^\ddag(x,t)\ge V^\ddag\bigl(u(x,t)\bigr) \ge 0
\quad\text{for all $x$ in $\bigl(\iLeft(t)\cup\jLeft(t)\bigr)\sqcup\bigl(\iRight(t)\cup\jRight(t)\bigr)$.}
\]
It follows that
\[
\iii(t) = \int_{\iLeft(t)\sqcup\iRight(t)} \chi(x,t) \, E^\ddag(x,t)\, dx
\quad\text{and}\quad
\jjj(t) = \int_{\jLeft(t)\sqcup\jRight(t)}E^\ddag(x,t)\, dx
\,.
\]
As a consequence, on the one hand,
\[
\begin{aligned}
\iii(t)=& \int_{\iLeft(t)} \psi_-(x,t)\, E^\ddag(x,t)\, dx + \int_{\iRight(t)} \psi_+(x,t)\, E^\ddag(x,t)\, dx \\
\le& \frac{1}{\coeffEn}\biggl(\int_{\iLeft(t)} \psi_-(x,t) \Bigl(\coeffEn E^\ddag(x,t) + \frac{1}{2}\bigl(u(x,t)-m_-\bigr)^2\Bigr)\, dx \\
&+ \int_{\iRight(t)} \psi_+(x,t) \Bigl(\coeffEn E^\ddag(x,t) + \frac{1}{2}\bigl(u(x,t)-m_+\bigr)^2\Bigr)\, dx\biggr)
\,.
\end{aligned}
\]
According to inequality \vref{def_weight_en}, the integrands of the two integrals of the right-hand side of this last inequality are nonnegative for every real quantity $x$. As a consequence, the inequality still holds if both domains of integration are extended to the whole real line. It follows that, still if $t$ is large enough, 
\[
\iii(t) \le \frac{1}{\coeffEn}\bigl(\fff_-(t)+\fff_+(t)\bigr)
\,,
\]
thus it follows from inequality \cref{exp_decrease_fff_plus_minus} that $\iii(t)$ goes to $0$ as $t$ goes to $+\infty$. On the other hand, it follows from inequality \vref{upper_bound_V_loc} that 
\[
\begin{aligned}
\int_{\jLeft(t)} E^\ddag(x,t)\, dx &\le \int_{\jLeft(t)} \Bigl(\frac{\eigDmax}{2}u_x(x,t)^2 + \eigVmaxBar(m_-)\bigl(u(x,t)-m_-\bigr)^2\Bigr) \, dx \,, \\ 
\text{and}\quad
\int_{\jRight(t)} E^\ddag(x,t)\, dx &\le \int_{\jRight(t)} \Bigl(\frac{\eigDmax}{2}u_x(x,t)^2 + \eigVmaxBar(m_+)\bigl(u(x,t)-m_+\bigr)^2\Bigr) \, dx
\,,
\end{aligned}
\]
and according to inequality \vref{exp_decrease_honeul_norm_beyond_invasion_speed} the right hand sides of these two inequalities go to $0$ as $t$ goes to $+\infty$, so that $\jjj(t)$ goes to $0$ as $t$ goes to $+\infty$. In view of inequality \cref{eee_of_t_smaller_than_iii_plus_jjj}, \cref{lem:convergence_to_asympt_enery_for_various_speeds_of_bounds} is proved. 
\end{proof}
\Cref{prop:asympt_en} follows from \cref{cor:asympt_en_as_limit_of_eee,lem:convergence_to_asympt_enery_for_various_speeds_of_bounds}. In \cref{subsec:app_zero_Ham_all_times} it will be proved that, provided that the additional hypothesis \textup{(\hyperlink{hypOnlyMin}{\hypOnlyMinRef{\valueOfV}})} holds, the asymptotic energy $\eeeAsympt[u]$ is actually either equal to $-\infty$ or nonnegative, as stated in conclusion \cref{item:thm_main_nonnegative_asympt_energy_approach_bistable_stationary} of \cref{thm:main}. 
\section{Upper semi-continuity of the asymptotic energy}
\label{sec:upp_semicont_asympt_en}
The aim of this \namecref{sec:upp_semicont_asympt_en} is to prove \cref{prop:scs_asympt_en} about the upper semi-continuity of the asymptotic energy with respect to bistable initial conditions. 

As everywhere else, let us consider a function $V$ in $\ccc^2(\rr^d,\rr)$ satisfying hypothesis \cref{hyp_coerc}. Let $(m_-,m_+)$ denote an ordered pair of points of $\mmm$ in the same level set of $V$, let 
$(\uzeron)_{n\in\nn}$ denote a sequence of functions in $\Xbist(m_-,m_+)$ (bistable initial conditions connecting $m_-$ to $m_+$), and let $\uzeroinfty$ denote a function in 
$\Xbist(m_-,m_+)$, such that
\[
\norm{\uzeron-\uzeroinfty}_{X} \to 0
\quad\text{as}\quad
n\to +\infty
\,.
\]
Our aim is to prove that
\begin{equation}
\label{eeeAsympt_of_u_zero_infty_larger_than_limsup_of_eeeAsympt_of_u_zero_n}
\eeeAsympt[\uzeroinfty]\ge\limsup_{n\to+\infty}\eeeAsympt[\uzeron]
\,.
\end{equation}
For every $n$ in $\nn\cup\{\infty\}$ and for all $x$ in $\rr$ and $t$ in $[0,+\infty)$, let $u_n(\cdot,\cdot)$ denote the solution of system~\cref{init_syst} for the initial condition $\uzeron$. Let us consider the same weight function $(x,t)\mapsto\chi(x,t)$ as the one defined in~\vref{def_weight_fct_en}, and, for every $n$ in $\nn\cup\{\infty\}$ and $t$ in $[0,+\infty)$, let us consider the quantities 
\[
\eee_n(t)
\quad\text{and}\quad
\fff_{-,n}(x,t)
\quad\text{and}\quad
\fff_{+,n}(x,t)
\quad\text{and}\quad
\fff_{-,n}(t)
\quad\text{and}\quad
\fff_{+,n}(t)
\,,
\]
defined exactly as in~\vref{def_energy,def_fire_of_bar_x_and_t,def_fire} for the solution $u_n$. 
\begin{lemma}[uniform bound on the derivative of localized energies]
\label{lem:up_bd_dt_en_p}
There exists a nonnegative time $t_0$ and a nonnegative integer $n_0$ such that, for every integer $n$ greater than $n_0$ and every time $t$ greater than $t_0$, the following inequality holds:
\begin{equation}
\label{up_bd_dt_en_p}
\eee_n'(t) \le \KE \exp\bigl(-\nuE(t-T)\bigr)
\,.
\end{equation}
\end{lemma}
\begin{proof}
Inequality \cref{up_bd_dt_en_p} will follow from inequality \vref{dt_en_loc_real} (the sole additional requirement is some uniformity with respect to $n$). 
Let $\Rinit$ denote the supremum of the set
\[
\Bigl\{\,\norm{\uzeron}_{X} : n\in\nn\cup\{\infty\} \, \Bigr\}
\]
(this quantity is finite). 
According to \vref{prop:glob_exist_sol_att_ball}, there exists a quantity $\Tatt$ (depending on $V$ and $\ddd $ and $\Rinit$, but not on $n$), such that, for every time $t$ greater than $\Tatt$ and every $n$ in $\nn\cup\{\infty\}$, 
\[
\sup_{x\in\rr}\abs{u_n(x,t)}\le \RattInfty
\,.
\]
It follows from the same arguments as in the proof of \vref{lem:sufficient_condition_small_firewall_at_infinity} that there exists a nonnegative time $t_0$ such that
\[
\sup_{x\in\iLeft(t_0)}\fff_{-,\infty}(x,t_0)\le\frac{\desc(m)^2}{8}
\quad\text{and}\quad
\sup_{x\in\iRight(t_0)}\fff_{+,\infty}(x,t_0)\le\frac{\desc(m)^2}{8}
\,;
\]
in addition, the time $t_0$ may be chosen greater than or equal to $\Tatt$. Then, by continuity of the semi-flow $(S_t)_{t\ge0}$ of system \cref{init_syst} with respect to initial conditions in $X$, there exists a nonnegative integer $n_0$ such that, for every integer $n$ greater than $n_0$, 
\[
\sup_{x\in\iLeft(t_0)}\fff_{-,n}(x,t_0)\le \frac{\desc(m)^2}{4}
\quad\text{and}\quad
\sup_{x\in\iRight(t_0)}\fff_{+,n}(x,t_0)\le \frac{\desc(m)^2}{4}
\,,
\]
and it follows from \vref{lem:exp_decrease_beyond_invasion_speed_uniform_version} that there exist positive quantities $\nuUnifMinus$ and $\nuUnifPlus$ and $\KunifMinus$ and $\KunifPlus$ (defined as in \cref{def_nuF_second_KF_second} with $\delta c$ equal to $1$) such that, for every integer $n$ greater than $n_0$ and every time $t$ greater than $t_0$,
\[
\begin{aligned}
\fff_{-,n}(t)&\le \KunifMinus\exp\bigl(-\nuUnifMinus(t-t_0)\bigr) \,,\\
\text{and}\qquad
\fff_{+,n}(t)&\le \KunifPlus\exp\bigl(-\nuUnifPlus(t-t_0)\bigr)
\end{aligned}
\,.
\]
Thus, introducing the same quantities 
\[
\nuE = \min(\nuUnifMinus,\nuUnifPlus)
\quad\text{and}\quad
\KE = \KEF(\KunifMinus+\KunifPlus)
\]
as in the proof of \vref{lem:loc_en_almost_decr}, inequality \cref{up_bd_dt_en_p} follows from inequality \cref{dt_en_loc_real} of \cref{lem:loc_en_almost_decr}. \Cref{lem:up_bd_dt_en_p} is proved. 
\end{proof}
\begin{proof}[Proof of \cref{prop:scs_asympt_en}]
Since $\eee_n(t)$ goes to $\eeeAsympt[\uzeron]$ as time goes to $+\infty$, it follows from inequality \cref{up_bd_dt_en_p} of \cref{lem:up_bd_dt_en_p} that, still for every integer $n$ greater than $n_0$ and every time $t$ greater than $t_0$,
\[
\eee_n(t) \ge \eeeAsympt[\uzeron] - \frac{\KE}{\nuE}\exp\bigl(-\nuE(t-t_0)\bigr)
\,.
\]
Passing to the limit as $n$ goes to $+\infty$, it follows from the continuity of the semi-flow $(S_t)_{t\ge0}$ of system \cref{init_syst} with respect to initial conditions in $X$ that, for every time $t$ greater than $t_0$, 
\[
\eee_\infty(t) \ge \limsup_{n\to+\infty}\eeeAsympt[\uzeron]-\frac{\KE}{\nuE}\exp\bigl(-\nuE(t-t_0)\bigr) \,. 
\]
Finally, passing to the limit as time goes to $+\infty$, inequality \cref{eeeAsympt_of_u_zero_infty_larger_than_limsup_of_eeeAsympt_of_u_zero_n} follows. \Cref{prop:scs_asympt_en} is proved. 
\end{proof}
\section{Finite asymptotic energy implies relaxation}
\label{sec:relaxation}
The aim of this \namecref{sec:relaxation} is to prove conclusions \cref{item:thm_main_time_derivative_goes_to_zero,item:thm_main_invasion_speeds_vanish} of \cref{thm:main}. 

As everywhere else, let us consider a function $V$ in $\ccc^2(\rr^d,\rr)$ satisfying hypothesis \cref{hyp_coerc}. As in \cref{sec:asympt_en}, let $(m_-,m_+)$ denote an ordered pair of points of $\mmm$ in the same level set of $V$, and let $(x,t)\mapsto u(x,t)$ denote a bistable solution connecting $m_-$ to $m_+$ for system~\cref{init_syst}. Additionally, let us assume that the asymptotic energy $\eeeAsympt[u]$ of the solution is finite: 
\begin{equation}
\label{assumption_finite_asymptotic_energy}
\eeeAsympt[u] > -\infty
\,.
\end{equation}
\subsection{Asymptotically vanishing time derivative}
The following lemma completes the proof of conclusion \cref{item:thm_main_time_derivative_goes_to_zero} of \cref{thm:main}. 
\begin{lemma}[time derivative goes to zero]
\label{lem:ut_small}
The following limit holds:
\begin{equation}
\label{ut_small}
\sup_{x\in\rr}\ \abs{u_t(x,t)}\to0
\quad\text{as}\quad
t\to+\infty
\,.
\end{equation}
\end{lemma}
\begin{proof}
According to inequality \cref{cUpp_larger_than_both_invasion_speeds} satisfied by the quantity $\cUpp$, both quantities
\[
\sup_{x\le-\cUpp\, t}\abs{u(x,t)-m_-}
\quad\text{and}\quad
\sup_{x\ge\cUpp\, t}\abs{u(x,t)-m_+}
\]
go to $0$ as time goes to $+\infty$. Thus, according to the bounds \cref{bound_u_ut_ck} on the solution, 
\begin{equation}
\label{ut_goes_to_zero_beyond_cCut_t}
\sup_{\abs{x}\ge \cUpp\, \, t}\abs{u_t(x,t)}\to 0
\quad\text{as}\quad
t\to+\infty
\,.
\end{equation}
Let us proceed by contradiction and assume that the limit \cref{ut_small} does not hold. Then, there exists a positive quantity $\varepsilon$ and a sequence $(x_n,t_n)_{n\in\nn}$ in $\rr\times[0,+\infty)$ such that $t_n\to+\infty$ as $n\to+\infty$ and such that, for every $n$ in $\nn$, 
\begin{equation}
\label{contradiction_hyp_u_t_not_small}
\abs{u_t(x_n,t_n)}\ge \varepsilon
\,.
\end{equation}
According \cref{ut_goes_to_zero_beyond_cCut_t}, it may be assumed (up to dropping the first terms of the sequence $(x_n,t_n)_{n\in\nn}$) that, for every $n$ in $\nn$, $x_n$ belongs to the interval $[-\cUpp t_n,\cUpp t_n]$. According to \cref{lem:compactness}, there exists an entire solution $\widebar{u}$ of system \cref{init_syst} such that, up to replacing the sequence $(x_n,t_n)_{n\in\nn}$ by a subsequence, with the notation of \cref{compactness},
\begin{equation}
\label{convergence_up_to_subsequence}
D^{2,1}u(x_n+\cdot,t_n+\cdot)\to D^{2,1}\widebar{u}
\quad\text{as}\quad
n\to+\infty
\,,
\end{equation}
uniformly on every compact subset of $\rr^2$. It follows from \cref{contradiction_hyp_u_t_not_small,convergence_up_to_subsequence} that the quantity $\abs{\widebar{u}_t(0,0)}$ is positive, so that the quantity
\[
\int_0^1 \left(\int_{\rr}e^{-\kappa\abs{\xi}}\abs{\widebar{u}_t(\xi,s)}^2 \, d\xi\right)\, ds
\] 
is also positive. This quantity is less than or equal to the quantity
\[
\liminf_{n\to+\infty}\int_0^1 \Delta(t_n+s) \, ds
\,,
\]
which is therefore also positive. On the other hand, according to the approximate decrease of energy \cref{dt_en_loc_real} and to assumption \cref{assumption_finite_asymptotic_energy}, the nonnegative function $t\mapsto \Delta(t)$ is integrable on $[0,+\infty)$, a contradiction. \Cref{lem:ut_small} is proved.  
\end{proof}
\subsection{Invasion speeds vanish}
\label{subsec:invasion_speeds_vanish}
For every nonnegative time $t$, let us introduce the quantity $\xEscPlus(t)$ in $\rr\cup\{-\infty,+\infty\}$ (``Escape point to the right''), defined as the supremum of the set 
\[
\bigl\{x\in\rr: \abs{u(x,t)-m_+}_{\ddd} =\dEsc(m_+)\bigr\}
\,,
\]
with the convention that this supremum equals $-\infty$ if this set is empty. It follows from inequality \cref{cUpp_larger_than_both_invasion_speeds} satisfied by the quantity $\cUpp$ that, for every large enough positive time $t$, 
\begin{equation}
\label{alternative_xEsc}
\text{either}\quad 
\xEscPlus(t) = -\infty
\quad\text{or}\quad 
-\cUpp\, t < \xEscPlus(t) < \cUpp\, t 
\,.
\end{equation}
\begin{lemma}[transversality at Escape point]
\label{lem:transv}
There exist positive quantities $\epsEscTransv$ and $\tEscTransv$ such that, for every $t$ in $[\tEscTransv,+\infty)$, if $\xEscPlus(t)$ is finite, then 
\[
\Bigl\langle u\bigl(\xEscPlus(t),t\bigr)-m_+ \,, \,u_x\bigl(\xEscPlus(t),t\bigr)\Bigr\rangle_{\ddd} \le -\epsEscTransv \,.
\] 
\end{lemma}
\begin{proof}
Let us proceed by contradiction and assume that there exists a sequence $(t_n)_{n\in\nn}$ such that $t_n$ goes to $+\infty$ as $n$ goes to $+\infty$ and such that, for every nonnegative integer $n$, 
\[
\begin{aligned}
&-\infty < \xEscPlus(t_n)<+\infty \\
\text{and}\quad 
&\Bigl\langle u\bigl(\xEscPlus(t_n),t_n\bigr)-m_+ \,, \,u_x\bigl(\xEscPlus(t_n),t_n\bigr)\Bigr\rangle_{\ddd} \ge -1/n 
\,. 
\end{aligned}
\]
Up to extracting a subsequence, it may be assumed, according to \cref{lem:compactness,lem:ut_small}, that the functions 
\[
\xi\mapsto u\bigl(\xEscPlus(t_n)+\xi,t_n\bigr)
\]
converge, uniformly on every compact subset of $\rr$, towards a stationary solution $\xi\mapsto u_\infty(\xi)$ of system~\cref{init_syst} satisfying 
\[
\bigl\langle u_\infty(0)-m_+,u_\infty'(0)\bigr\rangle_{\ddd} \ge0
\quad\text{and, for all $\xi$ in $[0,+\infty)$,}\quad
\abs{u_\infty(\xi)-m_+}_{\ddd} \le\dEsc(m_+)
\]
(the second property follows from the definition of $\xEscPlus(t)$). This is contradictory to assertion \cref{item:transv_spatial_asymptotics_sw} of \vref{lem:ham_equ_hyp}. \Cref{lem:transv} is proved.
\end{proof}
Up to increasing the quantity $\tEscTransv$, it may be assumed that assertion \cref{alternative_xEsc} holds for every time $t$ greater than or equal to $\tEscTransv$. 
\begin{corollary}[finiteness/infiniteness of $\xEscPlus(\cdot)$ dichotomy]
\label{cor:finitenes_infiniteness_of_xEscPlus_dichotomy}
One of the following two (mutually exclusive) alternatives occurs: 
\begin{enumerate}
\item for every time $t$ greater than or equal to $\tEscTransv$, the quantity $\xEscPlus(t)$ equals $-\infty$,
\label{item:alternative_xEscPlus_equals_minus_infinity_for_t_large}
\item (or) for every time $t$ greater than or equal to $\tEscTransv$, the quantity $\xEscPlus(t)$ is finite.
\label{item:alternative_xEscPlus_is_finite_for_t_large}
\end{enumerate}
\end{corollary}
\begin{proof}
Let us introduce the function
\[
f:\rr\times [0,+\infty) \to \rr, \quad
(x,t)\mapsto \frac{1}{2} \bigl(\abs{u(x,t)-m_+ }_{\ddd} ^2-\dEsc(m_+)^2\bigr)
\,.
\]
According to the smoothness properties of the solution recalled in \cref{subsec:glob_exist}, this function $f$ is of class $\ccc^1$ on $\rr\times(0,+\infty)$. 
For all $t$ in $[0,+\infty)$, if $\xEscPlus(t)$ is finite then $f\bigl(\xEscPlus(t),t\bigr)$ vanishes. If in addition $t$ is greater than or equal to the (positive) quantity $\tEscTransv$ defined in \cref{lem:transv}, then 
\begin{equation}
\label{dx_f_non_0}
\partial_x f\bigl(\xEscPlus(t),t\bigr) = \Bigl\langle u\bigl(\xEscPlus(t),t\bigr)-m_+ \,, \,u_x\bigl(\xEscPlus(t),t\bigr)\Bigr\rangle_{\ddd} \le -\epsEscTransv <0
\,.
\end{equation}
Let us introduce the set
\[
\ttt = \bigl\{t\in[\tEscTransv,+\infty): \xEscPlus(t)>-\infty\bigr\}
\,.
\]
It follows from inequality \cref{dx_f_non_0} and from the Implicit Function Theorem that this set is open in $[\tEscTransv,+\infty)$; and it follows from the definition of $\xEscPlus(t)$ and from assertion \cref{alternative_xEsc} (which was assumed to hold for every time $t$ greater than or equal to $\tEscTransv$) that this set is closed in $[\tEscTransv,+\infty)$. As a consequence, this set $\ttt$ is either empty or equal to $[\tEscTransv,+\infty)$, and \cref{cor:finitenes_infiniteness_of_xEscPlus_dichotomy} is proved. 
\end{proof}
\begin{lemma}[approach to a homogeneous equilibrium]
\label{lem:cv_homog_case}
Assume that alternative \cref{item:alternative_xEscPlus_equals_minus_infinity_for_t_large} of \cref{cor:finitenes_infiniteness_of_xEscPlus_dichotomy} occurs (that is, $\xEscPlus(t)$ equals $-\infty$ for every time $t$ greater than or equal to $\tEscTransv$). Then the minimum points $m_-$ and $m_+$ must be equal, and 
\[
\sup_{x\in\rr}\abs{u(x,t)-m_{\pm}} \to 0
\quad\text{as}\quad 
t\to +\infty
\,.
\]
\end{lemma}
\begin{proof}
The fact that $m_-$ equals $m_+$ follows from the definition of $\dEsc(m_\pm)$. The uniform convergence towards $m_+$ (equal to $m_-$) may again be obtained either by contradiction and a compactness argument, or by observing that, according to inequality \vref{dt_fire} and for every $x$ in $\rr$, the quantity $\fff(x,t)$ (which is nonnegative according to \cref{coerc_fire_nonnegative}) goes to $0$, at an exponential rate, when $t$ goes to $+\infty$. 
\end{proof}
\begin{lemma}[asymptotically vanishing time derivative of $\xEscPlus(t)$]
\label{lem:asymptotically_vanishing_time_derivative_xEscPlus}
Assume that alternative \cref{item:alternative_xEscPlus_is_finite_for_t_large} of \cref{cor:finitenes_infiniteness_of_xEscPlus_dichotomy} occurs (that is, $\xEscPlus(t)$ is finite for every time $t$ greater than or equal to $\tEscTransv$). Then, the function $t\mapsto \xEscPlus(t)$ is of class $\ccc^1$ on the interval $[\tEscTransv,+\infty)$ and
\[
\xEscPlus'(t)\to 0
\quad\text{as}\quad
t\to+\infty
\,.
\]
\end{lemma}
\begin{proof}
It follows from the Implicit Function Theorem applied to the function $f$ introduced in the proof of \cref{cor:finitenes_infiniteness_of_xEscPlus_dichotomy} that, if alternative \cref{item:alternative_xEscPlus_is_finite_for_t_large} occurs, then the function $t\mapsto \xEscPlus(t)$ is of class $\ccc^1$ on $[\tEscTransv,+\infty)$. For every time $t$ in this interval, the quantity $\xEscPlus'(t)$ reads:
\[
\xEscPlus'(t) = 
-\frac
{\partial_t f\bigl(\xEscPlus(t),t\bigr)}
{\partial_x f\bigl(\xEscPlus(t),t\bigr)}
= 
-\frac
{\Bigl\langle u\bigl(\xEscPlus(t) ,t\bigr)-m_+  \,, \,u_t\bigl(\xEscPlus(t), t\bigr)\Bigr\rangle_{\ddd} }
{\Bigl\langle u\bigl(\xEscPlus(t) ,t\bigr)-m_+  \,, \,u_x\bigl(\xEscPlus(t),t \bigr)\Bigr\rangle_{\ddd} }
\,.
\]
According to~\cref{lem:ut_small}, the numerator of this expression goes to $0$ as time goes to $+\infty$, while according to inequality \cref{dx_f_non_0} the absolute value of its denominator remains not smaller than $\epsEscTransv$, and the conclusion follows. 
\Cref{lem:asymptotically_vanishing_time_derivative_xEscPlus} is proved.
\end{proof}
\begin{proof}[Proof of conclusion \cref{item:thm_main_invasion_speeds_vanish} of \cref{thm:main}]
Conclusion \cref{item:thm_main_invasion_speeds_vanish} of \cref{thm:main} states that both invasion speeds $\cInvMinus[u]$ and $\cInvPlus[u]$ vanish. If alternative \cref{item:alternative_xEscPlus_equals_minus_infinity_for_t_large} of \cref{cor:finitenes_infiniteness_of_xEscPlus_dichotomy} occurs, then this statement follows from \cref{lem:cv_homog_case}. If alternative \cref{item:alternative_xEscPlus_is_finite_for_t_large} of \cref{cor:finitenes_infiniteness_of_xEscPlus_dichotomy} occurs, then according to \cref{lem:asymptotically_vanishing_time_derivative_xEscPlus} the quantity $\xEscPlus'(t)$ goes to $0$ as time goes to $+\infty$, and it follows from \cref{lem:exponential_decrease_firewall}, from the coercivity \cref{coerc_fire_hone} of $\fff(x,t)$ and from inequality \cref{Linfty_norm_bounded_from_above_by_Honeul} of \vref{cor:Linfty_norm_bounded_from_above_by_Honeul} that the invasion speed to the right $\cInvPlus[u]$ vanishes. The same arguments lead to the same conclusion for the invasion speed to the left $\cInvMinus[u]$. Conclusion \cref{item:thm_main_invasion_speeds_vanish} of \cref{thm:main} is proved. 
\end{proof}
\section{Approach to a set of bistable stationary solutions}
\label{sec:approach_set_bist_sol}
\subsection{Set-up}
The aim of this \namecref{sec:approach_set_bist_sol} is to prove conclusion \cref{item:thm_main_nonnegative_asympt_energy_approach_bistable_stationary} of \cref{thm:main}. Let us keep the assumptions and notation of the previous \namecref{sec:relaxation}, let 
\[
\valueOfV = V(m_-) = V(m_+)
\,,
\]
and, in addition, let us assume that hypothesis \textup{(\hyperlink{hypOnlyMin}{\hypOnlyMinRef{\valueOfV}})} holds. As in definition \vref{def_V_ddag_asymptotic_energy}, let us introduce the ``normalized'' potential $V^\ddag$ defined as 
\begin{equation}
\label{def_V_ddag_relaxation}
V^\ddag (v) = V(v)-\valueOfV = V(v)-V(m_\pm)
\,.
\end{equation}
Our task is to prove that
\begin{equation}
\sup_{x\in\rr} \, \dist\Bigl(\bigl(u(x,t),u_x(x,t)\bigr) \,, \, I\bigl(\PhiZero(\valueOfV)\bigr)\Bigr)\to 0
\quad\text{as}\quad
t\to+\infty
\,.
\end{equation}
Let $c$ denote a positive quantity (which may very well be chosen equal to $1$). According to conclusion \cref{item:thm_main_invasion_speeds_vanish} of \cref{thm:main} and to \cref{lem:exponential_decrease_beyond_invasion_speed}, both quantities
\begin{equation}
\label{u_close_min}
\sup_{x\le -c t}\abs{u(x,t)-m_-}  
\quad\text{and}\quad
\sup_{x\ge  ct}\abs{u(x,t)-m_+}
\end{equation}
go to $0$, at an exponential rate, as time goes to $+\infty$, and, according to the bounds~\cref{bound_u_ut_ck} of the solution, the same is true for the quantity
\begin{equation}
\label{ux_small}
\sup_{\abs{x}\ge ct}\abs{u_x(x,t)}  
\,.
\end{equation}
\subsection{Approach to normalized Hamiltonian level set zero for a sequence of times}
Recall the notation $H$ (already defined in \cref{subsubsec:not_ham}) to denote the Hamiltonian associated to the differential system of stationary solutions of system~\cref{init_syst}:
\[
H:\rr^d\times\rr^d\to\rr,\quad
(u,v)\mapsto \frac{1}{2}\abs{v}_{\ddd}^2 - V(u)
\,.
\]
and let us introduce the ``normalized'' Hamiltonian (with respect to the level $\valueOfV$):
\begin{equation}
\label{def_normalized_hamiltonian}
H^\ddag:\rr^d\times\rr^d\to\rr,\quad
(u,v)\mapsto \frac{1}{2}\abs{v}_{\ddd}^2 - V^\ddag(u) = H(u,v) + V(m_\pm) = H(u,v) + \valueOfV 
\,.
\end{equation}
\begin{lemma}[approach to normalized Hamiltonian level set zero for a sequence of times]
\label{lem:lim_inf_sup_H}
The following equality holds:
\begin{equation}
\label{lim_inf_sup_H}
\liminf_{t\to+\infty}\, \sup_{x\in\rr} \, \abs{ H^\ddag\bigl(u(x,t),u_x(x,t)\bigr)} =0 
\,. 
\end{equation}
\end{lemma}
\begin{proof}
Let us proceed by contradiction and assume that the converse is true. Then there exists a positive quantity $\delta$ such that, for every large enough positive time $t$, 
\begin{equation}
\label{hyp_H_not_small}
\sup_{x\in\rr}\abs{H^\ddag\bigl(u(x,t),u_x(x,t)\bigr)}\ge\delta 
\,. 
\end{equation}
Observe that, for all $x$ in $\rr$ and $t$ in $[0,+\infty)$, the ``space derivative of the normalized Hamiltonian'' along a solution has the following simple expression:
\[
\partial_x \Bigl( H^\ddag\bigl(u(x,t),u_x(x,t)\bigr) \Bigr) = u_x\cdot\ u_t 
\,. 
\]
Thus, in view of assertions \cref{u_close_min} and~\cref{ux_small} about the behaviour of the solution outside of the interval $[-ct,ct]$, it follows from hypothesis~\cref{hyp_H_not_small} that
\[
\liminf_{t\to+\infty}\int_{-ct}^{ct}\abs{u_x(x,t)\cdot\  u_t(x,t)}\, dx\ge 2\delta
\,.
\]
Thus it follows from the bound~\cref{bound_u_ut_ck} on $\abs{u_x}$ that the limit
\[
\liminf_{t\to+\infty} \int_{-ct}^{ct} \abs{u_t(x,t)}\, dx
\]
is positive; and thus it follows from Cauchy--Schwarz inequality that the limit
\[
\liminf_{t\to+\infty} \, t \, \int_{-ct}^{ct} u_t^2(x,t)\, dx
\]
is positive. As a consequence the same is true for the limit
\[
\liminf_{t\to+\infty} \  t \, \Delta(t)
\,,
\]
a contradiction with the fact that the function $t\mapsto \Delta(t)$ is integrable on $[1,+\infty)$. 
\end{proof}
\subsection{Approach to normalized Hamiltonian level set zero for all times}
\label{subsec:app_zero_Ham_all_times}
The aim of this \namecref{subsec:app_zero_Ham_all_times} is to prove that the limit~\cref{lim_inf_sup_H} of \cref{lem:lim_inf_sup_H} holds for all time going to infinity, and not only for a subsequence of times (in other words that the $\liminf$ in \cref{lim_inf_sup_H} can be substituted by a ``full'' limit). The proof involves the asymptotic compactness of solutions (\cref{lem:compactness}) and the results of \cref{subsec:infin_lag} about the value of the Lagrangian of stationary solutions, and is based on the next two lemmas.

As in definition \vref{def_Lagrangian_intro}, let us introduce the (pointwise) Lagrangian associated to system~\cref{init_syst}:
\[
L:\rr^d\times\rr^d\to\rr,\quad
(u,v)\mapsto \frac{1}{2}\abs{v}_{\ddd}^2 + V(u)
\,,
\]
and its normalized declination (with respect to the level $\valueOfV$):
\begin{equation}
\label{def_normalized_lagrangian}
L^\ddag:\rr^d\times\rr^d\to\rr,\quad
(u,v)\mapsto \frac{1}{2}\abs{v}_{\ddd}^2 + V^\ddag(u) = L(u,v) - \valueOfV = L(u,v)-V(m_\pm)
\,.
\end{equation}
The positive quantity $\deltaHam$ defined in \vref{subsec:infin_lag} will also be used.
\begin{lemma}[small normalized Hamiltonian forces positive normalized Lagrangian]
\label{lem:contrib_Lag}
There exists a positive quantity $T$ (depending on the solution $(x,t)\mapsto u(x,t)$ under consideration) such that, for every time $t$ greater than $T$ and every real quantity $\widebar{x}$,
\[
\abs{H^\ddag\bigl(u(\widebar{x},t),u_x(\widebar{x},t)\bigr)}\le\deltaHam 
\implies
\int_{\widebar{x}}^{\widebar{x}+1} L^\ddag\bigl(u(x,t),u_x(x,t)\bigr) \, dx \ge 0
\,.
\]
\end{lemma}
\begin{proof}
Let us proceed by contradiction and assume that the converse is true. Then there exists a sequence $(x_n,t_n)_{n\in\nn}$ in $\rr\times[0,+\infty)$ such that
$t_n$ goes to $+\infty$ as $n$ goes to $+\infty$ and such that, for every nonnegative integer $n$, 
\begin{equation}
\label{hyp_contrib_neg}
\abs{H^\ddag\bigl(u(x_n,t_n),u_x(x_n,t_n)\bigr)}\le\deltaHam 
\quad\text{and}\quad
\int_{x_n}^{x_n+1} L^\ddag\bigl(u(x,t),u_x(x,t)\bigr) \, dx < 0
\,.
\end{equation}
According to \cref{lem:compactness,lem:ut_small}, up to extracting a subsequence, it may be assumed that the functions $\xi\mapsto u(x_n+\xi,t_n)$ converge, uniformly on every compact subset of $\rr$, towards a stationary solution $\xi\mapsto u_\infty(\xi)$ of system~\cref{init_syst}, satisfying
\begin{equation}
\label{hyp_contrib_neg_lim}
\abs{H^\ddag\bigl(u_\infty(\cdot),u_\infty'(\cdot)\bigr)}\le\deltaHam 
\quad\text{and}\quad
\int_{0}^{1} L^\ddag\bigl(u_\infty(\xi),u_\infty'(\xi)\bigr) \, d\xi \le 0
\,.
\end{equation}
According to \vref{lem:lag_sig_far_pos} and to the first inequality of~\cref{hyp_contrib_neg_lim}, there must exist a point $m$ of $\mmm$ in the level set $\valueOfV$ of $V$ such that $\abs{u_\infty(\xi)-m}_{\ddd}\le\dEsc(m)$ for all $\xi$ in $[0,1]$. Then it follows from the second inequality of~\cref{hyp_contrib_neg_lim} that $u_\infty$ must be identically equal to $m$, a contradiction with the second inequality of~\cref{hyp_contrib_neg}.
\end{proof}
\begin{lemma}[approach to normalized Hamiltonian level set zero for all times]
\label{lem:H_small}
The following limit holds:
\[
\sup_{x\in\rr} \abs{H^\ddag\bigl(u(x,t),u_x(x,t)\bigr)}\to 0
\quad\text{as}\quad t\to+\infty \,. 
\]
\end{lemma}
\begin{proof}
Let us proceed by contradiction and assume that the converse is true. Then, according to \cref{lem:lim_inf_sup_H} and since the quantity 
\[
\abs{H^\ddag\bigl(u(x,t),u_x(x,t)\bigr)}
\]
depends continuously on $x$ and $t$, there exists a positive quantity $\deltaHamBis$, not larger than $\deltaHam$, and a sequence $(x_n,t_n)_{n\in\nn}$ in $\rr\times[0,+\infty)$ such that $t_n$ goes to $+\infty$ as $n$ goes to $+\infty$ and such that, for every $n$ in $\nn$, 
\begin{align}
\label{bd_H_in_proof_H_small}
\sup_{x\in\rr}\abs{H^\ddag\bigl(u(x,t_n),u_x(x,t_n)\bigr)} &\le \deltaHam
\,, \\
\nonumber
\text{and}\quad
\abs{H^\ddag\bigl(u(x_n,t_n),u_x(x_n,t_n)\bigr)} &= \deltaHamBis
\,.
\end{align}
According to \cref{lem:compactness,lem:ut_small}, up to extracting a subsequence, it may be assumed that the functions $\xi\mapsto u(x_n+\xi,t_n)$ converge, uniformly on every compact subset of $\rr$, towards a stationary solution 
$\xi\mapsto u_\infty(\xi)$ of system~\cref{init_syst}, satisfying
\[
\abs{H^\ddag\bigl(u_\infty(\cdot),u_\infty'(\cdot)\bigr)} = \deltaHamBis \not= 0
\,.
\]
Since the Hamiltonian of this stationary solution is nonzero, this solution cannot be in $\PhiZero(\valueOfV)$ and thus, according to \vref{prop:infin_lag} (this is the key argument of this proof), 
\begin{equation}
\label{use_Lag_infty}
\int_{-\ell}^\ell L^\ddag\bigl(u_\infty(\xi),u_\infty'(\xi)\bigr)\, d\xi \to +\infty
\quad\text{as}\quad
\ell\to +\infty
\,.
\end{equation}
Let $n$ denote a nonnegative integer. It follows from assertions \cref{u_close_min,ux_small} about the behaviour of the solution outside of the interval $[-c t,c t]$ that, if $n$ is large enough,
\[
-c t_n \le x_n-\ell
\quad\text{and}\quad
x_n + \ell \le c t_n
\,.
\]
Thus, if $\Sigma_n$ denotes the set 
\[
(-\infty, x_n-\ell]\cup [x_n+\ell, +\infty)
\,,
\]
then the energy $\eee(t_n)$ defined in \cref{subsec:localized_energy} reads
\[
\int_{x_n-\ell}^{x_n+\ell} L^\ddag\bigl(u(x,t_n),u_x(x,t_n)\bigr)\, dx + \int_{\Sigma_n} \chi(x,t_n) \, L^\ddag\bigl(u(x,t_n),u_x(x,t_n)\bigr)\, dx 
\,.
\]
According to \cref{bd_H_in_proof_H_small} and to \cref{lem:contrib_Lag} above, the second of these integrals is nonnegative, and according to the limit \cref{use_Lag_infty} above, the first of these integrals is positive and arbitrarily large if $n$ is large enough (depending on the choice of $\ell$), a contradiction with the fact that the (almost decreasing) quantity $\eee(t)$ is bounded from above, uniformly with respect to $t$. 
\end{proof}
It follows from \cref{lem:contrib_Lag,lem:H_small} that the asymptotic energy of the solution is nonnegative (provided that this asymptotic energy is not equal to $-\infty$).
\begin{remark}
If the diffusion matrix $\ddd$ is the identity matrix, the fact that the asymptotic energy is either nonnegative or equal to minus infinity can be proved by another method (different from the one of \cref{subsec:app_zero_Ham_all_times}), namely by setting up a variational scheme in a referential travelling at a small nonzero speed, see \cite[\GlobalBehaviourPropNonNegativeAsymptoticEnergy]{Risler_globalBehaviour_2016}. 
\end{remark}
\subsection{Approach to the set of bistable stationary solutions in the normalized Hamiltonian level set zero}
\label{subsec:app_bist}
The following lemma completes the proof of conclusion \cref{item:thm_main_nonnegative_asympt_energy_approach_bistable_stationary} of \cref{thm:main}.
\begin{lemma}[approach to bistable stationary solutions in the normalized Hamiltonian level set zero]
\label{lem:app_bist}
The following limit holds.
\[
\sup_{x\in\rr}\dist\Bigl(\bigl(u(x,t),u_x(x,t)\bigr) \,, \,I\bigl(\PhiZero(\valueOfV)\bigr)\Bigr)\to 0
\quad\text{as}\quad
t\to +\infty
\,.
\]
\end{lemma}
\begin{proof}
Let us proceed by contradiction and assume that the converse is true. Then there exists a positive quantity $\delta$ and a sequence $(x_n,t_n)_{n\in\nn}$ in $\rr\times[0,+\infty)$ such that $t_n$ goes to $+\infty$ as $n$ goes to $+\infty$ and such that, for every nonnegative integer $n$, 
\begin{equation}
\label{hyp_far_bist}
\dist\Bigl(\bigl(u(x_n,t_n),u_x(x_n,t_n)\bigr) \,, \,I\bigl(\PhiZero(\valueOfV)\bigr)\Bigr)\ge\delta
\,.
\end{equation}
According to \cref{lem:compactness,lem:ut_small}, up to extracting a subsequence, it may be assumed that the functions $\xi\mapsto u(x_n+\xi,t_n)$ converge, uniformly on every compact subset of $\rr$, towards a stationary solution 
$\xi\mapsto u_\infty(\xi)$ of system~\cref{init_syst}. According to \cref{lem:H_small}, the normalized Hamiltonian of this stationary solution must be equal to zero, and according to hypothesis \cref{hyp_far_bist} above this stationary solution cannot belong to the set $\PhiZero(\valueOfV)$. As a consequence, according to \vref{prop:infin_lag} (this is again the key argument of this proof), 
\[
\int_{-\ell}^\ell L^\ddag\bigl(u_\infty(\xi),u_\infty'(\xi)\bigr)\, d\xi \to +\infty
\quad\text{as}\quad
\ell\to +\infty
\,.
\]
Thus, if $\ell$ is a large enough positive quantity, and if $n$ is a large enough (depending on the choice of $\ell$) positive integer, the quantity
\[
\int_{x_n-\ell}^{x_n+\ell} L^\ddag\bigl(u(x,t_n),u_x(x,t_n)\bigr)\, dx 
\]
is positive and arbitrarily large, and the contradiction is the same as in the proof of \cref{lem:H_small} stated previously.
\end{proof}
The proof of conclusion \cref{item:thm_main_nonnegative_asympt_energy_approach_bistable_stationary} of \cref{thm:main} is complete. 
\section{Convergence towards a standing terrace of bistable stationary solutions and value of the asymptotic energy}
\label{sec:cv_standing_terrace_value_asympt_en}
The aim of this \namecref{sec:approach_set_bist_sol} is to prove conclusion \cref{item:thm_main_approach_standing_terrace_and_value_asymptotic_energy} of \cref{thm:main}. Let us keep the assumptions and notation of the previous \namecref{sec:approach_set_bist_sol}, and let us assume in addition that the potential $V$ satisfies hypothesis \textup{(\hyperlink{hypDiscStat}{\hypDiscStatRef{\valueOfV}})}, namely that the set $\PhiZeroNorm(\valueOfV)$ is totally disconnected in $X$. To emphasize the analogy with the convergence towards fronts travelling at a positive speed (see \cite{Risler_globalBehaviour_2016}), the standing terrace of bistable stationary solutions approached by the solution will be obtained ``from right to left'' (as a consequence the numbering of the stationary solutions involved in this terrace will be opposite to the one of \vref{def:standing_terrace}).
\subsection{Convergence towards a standing terrace of bistable stationary solutions}
\label{subsec:cv_standing_terrace}
The aim of this \namecref{subsec:cv_standing_terrace} is to prove the following proposition, which is nothing but conclusion \cref{item:thm_main_approach_standing_terrace} of \cref{thm:main}. 
\begin{proposition}[convergence towards a standing terrace of bistable stationary solutions]
\label{prop:approach_terrace_Czero_norm}
There exists a standing terrace $(x,t)\mapsto \ttt(x,t)$ of bistable stationary solutions, connecting $m_-$ to $m_+$, such that 
\begin{equation}
\label{approach_terrace_Czero_norm}
\sup_{x\in\rr}\abs{u(x,t)-\ttt(x,t)} \to 0
\quad\text{as}\quad
t\to+\infty
\,,
\end{equation}
\end{proposition}
\begin{proof}
\renewcommand{\qedsymbol}{}
Let us denote by $\xEscOf{1}{+}(t)$ the quantity $\xEscPlus(t)$ defined in \cref{subsec:invasion_speeds_vanish}. Observe that, if alternative \cref{item:alternative_xEscPlus_equals_minus_infinity_for_t_large} of \cref{cor:finitenes_infiniteness_of_xEscPlus_dichotomy} holds (that is, if $\xEscOf{1}{+}(t)$ is equal to $-\infty$ for every large enough positive time $t$), then, according to \cref{lem:cv_homog_case}, \cref{prop:approach_terrace_Czero_norm} holds (with a standing terrace reduced to the homogeneous solution at $m_-$, which is equal to $m_+$). Thus it may be assumed that alternative \cref{item:alternative_xEscPlus_is_finite_for_t_large} of \cref{cor:finitenes_infiniteness_of_xEscPlus_dichotomy} holds (that is, $\xEscOf{1}{+}(t)$ is finite for every large enough positive time $t$). Recall that, in this case, according to \cref{lem:asymptotically_vanishing_time_derivative_xEscPlus}, 
\begin{equation}
\label{xEscOnePlusPrime_of_t_goes_to_zero}
\xEscOf{1}{+}'(t)\to 0
\quad\text{as}\quad
t\to+\infty
\,.
\end{equation}
The next lemma (and the repetition of the same argument if the number of stationary solutions involved in the asymptotic pattern is larger than $1$) is the only place in this paper where hypothesis \textup{(\hyperlink{hypDiscStat}{\hypDiscStatRef{\valueOfV}})} is required. For an illustration see \cref{fig:value_asymptotic_energy}.
\end{proof}
\begin{lemma}[approach to an inhomogeneous stationary solution]
\label{lem:cv_inhomog_case}
There exists a stationary solution $\phi_1$ in the set $\PhiZeroNorm(\valueOfV)$ such that $\phi_1(\xi)$ goes to $m_+$ as $\xi$ goes to $+\infty$, and such that the functions
\[
\rr\to\rr^d,\quad \xi\mapsto u\bigl(\xEscOf{1}{+}(t)+\xi,t\bigr)
\]
converge, uniformly on every compact subset of $\rr$, towards $\phi_1$ as time goes to $+\infty$. 
\end{lemma}
\begin{proof}[Proof of \cref{lem:cv_inhomog_case}]
Take a sequence $(t_n)_{n\in\nn}$, such that $t_n$ goes to $+\infty$ as $n$ goes to $+\infty$. According to \cref{lem:compactness,lem:ut_small}, up to extracting a subsequence, it may be assumed that the functions 
\[
\xi\mapsto u\bigl(\xEscOf{1}{+}(t_n)+\xi,t_n\bigr)
\]
converge, uniformly on every compact subset of $\rr$, towards a stationary solution $\xi\mapsto \phi_1(\xi)$ of system~\cref{init_syst}. It follows from the definition of $\xEscOf{1}{+}(t)$ that 
\[
\abs{\phi_1(0)-m_+}_{\ddd} =\dEsc(m_+)
\quad\text{and, for all } \xi \text{ in } [0,+\infty), \quad
\abs{\phi_1(\xi)-m_+}_{\ddd} \le\dEsc(m_+)
\,.
\]
Thus, it follows from assertions \cref{item:cv_spatial_asymptotics_sw,item:closer_spatial_asymptotics_sw} of \vref{lem:ham_equ_hyp} that
\[
\phi_1(\xi)\to m_+
\quad\text{as}\quad
\xi\to+\infty
\quad\text{and, for all } \xi \text{ in } (0,+\infty), \quad
\abs{\phi_1(\xi)-m_+}_{\ddd} <\dEsc(m_+)
\,,
\]
and according to \cref{lem:app_bist}, this stationary solution $\phi_1$ must actually belong to $\PhiZeroNorm(\valueOfV)$. 

Let $\mathcal{L}$ denote the set of all possible limits (in the sense of uniform convergence on compact subsets of $\rr$) of sequences of functions 
\[
\xi\mapsto u\bigl(\xEscOf{1}{+}(t'_n)+\xi,t'_n\bigr)
\]
for all possible sequences $(t_n')_{n\in\nn}$ such that $t_n'$ goes to $+\infty$ as $n$ goes to $+\infty$. This set $\mathcal{L}$ is included in the set $\PhiZeroNorm(\valueOfV)$ defined in \vref{def_phiZeronorm_of_mmm_h}, and, because the semi-flow of system~\cref{init_syst} is continuous on $X$, this set $\mathcal{L}$ is a continuum (a compact connected subset) of $X$. 

Since on the other hand, according to hypothesis \textup{(\hyperlink{hypDiscStat}{\hypDiscStatRef{\valueOfV}})}, the set $\PhiZeroNorm(\valueOfV)$ is totally disconnected, this set $\mathcal{L}$ must actually be reduced to the singleton $\{\phi_1\}$. \Cref{lem:cv_inhomog_case} is proved. 
\end{proof}
\begin{proof}[End of the proof of \cref{prop:approach_terrace_Czero_norm}]
With the notation $\phi_1$ introduced in this \cref{lem:cv_inhomog_case}, let us denote by $m_1$ the limit of $\phi_1(\xi)$ as $\xi$ goes to $-\infty$ (this point must belong to $\mmm_{\valueOfV}$), and let $L_1$ denote the infimum of the (nonempty) set 
\[
\bigl\{\xi\in\rr: \abs{\phi_1(\xi)-m_{1}}_{\ddd} =\dEsc(m_1)\bigr\}
\,.
\]
According to assertion \cref{item:transv_spatial_asymptotics_sw} of \vref{lem:ham_equ_hyp}, 
\[
\bigl\langle \phi_1(L_1)-m_{1},\phi_1'(L_1)\bigr\rangle_{\ddd} >0
\,.
\]
As a consequence, for every large enough positive time $t$ there exists a unique quantity $\xEscOf{1}{-}(t)$ close to $\xEscOf{1}{+}(t)-L_1$ and such that 
\[
\abs{u\bigl(\xEscOf{1}{-}(t),t\bigr)-m_{1}}_{\ddd} =\dEsc(m_1)
\,,
\]
and, as for $\xEscOf{1}{+}(t)$ and for the same reason, 
\begin{equation}
\label{xEscOneMinusPrime_of_t_goes_to_zero}
\xEscOf{1}{-}'(t)\to 0 
\quad\text{as}\quad
t\to+\infty
\,,
\end{equation}
see \cref{fig:value_asymptotic_energy}.
Let us repeat to the left of $\xEscOf{1}{-}(t)$ the same construction and let us denote by $\xEscOf{2}{-}(t)$ the supremum of the set
\[
\Bigl\{x\text{ in } \bigl(-\infty,\xEscOf{1}{-}(t)\bigl) : \abs{u(x,t)-m_{1}}_{\ddd} =\dEsc(m_1)\Bigr\}
\]
(with the convention that $\xEscOf{2}{-}(t)=-\infty$ if this set is empty). It follows from \cref{lem:cv_inhomog_case} that 
\begin{equation}
\label{xEscOneMinus_Minus_xEscTwoPlus_goes_to_infinity}
\xEscOf{1}{-}(t)-\xEscOf{2}{+}(t)\to+\infty
\quad\text{as}\quad 
t\to+\infty 
\,.
\end{equation}
At this stage, it can be observed that \cref{cor:finitenes_infiniteness_of_xEscPlus_dichotomy} applies again, with $\xEscPlus(t)$ replaced by $\xEscOf{2}{+}(t)$ and $m_+$ replaced by $m_1$. Thus, the same alternative holds again: for $t$ large enough positive, the quantity $\xEscOf{2}{+}(t)$ is either always finite or always equal to $-\infty$. In addition, 
\begin{enumerate}
\item if $\xEscOf{2}{+}(t)$ equals $-\infty$ for all $t$ large enough positive, then, since the solution under consideration is assumed to be a bistable solution connecting $m_-$ to $m_+$, it follows from the definition of $\dEsc(m_1)$ that $m_1$ must be equal to $m_-$;
\item and if $\xEscOf{2}{+}(t)$ is finite for all $t$ large enough positive, then it can be argued as in \cref{lem:asymptotically_vanishing_time_derivative_xEscPlus} that $\xEscOf{2}{+}(t)$ goes to $0$ as time goes to $+\infty$, and as in \cref{lem:cv_inhomog_case} that there exists $\phi_2$ in $\PhiZeroNorm(\valueOfV)$ such that the function $\xi\mapsto u\bigl(\xEscOf{2}{+}(t)+\xi,t\bigr)$ converges towards $\phi_2$ on every compact subset of $\rr$ as $t$ goes to $+\infty$. And the same construction can be repeated again introducing the supremum $\xEscOf{3}{+}(t)$ of the set 
\[
\Bigl\{x \text{ in } \bigl(-\infty,\xEscOf{2}{-}(t)\bigr) : \abs{u(x,t)-m_{2}}_{\ddd} \ge\dEsc(m_2)\Bigr\} 
\,,
\]
where $m_2$ is the limit of $\phi_2(\xi)$ as $\xi$ goes to $-\infty$, and $\xEscOf{2}{-}(t)$ is defined as $\xEscOf{1}{-}(t)$ above. 
\end{enumerate}
Because the localized energy $\eee(t)$ is bounded from above, this process must eventually end up at some $q$ in $\nn^*$ for which $\xEscOf{q+1}{+}(t)$ equals $+\infty$ for all $t$ large enough positive. Then the limit $m_{q}$ at $-\infty$ of the last stationary solution $\phi_q$ must be equal to $m_-$ (see \cref{fig:value_asymptotic_energy}). For every $i$ in $\{1,\dots,q\}$, the limits \cref{xEscOnePlusPrime_of_t_goes_to_zero,xEscOneMinusPrime_of_t_goes_to_zero} still hold, for the same reason, if the index ``$1$'' is replaced by ``$i$'', that is:
\begin{equation}
\label{xEsciPlusMinusPrime_of_t_goes_to_zero}
\xEscOf{i}{+}'(t)\to 0 
\quad\text{as}\quad
t\to+\infty\,,
\quad\text{and}\quad
\xEscOf{i}{-}'(t)\to 0 
\quad\text{as}\quad
t\to+\infty
\,.
\end{equation}
And, if $q$ is not smaller than $2$, for every $i$ in $\{1,\dots,q-1\}$, due to the same reason as the limit \cref{xEscOneMinus_Minus_xEscTwoPlus_goes_to_infinity}, the following limit holds:
\begin{equation}
\label{xEsciPlusOneMinus_Minus_xEsciPlus_goes_to_infinity}
\xEscOf{i+1}{-}(t)-\xEscOf{i}{+}(t)\to+\infty
\quad\text{as}\quad 
t\to+\infty 
\,.
\end{equation}
Let us denote $m_+$ by $m_0$ and, for $t$ large enough positive, let us introduce the standing terrace $\ttt(\cdot,t)$ defined as 
\[
\ttt(x,t) = m_+ + \sum_{i=1}^q \Bigl[\phi_i\bigl(x-\xEscOf{i}{+}(t)\bigr)-m_{i-1}\Bigr]
\,.
\]
For every positive quantity $L$, it follows from the limits \cref{xEsciPlusOneMinus_Minus_xEsciPlus_goes_to_infinity} and from the asymptotics of $\phi_1,\dots,\phi_q$ at the ends of $\rr$ that, for every $i$ in $\{1,\dots,q\}$, 
\[
\sup_{x\in[\xEscOf{i}{-}(t)-L,\, \xEscOf{i}{+}(t)+L]}\abs{\ttt(x,t)-\phi_i\bigl(x-\xEscOf{i}{+}(t)\bigr)}\to0
\quad\text{as}\quad 
t\to+\infty 
\,.
\]
On the other hand, it follows from the construction of the profile $\phi_i$ that
\begin{equation}
\label{approach_phii_around_xEscPlusMinusi}
\sup_{x\in[\xEscOf{i}{-}(t)-L,\, \xEscOf{i}{+}(t)+L]}\abs{u(x,t)-\phi_i\bigl(x-\xEscOf{i}{+}(t)\bigr)}\to0
\quad\text{as}\quad 
t\to+\infty 
\,.
\end{equation}
Thus, according to these two limits, 
\[
\sup_{i\in\{1,\dots,q\},\ x\in[\xEscOf{i}{-}(t)-L,\xEscOf{i}{+}(t)+L]}\abs{u(x,t)-\ttt(x,t)}\to0
\quad\text{as}\quad 
t\to+\infty 
\,.
\]
Finally, by the same argument as in the proof of \cref{lem:cv_homog_case}, the stronger limit \cref{approach_terrace_Czero_norm} must also hold, and \cref{prop:approach_terrace_Czero_norm} is proved. 
\end{proof}
Conclusion \cref{item:thm_main_approach_standing_terrace} of \cref{thm:main} is therefore also proved.
\subsection{Value of the asymptotic energy}
\label{subsec:value_asymptotic_energy}
Recall (\cref{def:energy_standing_terrace}) that the energy of the standing terrace $\ttt$ is the quantity $\eee[\ttt]$ defined as
\[
\eee[\ttt] = 0
\quad\text{if}\quad
q = 0 \,,
\quad\text{and}\quad
\eee[\ttt] = \sum_{i=1}^q \eee[\phi_i]
\quad\text{if}\quad
q > 0
\,.
\]
The aim of this \namecref{subsec:value_asymptotic_energy} is to prove the following proposition, which yields conclusion \cref{item:thm_main_value_asymptotic_energy} of \cref{thm:main}. 
\begin{proposition}[the asymptotic energy of the solution equals the energy of the standing terrace]
\label{prop:eeeAsympt_of_u_equals_eee_of_ttt}
The asymptotic energy of the solution equals the energy of the standing terrace. That is, with symbols, 
\begin{equation}
\label{eeeAsympt_of_u_equals_eee_of_ttt}
\eeeAsympt[u] = \eee[\ttt]
\,.
\end{equation}
\end{proposition}
\begin{proof}
\renewcommand{\qedsymbol}{}
Let $\varepsilon$ denote a (small) positive quantity, to be chosen later (its value is chosen in \cref{def_varepsilon} and depends only on $m_0,\dots,m_q$). According to \cref{lem:convergence_to_asympt_enery_for_various_speeds_of_bounds}, with the notation $E^\ddag(x,t)$ of \cref{def_Eddag}, the following limit holds:
\begin{equation}
\label{cv_asympt_energy_to_get_its_value}
\int_{-\varepsilon t}^{\varepsilon t}E^\ddag(x,t)\, dx \to \eeeAsympt[u]
\quad\text{as}\quad
t\to+\infty
\,.
\end{equation}
Let us adopt the following conventions:
\begin{itemize}
\item the set $\{1,\dots,q\}$ denotes the empty set if $q$ equals zero, 
\item and the set $\{1,\dots,q-1\}$ denotes the empty set if $q$ is not larger than $1$, 
\item and the supremum of any expression over an empty set equals $-\infty$, 
\item and the infimum of any expression over an empty set equals $+\infty$.
\end{itemize}
Let $t_0$ denote a positive time, large enough so that, for every $i$ in $\{1,\dots,q\}$, the functions $\xEscOf{i}{-}(t)$ and $\xEscOf{i}{+}(t)$ are defined and of class $\ccc^1$. For every $i$ in $\{1,\dots,q\}$, let us introduce the function $\rr\times[t_0,+\infty)\to\rr$, $(x,t)\mapsto E^\ddag_i(x,t)$ defined as 
\[
E^\ddag_i(x,t) = \frac{1}{2}\abs{\phi_i'\bigl(x-\xEscOf{i}{+}(t)\bigr)}_{\ddd}^2 + V^\ddag\Bigl(\phi_i\bigl(x-\xEscOf{i}{+}(t)\bigr)\Bigr)
\,.
\]
For every positive quantity $L$, according to the bounds \cref{bound_u_ut_ck} on the solution, it follows from the limit \cref{approach_phii_around_xEscPlusMinusi} that
\begin{equation}
\label{approach_phii_around_xEscPlusMinusi_derivatives}
\sup_{x\in[\xEscOf{i}{-}(t)-L,\,\xEscOf{i}{+}(t)+L]}\abs{u_x(x,t)-\phi_i'\bigl(x-\xEscOf{i}{+}(t)\bigr)} \to 0
\quad\text{as}\quad
t\to+\infty
\,,
\end{equation}
and it follows from the limits \cref{approach_phii_around_xEscPlusMinusi,approach_phii_around_xEscPlusMinusi_derivatives} that 
\begin{equation}
\label{integral_of_E_minus_Ei_goes_to_zero}
\int_{\xEscOf{i}{-}(t)-L}^{\xEscOf{i}{+}(t)+L} \bigl(E^\ddag(x,t) - E^\ddag_i(x,t)\bigr)\, dx \to 0
\quad\text{as}\quad
t\to+\infty
\,.
\end{equation}
For every positive integer $n$, let us denote by $\tau_n$ the supremum of the set 
\begin{align}
\nonumber
\biggl\{t\in[t_0,+\infty):& \text{ at least one among the following two inequalities holds: }\\
&\sup_{i\in\{1,\dots,q\}}\int_{\xEscOf{i}{-}(t)-n}^{\xEscOf{i}{+}(t)+n} \abs{E^\ddag(x,t) - E^\ddag_i(x,t)}\, dx \ge \frac{1}{n}\,, 
\label{energies_on_bulk_of_width_2n_differs}\\
\text{and}\quad&
\inf_{i\in\{1,\dots,q-1\}} \xEscOf{i+1}{-}(t) - \xEscOf{i}{+}(t) \le 2n
\biggr\}
\label{gap_is_smaller_than_2n}
\,.
\end{align}
According to the limits \cref{xEsciPlusOneMinus_Minus_xEsciPlus_goes_to_infinity,integral_of_E_minus_Ei_goes_to_zero}, this quantity $\tau_n$ is finite. And according to this definition, for every time $t$ greater than or equal to $\tau_n$, both inequalities \cref{energies_on_bulk_of_width_2n_differs,gap_is_smaller_than_2n} are false. Let us introduce the sequence $(t_n)_{n\in\nn}$ starting at the time $t_0$ introduced above and defined, for every positive integer $n$, by the recurrence relation: 
\[
t_n = \max(t_{n-1} + n,\tau_n)
\,.
\]
Let $\chi$ denote a smooth function $\rr\to\rr$ satisfying
\[
\chi\equiv 0 \text{ on } (-\infty,0]
\quad\text{and}\quad
\chi\equiv 1 \text{ on } [1,+\infty)
\quad\text{and}\quad
0 \le \chi \le 1 \text{ and }\chi' \ge 0\text{ on }[0,1]
\,,
\]
and let us introduce the function $\Lext:[t_0,+\infty)\to[0,+\infty)$ defined as
\[
\Lext(t) = n-1 + \chi\left(\frac{t-t_{n-1}}{t_{n}-t_{n-1}}\right) 
\quad\text{for $n$ in $\nn^*$ and $t$ in $[t_{n-1},t_{n}]$}
\,.
\]
This function will play the role of a (growing) ``extent length'' of the intervals over which the solution gets close (in terms of energy) to translates of the stationary solutions $\phi_i$ (accordingly, the subscript ``ext'' refers to the word ``extent''). 
\begin{figure}[!htbp]
	\centering
    \includegraphics[width=\textwidth]{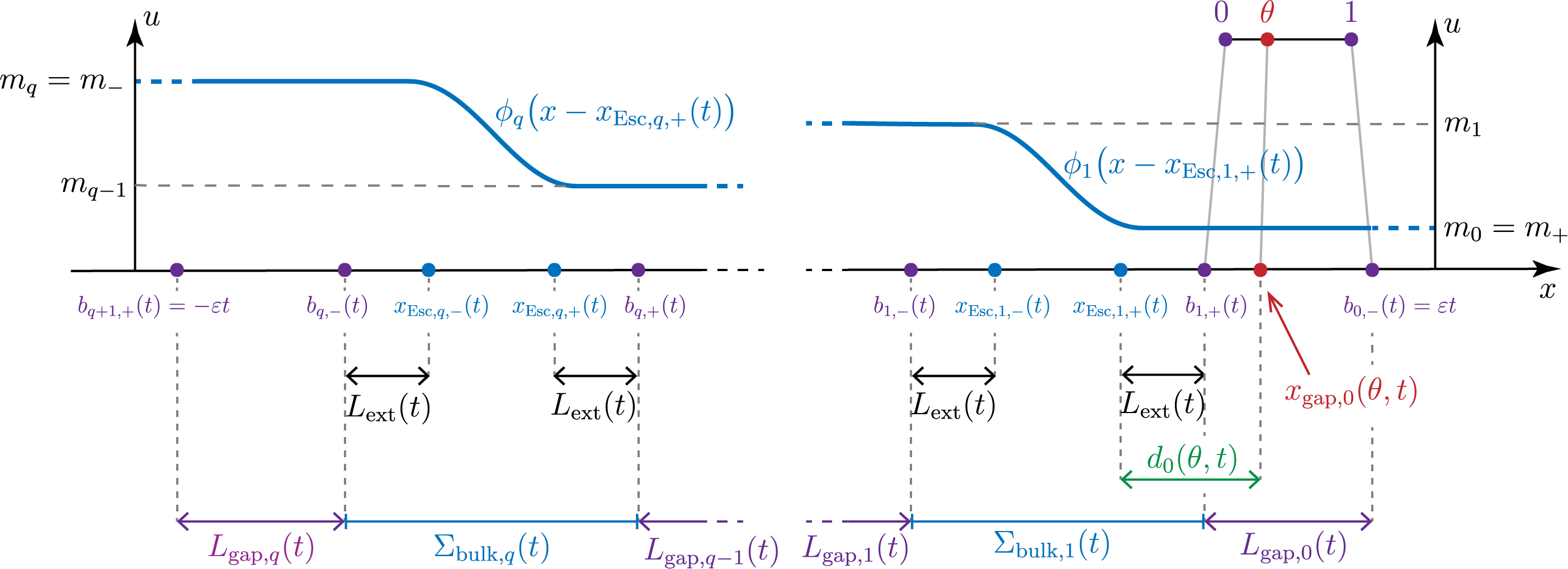}
    \caption{Illustration of the notation of \cref{subsec:value_asymptotic_energy}.}
    \label{fig:value_asymptotic_energy}
\end{figure}
It follows from this definition that this function $\Lext(\cdot)$ is non-decreasing and of class $\ccc^1$ on $[t_0,+\infty)$, and that 
\begin{equation}
\label{asymptotics_of_Lext_and_Lextprime}
\Lext(t)\to+\infty
\quad\text{and}\quad
\Lext'(t)\to0
\quad\text{as}\quad
t\to+\infty
\,.
\end{equation}
For every $i$ in $\{1,\dots,q\}$, let 
\[
b_{i,-}(t) = \xEscOf{i}{-}(t)-\Lext(t)
\quad\text{and}\quad
b_{i,+}(t) = \xEscOf{i}{+}(t)+\Lext(t)
\]
(this notation refers to the word ``border''), see \cref{fig:value_asymptotic_energy}, and let us denote by $\SigmaBulkOf{i}(t)$ the interval $[b_{i,-}(t),b_{i,+}(t)]$. It follows from the definition of $\tau_n$ that, for every $i$ in $\{1,\dots,q\}$,
\begin{equation}
\label{integral_of_E_minus_Ei_goes_to_zero_on_growing_interval}
\int_{\SigmaBulkOf{i}(t)} \bigl(E^\ddag(x,t) - E^\ddag_i(x,t)\bigr)\, dx \to 0 
\quad\text{as}\quad
t\to+\infty
\,,
\end{equation}
and, for every $i$ in $\{1,\dots,q-1\}$,
\begin{equation}
\label{xEsciPlusone_minus_xEsci_no_overlap}
\xEscOf{i}{+}(t)+\Lext(t)\le \xEscOf{i+1}{-}(t)-\Lext(t)
\,,
\end{equation}
and it follows from the limits \cref{xEsciPlusMinusPrime_of_t_goes_to_zero,asymptotics_of_Lext_and_Lextprime} that, as $t$ goes to $+\infty$, 
\begin{equation}
\label{energy_on_SigmaBulk_i_goes_to_energy_of_phi_i}
b_{i,-}'(t)\to0
\quad\text{and}\quad
b_{i,+}'(t)\to0
\quad\text{and}\quad
\int_{\SigmaBulkOf{i}(t)} E^\ddag_i(x,t)\, dx \to \eee[\phi_i]
\,.
\end{equation}
Let 
\[
\SigmaBulk(t) = \bigcup_{i\in\{1,\dots,q\}} \SigmaBulkOf{i}(t)
\,.
\]
According to inequality \cref{xEsciPlusone_minus_xEsci_no_overlap}, this union is actually a \emph{disjoint} union, thus it follows from the limits \cref{integral_of_E_minus_Ei_goes_to_zero_on_growing_interval,energy_on_SigmaBulk_i_goes_to_energy_of_phi_i} that 
\[
\int_{\SigmaBulk(t)}E^\ddag(x,t)\, dx \to \eee[\ttt]
\quad\text{as}\quad
t\to+\infty
\,.
\]
In view of the limit \cref{cv_asympt_energy_to_get_its_value}, \cref{prop:eeeAsympt_of_u_equals_eee_of_ttt} is a consequence of the following lemma. 
\end{proof}
\begin{lemma}[the energy over the set $[-\varepsilon t,\varepsilon t{]}\setminus\SigmaBulk(t)$ goes to $0$]
\label{lem:no_energy_besides_SigmaBulk}
The following limit holds:
\begin{equation}
\label{no_energy_besides_SigmaBulk}
\int_{[-\varepsilon t,\varepsilon t]\setminus\SigmaBulk(t)}E^\ddag(x,t)\, dx \to 0
\quad\text{as}\quad
t\to+\infty
\,.
\end{equation}
\end{lemma}
\begin{proof}[Proof of \cref{lem:no_energy_besides_SigmaBulk}]
\renewcommand{\qedsymbol}{}
According to the first two limits of \cref{energy_on_SigmaBulk_i_goes_to_energy_of_phi_i}, there exists a time $t'_0$, greater than or equal to $t_0$, such that, for every time $t$ greater than or equal to $t'_0$, 
\[
-\varepsilon t < b_{q,-}(t)
\quad\text{and}\quad
b_{1,+}(t) < \varepsilon t
\,,
\]
see \cref{fig:value_asymptotic_energy}. Let us denote by $b_{q+1,+}(t)$ the quantity $-\varepsilon t$ and by $b_{0,-}(t)$ the quantity $\varepsilon t$, and let us assume that $t$ is greater than or equal to $t'_0$. Then the following equality holds: 
\[
(-\varepsilon t,\varepsilon t)\setminus\SigmaBulk(t) = \bigsqcup_{i=0}^q \bigl(b_{i+1,+}(t),b_{i,-}(t)\bigr)
\,.
\]
For every $i$ in $\{0,\dots,q\}$, let us introduce the quantity
\begin{equation}
\label{def_LgapOf_i}
\LgapOf{i}(t) = b_{i,-}(t)- b_{i+1,+}(t)
\,.
\end{equation}
Let us introduce the quantities $\kappa_i$ and $\nuFOf{i}$ and $\KFOf{i}$ and the set $\SigmaEscOf{i}(t)$ obtained by replacing, in the expression \cref{def_kappa} of $\kappa$ and \cref{def_nuF} of $\nuF$ and \cref{def_KF} of $\KF$ and \cref{def_SigmaEsc} of $\SigmaEsc(t)$, the minimum point $m$ considered in \cref{sec:spat_asympt} by $m_i$; and let us denote by $\psi_i$ the function obtained by replacing, in the expression \cref{def_psi} of $\psi$, the quantity $\kappa$ by $\kappa_i$. Let us introduce the functions $F^\ddag_i(\cdot,\cdot)$ and $\fff_i(\cdot,\cdot)$ defined as 
\[
F^\ddag_i(x,t) = E^\ddag(x,t) + \frac{1}{2}\bigl(u(x,t)-m_i\bigr)^2
\quad\text{and}\quad
\fff_i(\widebar{x},t) = \int_{\rr} T_{\widebar{x}}\psi_i(x) F^\ddag_i(x,t)\, dx
\,,
\]
and let us introduce the functions $\xGapOf{i}$ and $\gggg_i$, from $[0,1]\times[t_0,+\infty)$ to $\rr$, defined as
\[
\xGapOf{i}(\theta,t) = (1-\theta)b_{i+1,+}(t)+\theta b_{i,-}(t)
\quad\text{and}\quad
\gggg_i(\theta,t) = \fff_i\bigl(\xGapOf{i}(\theta,t),t\bigr)
\,.
\]
For every $(\theta,t)$ in $[0,1]\times[t'_0,+\infty)$, 
\[
\partial_t \gggg_i(\theta,t) = \partial_{\widebar{x}}\fff_i\bigl(\xGapOf{i}(\theta,t),t\bigr) \partial_t \xGapOf{i}(\theta,t) + \partial_t\fff_i\bigl(\xGapOf{i}(\theta,t),t\bigr)
\,.
\]
According to inequality \cref{dt_fire} of \cref{lem:approx_decrease_fire} (after substituting the notation used in this inequality with the notation above), 
\[
\begin{aligned}
\partial_t\fff_i\bigl(\xGapOf{i}(\theta,t),t\bigr) &\le -\nuFOf{i}\, \fff_i\bigl(\xGapOf{i}(\theta,t),t\bigr) + \KFOf{i}\, \int_{\SigmaEscOf{i}(t)} T_{\xGapOf{i}(\theta,t)}\psi_i(x)\, dx \\
\le &-\nuFOf{i}\, \fff_i\bigl(\xGapOf{i}(\theta,t),t\bigr) + \frac{2\KFOf{i}}{\kappa_i}\exp\Bigl(-\kappa_i \dist\bigl(\xGapOf{i}(\theta,t),\SigmaEscOf{i}(t)\bigr)\Bigr) 
\,,
\end{aligned}
\]
where $\dist\bigl(\xGapOf{i}(\theta,t),\SigmaEscOf{i}(t)\bigr)$ denotes the distance (in $\rr$) between the point $\xGapOf{i}(\theta,t)$ and the set $\SigmaEscOf{i}(t)$. Observe that this distance is not smaller than the quantity $d_i(\theta,t)$ defined as
\begin{equation}
\label{def_d_i}
d_i(\theta,t) = \Lext(t) + \min(\theta,1-\theta)\LgapOf{i}(t)
\,.
\end{equation}
As a consequence, it follows from the previous inequality that 
\[
\partial_t\fff_i\bigl(\xGapOf{i}(\theta,t),t\bigr) \le  -\nuFOf{i}\, \fff_i\bigl(\xGapOf{i}(\theta,t),t\bigr) + \frac{2\KFOf{i}}{\kappa_i}\exp\bigl(-\kappa_i d_i(\theta,t) \bigr) 
\,.
\]
Besides, 
\[
\abs{\partial_t \xGapOf{i}(\theta,t)}\le \max\left(\abs{b_{i+1,+}'(t)},\abs{b_{i,-}'(t)}\right)
\,,
\]
so that, according to the first two limits of \cref{energy_on_SigmaBulk_i_goes_to_energy_of_phi_i} and up to increasing $t'_0$, for $t$ greater than or equal to $t'_0$, 
\[
\abs{\partial_t \xGapOf{i}(\theta,t)}\le \varepsilon
\,.
\]
Besides, according to the definition of the weight function $\psi_i$,  
\[
\abs{\partial_{\widebar{x}}\fff_i\bigl(\xGapOf{i}(\theta,t),t\bigr)} \le \kappa_i \fff_i\bigl(\xGapOf{i}(\theta,t),t\bigr)
\,.
\]
It follows that, for $t$ greater than or equal to $t'_0$, 
\[
\partial_t \gggg_i(\theta,t) \le - (\nuFOf{i}-\varepsilon\kappa_i)\gggg_i(\theta,t) + \frac{2\KFOf{i}}{\kappa_i}\exp\bigl(-\kappa_i d_i(\theta,t) \bigr)
\,,
\]
so that if the quantity $\varepsilon$ is chosen as
\begin{equation}
\label{def_varepsilon}
\varepsilon = \min_{i\in\{0,\dots,q\}}\frac{\nuFOf{i}}{16\kappa_i}
\,,
\end{equation}
then the previous inequality yields
\begin{equation}
\label{upper_bound_partial_t_gggg_i}
\partial_t \gggg_i(\theta,t) \le - \frac{\nuFOf{i}}{2} \gggg_i(\theta,t) + \frac{2\KFOf{i}}{\kappa_i}\exp\bigl(-\kappa_i d_i(\theta,t) \bigr)
\,.
\end{equation}
At this stage, a factor $2$ instead of $16$ in the denominator of the right-hand side of \cref{def_varepsilon} would be sufficient for the previous inequality to hold, but this factor $16$ will turn out to be useful in the proof of the next lemma. The proof of \cref{lem:no_energy_besides_SigmaBulk} will be completed through the next three statements. 
\end{proof}
\begin{lemma}[upper bound on $\gggg_i(\theta,t)$ for $t$ large positive]
\label{lem:upper_bound_gggg_i}
There exists a time $t''_0$ greater than or equal to $t'_0$ such that, for every $i$ in $\{0,\dots,q\}$ and for every $\theta$ in $[0,1]$ and for every time $t$ greater than or equal to $t''_0$, 
\begin{equation}
\label{upper_bound_gggg_i}
\gggg_i(\theta,t) \le \frac{8\KFOf{i}}{\kappa_i\nuFOf{i}}\exp\bigl(-\kappa_i d_i(\theta,t) \bigr)
\,.
\end{equation}
\end{lemma}
\begin{proof}
For every $i$ in $\{0,\dots,q\}$, let us introduce the function $\hhh_i(\cdot,\cdot)$ defined as
\begin{equation}
\label{def_hhh_i}
\hhh_i(\theta,t) = \gggg_i(\theta,t) - \frac{8\KFOf{i}}{\kappa_i\nuFOf{i}}\exp\bigl(-\kappa_i d_i(\theta,t) \bigr)
\,.
\end{equation}
It follows from inequality \cref{upper_bound_partial_t_gggg_i} that, for every $\theta$ in $[0,1]$ and for every time $t$ greater than or equal to $t'_0$, 
\[
\begin{aligned} 
\partial_t\hhh_i(\theta,t) &= - \frac{\nuFOf{i}}{2}\gggg_i(\theta,t) + \frac{2\KFOf{i}}{\kappa_i}\exp\bigl(-\kappa_i d_i(\theta,t) \bigr) + \frac{8\KFOf{i}}{\nuFOf{i}}\partial_t d_i(\theta,t)\exp\bigl(-\kappa_i d_i(\theta,t) \bigr) \\
&\le - \frac{\nuFOf{i}}{2}\gggg_i(\theta,t) + \frac{2\KFOf{i}}{\kappa_i}\exp\bigl(-\kappa_i d_i(\theta,t) \bigr)\left(1 + \frac{4\kappa_i}{\nuFOf{i}}\partial_t d_i(\theta,t)\right) \\
&\le - \frac{\nuFOf{i}}{2}\hhh_i(\theta,t) - \frac{2\KFOf{i}}{\kappa_i}\exp\bigl(-\kappa_i d_i(\theta,t) \bigr)\left(1 -\frac{4\kappa_i}{\nuFOf{i}}\partial_t d_i(\theta,t)\right)
\,.
\end{aligned}
\]
According to the definition \cref{def_d_i} of $d_i(\theta,t)$ and \cref{def_LgapOf_i} of $\LgapOf{i}(t)$, 
\[
\partial_t d_i(\theta,t) = \Lext'(t) + \min(\theta,1-\theta)\bigl(b_{i,-}'(t)- b_{i+1,+}'(t)\bigr)
\,.
\]
Since $\Lext'(t)$ goes to $0$ as $t$ goes to $+\infty$, and since $b_{i,-}'(t)$ either goes to $0$ as $t$ goes to $+\infty$ or (if $i$ equals $0$) is equal to $\varepsilon$, and since $b_{i+1,+}'(t)$ either goes to $0$ as $t$ goes to $+\infty$ or (if $i+1$ equals $q+1$) is equal to $-\varepsilon$, there exists a time $t_0'''$, greater than or equal to $t_0'$, such that, if $t$ is greater than or equal to $t'''_0$, then (for every $i$ in $\{1,\dots,q\}$)
\[
\partial_t d_i(\theta,t) \le 2\varepsilon\,,
\quad\text{thus}\quad
\frac{4\kappa_i}{\nuFOf{i}}\partial_t d_i(\theta,t) \le\frac{1}{2}
\,,
\]
and as a consequence, 
\[
\begin{aligned}
\partial_t\hhh_i(\theta,t) &\le -\frac{\nuFOf{i}}{2}\hhh_i(\theta,t) - \frac{\KFOf{i}}{\kappa_i}\exp\bigl(-\kappa_i d_i(\theta,t) \bigr) \\
&\le -\frac{\nuFOf{i}}{2}\hhh_i(\theta,t)- \frac{\KFOf{i}}{\kappa_i}\exp\bigl(-\kappa_i d_i(\theta,t_0''')\bigr)\exp\bigl(-2 \varepsilon \kappa_i (t-t_0''')\bigr)
\,.
\end{aligned}
\]
Let us introduce the function $\jjj_i(\cdot,\cdot)$ defined as
\[
\jjj_i(\theta,t) = \hhh_i(\theta,t) \exp\bigl(2\varepsilon \kappa_i (t-t'''_0)\bigr)
\,.
\]
Then, for every time $t$ greater than or equal to $t'''_0$, 
\[
\begin{aligned}
\partial_t \jjj_i(\theta,t) &\le \Biggl(\biggl(-\frac{\nuFOf{i}}{2} + 2\varepsilon \kappa_i\biggr) \hhh_i(\theta,t) - \frac{\KFOf{i}}{\kappa_i}\exp\bigl(-\kappa_i d_i(\theta,t_0''')\bigr)\exp\bigl(-2 \varepsilon \kappa_i (t-t_0''')\bigr)\Biggr) \\
& \qquad \times \exp\bigl(2\varepsilon \kappa_i (t-t'''_0)\bigr) \\
&\le -\frac{\nuFOf{i}}{4} \jjj_i(\theta,t) - \frac{\KFOf{i}}{\kappa_i}\exp\bigl(-\kappa_i d_i(\theta,t_0''')\bigr) \\
&\le -\frac{\nuFOf{i}}{4} \jjj_i(\theta,t) - \frac{\KFOf{i}}{\kappa_i}\exp\bigl(-\kappa_i d_i(1/2,t_0''')\bigr)
\,.
\end{aligned}
\]
This last inequality shows that $\jjj_i(\theta,t)$ must eventually become negative (and remain negative afterwards) as time increases. More precisely, since according to the bounds \cref{bound_u_ut_ck} on the solution the quantity $\gggg_i(\theta,t'''_0)$ is bounded uniformly with respect to $\theta$, the same is true for the quantity $\jjj_i(\theta,t'''_0)$. As a consequence, there must exist a time $t''_0$, greater than or equal to $t'''_0$, such that, for every $i$ in $\{1,\dots,q\}$ and $\theta$ in $[0,1]$ and $t$ in $[t''_0,+\infty)$, 
\[
\jjj_i(\theta,t) \le 0 \,,
\quad\text{so that}\quad
\hhh_i(\theta,t) \le 0
\,,
\]
and in view of the definition \cref{def_hhh_i} of $\hhh_i(\theta,t)$, inequality \cref{upper_bound_gggg_i} follows. \Cref{lem:upper_bound_gggg_i} is proved. 
\end{proof}
\begin{corollary}[Upper bound on the integral of the firewall over a gap]
\label{cor:upper_bound_integral_fff_i}
For every time $t$ greater than or equal to $t_0''$ and for every $i$ in $\{1,\dots,q\}$, the following inequality holds:
\begin{equation}
\label{upper_bound_integral_fff_i}
\int_{b_{i+1,+}(t)}^{b_{i,-}(t)} \fff_i(x,t)\, dx \le \frac{16\KFOf{i}}{\kappa_i^2\nuFOf{i}}\exp\bigl(-\kappa_i \Lext(t) \bigr)
\,.
\end{equation}
\end{corollary}
\begin{proof}[Proof of \cref{cor:upper_bound_integral_fff_i}]
For every time $t$ greater than or equal to $t_0''$ and for every $i$ in $\{1,\dots,q\}$, using the notation $\LgapOf{i}(t)$ introduced in \cref{def_LgapOf_i}, 
\[
\int_{b_{i+1,+}(t)}^{b_{i,-}(t)} \fff_i(x,t)\, dx = \LgapOf{i}(t) \int_0^1 \gggg_i(\theta,t)\, d\theta
\,,
\]
so that, according to \cref{lem:upper_bound_gggg_i} and to the expression \cref{def_d_i} of $d_i(\theta,t)$, 
\[
\begin{aligned}
\int_{b_{i+1,+}(t)}^{b_{i,-}(t)} \fff_i(x,t)\, dx &\le \LgapOf{i}(t) \frac{8\KFOf{i}}{\kappa_i\nuFOf{i}}\exp\bigl(-\kappa_i \Lext(t) \bigr) \\
&\qquad\times\int_0^1 \exp\bigl(-\kappa_i \min(\theta,1-\theta)\LgapOf{i}(t) \bigr)\, d\theta \\
&\le \LgapOf{i}(t) \frac{16\KFOf{i}}{\kappa_i\nuFOf{i}}\exp\bigl(-\kappa_i \Lext(t) \bigr) \int_0^{1/2} \exp\bigl(-\kappa_i\theta\LgapOf{i}(t) \bigr)\, d\theta \\
&\le \LgapOf{i}(t) \frac{16\KFOf{i}}{\kappa_i\nuFOf{i}}\exp\bigl(-\kappa_i \Lext(t) \bigr) \frac{1}{\kappa_i \LgapOf{i}(t)}
\,,
\end{aligned}
\]
and inequality \cref{upper_bound_integral_fff_i} follows. \Cref{cor:upper_bound_integral_fff_i} is proved. 
\end{proof}
\begin{lemma}[integral of firewall dominates integral of energy]
\label{lem:integral_of_fff_i_dominates_integral_of_energy}
For every time $t$ greater than or equal to $t_0$ and for every $i$ in $\{1,\dots,q\}$, the following inequalities hold: 
\begin{equation}
\label{integral_of_fff_i_dominates_integral_of_energy}
0\le\int_{b_{i+1,+}(t)}^{b_{i,-}(t)} E^\ddag(x,t)\, dx \le \frac{2\kappa_i}{\coeffEn} \int_{b_{i+1,+}(t)}^{b_{i,-}(t)} \fff_i(x,t)\, dx 
\,.
\end{equation}
\end{lemma}
\begin{proof}[Proof of \cref{lem:integral_of_fff_i_dominates_integral_of_energy}]
The left inequality follows from the empty intersection between the interval $[b_{i+1,+}(t),b_{i,-}(t)]$ and the set $\SigmaEscOf{i}(t)$. Concerning the right inequality, for every time $t$ greater than or equal to $t_0$ and for every $i$ in $\{1,\dots,q\}$,
\begin{equation}
\label{integral_of_fff_i_dominates_integral_of_F_i_step}
\begin{aligned}
\int_{b_{i+1,+}(t)}^{b_{i,-}(t)} \fff_i(x,t)\, dx &= \int_{b_{i+1,+}(t)}^{b_{i,-}(t)} \left(\int_{\rr} \psi_i(y-x)F^\ddag_i(y,t)\, dy \right)\, dx \\
&= \int_{\rr} F^\ddag_i(y,t) \left(\int_{b_{i+1,+}(t)}^{b_{i,-}(t)} \psi_i(y-x) \, dx \right)\, dy \\
&\ge \int_{b_{i+1,+}(t)}^{b_{i,-}(t)} F^\ddag_i(y,t) \left(\int_{b_{i+1,+}(t)}^{b_{i,-}(t)} \psi_i(y-x) \, dx \right)\, dy 
\,,
\end{aligned}
\end{equation}
and, with the notation $\LgapOf{i}(t)$ introduced in \cref{def_LgapOf_i},
\[
\begin{aligned}
\int_{b_{i+1,+}(t)}^{b_{i,-}(t)} \psi_i(y-x) \, dx &\ge \int_0^{\LgapOf{i}(t)} e^{-\kappa_i \xi}\, d\xi \\
&\ge \frac{1}{\kappa_i} \Bigl(1 - \exp\bigl(-\kappa_i \LgapOf{i}(t)\bigr)\Bigr)
\,.
\end{aligned}
\]
Since the quantity $\LgapOf{i}(t)$ goes to $+\infty$ as time goes to $+\infty$, it is greater than $\log(2)/\kappa_i$ if $t$ is large enough positive, and in this case it follows from the previous inequality that
\[
\int_{b_{i+1,+}(t)}^{b_{i,-}(t)} \psi_i(y-x) \, dx \ge \frac{1}{2\kappa_i} 
\,,
\]
and in view of inequality \cref{integral_of_fff_i_dominates_integral_of_F_i_step} and since $F^\ddag_i(x,t)$ is not smaller than $\coeffEn E^\ddag(x,t)$, the right inequality of \cref{integral_of_fff_i_dominates_integral_of_energy} follows. \Cref{lem:integral_of_fff_i_dominates_integral_of_energy} is proved. 
\end{proof}
\begin{proof}[End of the proof of \cref{lem:no_energy_besides_SigmaBulk,prop:eeeAsympt_of_u_equals_eee_of_ttt}]
Since $\Lext(t)$ goes to $+\infty$ as $t$ goes to $+\infty$, it follows from inequality \cref{upper_bound_integral_fff_i} of \cref{cor:upper_bound_integral_fff_i} and from inequalities \cref{integral_of_fff_i_dominates_integral_of_energy} of \cref{lem:integral_of_fff_i_dominates_integral_of_energy} that the centre quantity in inequalities \cref{integral_of_fff_i_dominates_integral_of_energy} goes to $0$ as $t$ goes to $+\infty$, and \cref{lem:no_energy_besides_SigmaBulk} follows. 
In view of the limit \cref{cv_asympt_energy_to_get_its_value}, \cref{prop:eeeAsympt_of_u_equals_eee_of_ttt} follows. 
\end{proof}
In view of \cref{prop:eeeAsympt_of_u_equals_eee_of_ttt}, statement \cref{item:thm_main_value_asymptotic_energy} of \cref{thm:main} is proved, and the proof of \cref{thm:main} is complete. 
\section{Existence results for stationary solutions and basin of attraction of a stable homogeneous solution}
\label{sec:exist_res_bas_att}
The aim of this \namecref{sec:exist_res_bas_att} is to recover standard results concerning existence of homoclinic or heteroclinic stationary solutions and the basin of attraction of a stable homogeneous solution, as direct consequences of \cref{prop:scs_asympt_en} (upper semi-continuity of the asymptotic energy) and \cref{thm:main}. These results are stated as four independent corollaries. The proofs are given after the four statements. Elementary examples illustrating these results will be discussed in the next \namecref{sec:examples}.

As everywhere else, let us consider a function $V$ in $\ccc^2(\rr^d,\rr)$ satisfying hypothesis \cref{hyp_coerc}, and let us denote by $V_{\min}$ the global minimum value of $V$:
\[
V_{\min} = \min_{u\in\rr^d} V(u)
\,.
\]
\subsection{Existence results for stationary solutions}
The following two corollaries deal with the stationary solutions of system~\cref{init_syst}, and are variants of well-known results, usually obtained by calculus of variation techniques (minimization or  mountain-pass arguments, see references below). 
\subsubsection{Global minimum level set}
The following ``minimization'' corollary is illustrated by cases (a) and (b) of \vref{fig:shape_pot}. It is similar to (or contained in) results going back to the early nineties (see P. Rabinowitz \cite{Rabinowitz_periodicHeteroclinicOrbits_1989} and P. Sternberg \cite{Sternberg_vectorValuedLocMin_1991} and for instance N. Alikakos and G. Fusco \cite{AlikakosFusco_connectionPbSevGlobMin_2008} for recent results and additional references). It is by the way implicitly contained in Theorem~3 of Béthuel, Orlandi, Smets \cite{BethuelOrlandi_slowMotion_2011}.
\begin{corollary}[existence of a chain of heteroclinic stationary solutions]
\label{cor:nonneg_stat}
Assume that $V$ satisfies hypothesis \cref{hyp_coerc}. Assume furthermore that:
\begin{itemize}
\item hypothesis \textup{(\hyperlink{hypOnlyMin}{\hypOnlyMinRef{V_{\min}}})} holds; in other words every global minimum point of $V$ is nondegenerate;
\item there is more than one global minimum point of $V$; in other words the cardinal of $\mmm_{V_{\min}}$ is larger than $1$. 
\end{itemize}
Then, for every ordered pair $(m_-,m_+)$ in $\mmm_{V_{\min}}^2$ such that $m_-$ differs from $m_+$, there exist a positive integer $q$, and $q-1$ \emph{distinct} minimum points $m_1,\dots,m_{q-1}$ in $\mmm_{V_{\min}}$ such that, if $m_-$ is denoted by $m_0$ and $m_+$ by $m_q$, then for every integer $i$ in $\{0,\dots,q-1\}$, the set $\Phi_0(m_i,m_{i+1})$ is nonempty. In other words, there exists a ``chain'' of bistable stationary solutions connecting $m_-$ to $m_+$.
\end{corollary}
\subsubsection{Local minimum level set}
The following ``mountain pass'' corollary is illustrated by cases (c), (d), and (e) of \vref{fig:shape_pot}. It is similar to (or contained in) results going back the early nineties (see A. Ambrosetti and M. L. Bertotti \cite{AmbrosettiBertotti_homoclinicsSecondOrder_1992}, Bertotti \cite{Bertotti_homoclinicsLagrangian_1992}, and Rabinowitz and K. Tanaka \cite{RabinowitzTanaka_connectingOrbits_1991}).
\begin{corollary}[existence of a homoclinic stationary solution]
\label{cor:neg_stat}
Assume that $V$ satisfies hypothesis \cref{hyp_coerc} and let $m$ be a point in $\mmm$. Let us assume that:
\begin{itemize}
\item $m$ is not a global minimum point of $V$; in other words $V_{\min}<V(m)$;
\item and there is no critical point other than $m$ in the level set $V^{-1}(\{V(m)\})$. 
\end{itemize}
Then there exists at least one nonconstant stationary solution that is homoclinic to $m$. In other words, the set $\Wu(m,0)\cap \Ws(m,0)$ is nonempty, or in other words the set $\Phi_0(m,m)$ is nonempty. 
\end{corollary}
\subsection{Basin of attraction of a stable homogeneous stationary solution}
The next two corollaries can be viewed as ``dynamical'' versions of the two previous ones. They require the following notation. 
\begin{notation}
If $m$ is point in $\mmm$, let $\basatt(m)$ denote the basin of attraction (for the semi-flow of system~\cref{init_syst}) of the homogeneous equilibrium $m$: 
\[
\basatt(m)=\bigl\{ u_0\in X: (S_t u_0)(x) \to m \text{ , uniformly with respect to } x, \text{ as } t\to+\infty \bigr\} \,,
\]
and let $\partial\basatt(m)$ denote the topological border, in $X$, of $\basatt(m)$. 
\end{notation}
\subsubsection{Global minimum point}
\cref{cor:nonneg_dyn} below applies to example (c) of \vref{fig:shape_pot}. As \cref{cor:nonneg_stat} above, it is implicitly contained in \cite[Theorem~3]{BethuelOrlandi_slowMotion_2011}.
\begin{corollary}[global stability of the unique global minimum point]
\label{cor:nonneg_dyn}
Assume that $V$ satisfies hypothesis \cref{hyp_coerc}. Assume furthermore that:
\begin{itemize}
\item the potential $V$ has a unique global minimum point, and this minimum point (denoted by $m$) is nondegenerate;
\item and there exists no nonconstant stationary solution homoclinic to $m$; in other words the set $\Phi_0(m,m)$ is reduced to the function identically equal to $m$, or in other words the set $\Wu(m,0)\cap \Ws(m,0)$ is empty.
\end{itemize}
Then every bistable solution connecting $m$ to $m$ converges to $m$, uniformly in space, as time goes to $+\infty$. In other words, 
\[
\Xbist(m,m) = \basatt(m).
\]
\end{corollary}
\subsubsection{Local minimum point}
\cref{cor:neg_dyn} below applies to cases (c), (d), and (e) of \vref{fig:shape_pot}, and is analogous in spirit to results of author's previous paper \cite{Risler_borderBasinsAtt_2000}. It is somehow related to the huge amount of existing literature about (codimension one) threshold phenomena in reaction-diffusion equations, going back (at least) to Fife's paper \cite{Fife_longTimeBistable_1979} of 1979 and the contributions of G. Flores in the late eighties \cite{Flores_stabManStandWave_1989}. Other references about this subject can be found in the recent paper \cite{MuratovZhong_thresholdSymSol_2013} of Muratov and Zhong, where various threshold results of the same kind are obtained. The arguments used by these authors are based on the energy function~\vref{form_en}, and are quite close in essence (although applied in a different setting limited to the scalar case $d$ equals $1$) to those of the present paper. 
\begin{corollary}[attractor of the border of the basin of attraction of a local minimum point]
\label{cor:neg_dyn}
Assume that $V$ satisfies hypothesis \cref{hyp_coerc}. Let $m$ be a point in $\mmm$ and let us write $\valueOfV=V(m)$. Let us assume that:
\begin{itemize}
\item $m$ is not a global minimum point of $V$; in other words $V_{\min}<\valueOfV$;
\item and hypothesis \textup{(\hyperlink{hypOnlyMin}{\hypOnlyMinRef{\valueOfV}})} holds; in other words,  every critical point in the level set $V^{-1}(\{\valueOfV\})$ is a nondegenerate minimum point. 
\end{itemize}
Then the following conclusions hold. 
\begin{enumerate}
\item There exists at least one bistable initial condition connecting $m$ to himself and belonging to the border of the basin of attraction of the spatially homogeneous equilibrium $m$; in other words, the set
\begin{equation}
\label{border_of_basin_of_attraction}
\partial\basatt(m)\cap \Xbist(m,m)
\end{equation}
is nonempty.
\label{item:first_conclusion_cor_attractor_border_basin_attraction}
\item Every solution $(x,t)\mapsto u(x,t)$ of system~\cref{init_syst} in this nonempty set \cref{border_of_basin_of_attraction} has a positive asymptotic energy, and for every such solution, the quantities
\[
\sup_{x\in\rr} \ \abs{u_t(x,t)}
\quad\text{and}\quad 
\sup_{x\in\rr} \ \dist\biggl(\Bigl(u(x,t),u_x(x,t)\Bigr) \,, \,I\bigl(\Phi_0(\valueOfV)\bigr)\biggr)
\]
go to $0$ as time goes to $+\infty$.
\label{item:second_conclusion_cor_attractor_border_basin_attraction}
\end{enumerate}
\end{corollary}
\begin{remark}
Assume that the potential $V$ has a unique global minimum point $m$, which is non degenerate (the first bullet hypothesis of \cref{cor:nonneg_dyn}). 
\begin{itemize}
\item If furthermore $d$ equals $1$ (scalar case), then the set $\PhiZero(m,m)$ is necessarily empty (indeed every solution of the Hamiltonian system~\vref{ham_ord_1} in the unstable manifold of $(m,0)$ goes to infinity as time goes to $+\infty$, since the velocity can never vanish). As a consequence, the conclusions of \cref{cor:nonneg_dyn} hold: every bistable solution connecting $m$ to $m$ converges to $m$, uniformly with respect to the space coordinate, as time goes to $+\infty$.
\item The situation is quite different in the vector case $d$ larger than $1$, where nonconstant stationary solutions homoclinic to a unique global minimum point might very well exist. Here is an example (the parameter $\varepsilon$ is a small positive quantity): 
\[
V:\rr^2\to\rr,
\quad
(u_1,u_2)\mapsto -\frac{u_1^2+u_2^2}{2} + \frac{(u_1^2+u_2^2)^2}{4} - \varepsilon u_1
\,.
\]
For additional information and comments see P. Coullet \cite{Coullet_locPattFronts_2002}.
\end{itemize}
\end{remark}
\subsection{Proof of \texorpdfstring{\cref{cor:nonneg_stat,cor:nonneg_dyn}}{Corollaries \ref{cor:nonneg_stat} and \ref{cor:nonneg_dyn}} (global minimum)}
Let us assume that $V$ satisfies hypotheses \cref{hyp_coerc} and \textup{(\hyperlink{hypOnlyMin}{\hypOnlyMinRef{V_{\min}}})}. The asymptotic energy of every bistable initial condition in $\Xbist(\mmm_{V_{\min}})$ is nonnegative (as a limit of nonnegative quantities), therefore the corresponding solution must converge towards the set $I\bigl(\PhiZero(V_{\min})\bigr)$, as stated in conclusion \cref{item:thm_main_nonnegative_asympt_energy_approach_bistable_stationary} of \cref{thm:main}. 

Let us assume that the set $\mmm_{V_{\min}}$ of (global) minimum points of $V$ is not reduced to a singleton, and let $m_-$ and $m_+$ be two distinct points in this set. According to \cref{cor:bist} the set $\Xbist(m_-,m_+)$ of bistable initial conditions connecting these two points is nonempty, and according to conclusion \cref{item:thm_main_nonnegative_asympt_energy_approach_bistable_stationary} of \cref{thm:main} applied to a solution in this set, the set $I\bigl(\PhiZero(V_{\min})\bigr)$ must connect the two points $(m_-,0)$ and $(m_+,0)$ in $\rr^{2d}$. This ensures the existence of a ``chain'' of heteroclinic stationary solutions connecting $m_-$ to $m_+$. \Cref{cor:nonneg_stat} is thus proved.

If conversely the set $\mmm_{V_{\min}}$ is reduced to a single point $m$ and if there is no nonconstant stationary solution homoclinic to $m$, then the set $I\bigl(\PhiZero(V_{\min})\bigr)$ must be reduced to the singleton $\{(m,0)\}$. In this case conclusion \cref{item:thm_main_nonnegative_asympt_energy_approach_bistable_stationary} of \cref{thm:main} shows that every bistable solution connecting $m$ to himself must converge towards $m$, uniformly in space, as time goes to $+\infty$. This proves \cref{cor:nonneg_dyn}.
\subsection{Proof of \texorpdfstring{\cref{cor:neg_stat,cor:neg_dyn}}{Corollaries~\ref{cor:neg_stat} and~\ref{cor:neg_dyn}} (local minimum)}
Let us assume that $V$ satisfies \cref{hyp_coerc}, let $m$ denote a point in $\mmm$, let $\valueOfV$ denote the quantity $V(m)$, and let us assume that $V_{\min}$ is less than $\valueOfV$. Let $m_{\min}$ denote a point of $\rr^d$ where $V$ reaches its global minimum, in other words such that $V(m_{\min})$ equals $V_{\min}$. Let $\chi:\rr\to[0,1]$ be a smooth cutoff function satisfying
\[
\chi(x)=1 \text{ for all } x \text{ in } (-\infty,0] \text{ and } \chi(x)=0 \text{ for all } x \text{ in } [1+\infty) \,.
\]
For every positive quantity $L$, let us introduce the function $u_{0,L}:\rr^d\to\rr$ defined as
\[
u_{0,L}(x)=\left\{
\begin{aligned}
m+\chi(x-L)(v-m)\quad & \text{for}\quad x\ge0 \\
u_{0,L}(-x)\quad & \text{for}\quad x\le0 
\,,
\end{aligned}
\right.
\]
see \cref{fig:graph_u1},
\begin{figure}[!htbp]
\centering
\includegraphics[width=.8\textwidth]{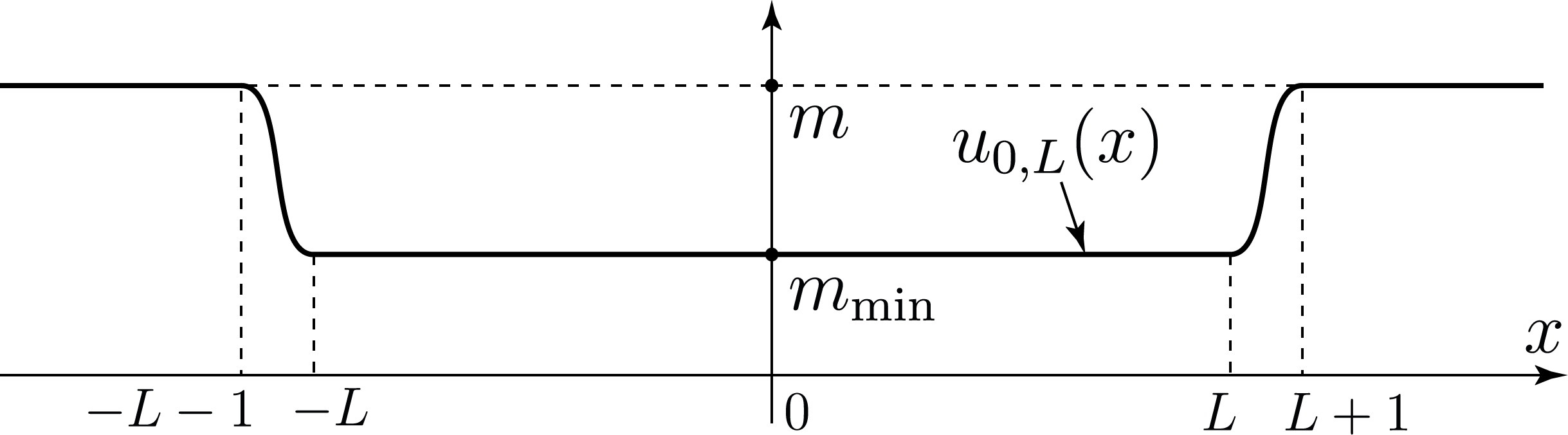}
\caption{Graph of the function $x\mapsto u_{0,L}(x)$.}
\label{fig:graph_u1}
\end{figure}
and, for every $s$ in $[0,1]$, let us introduce the function $u_{0,L,s}:\rr^d\to\rr$ defined as
\[
u_{0,L,s} = (1-s)m + s(u_{0,L}-m)
\,.
\]
The following observations can be made. 
\begin{itemize}
\item According to \vref{lem:sufficient_condition_stability_right_end_of_space}, every function $u_{0,L,s}$ (for every positive quantity $L$ and every $s$ in $[0,1]$) belongs to $\Xbist(m,m)$ (it is a bistable initial condition connecting $m$ to itself).
\item The function $u_{0,L,0}$ is identically equal to $m$, it is thus a (stable, homogeneous) stationary solution of system \cref{init_syst}, and it belongs to the basin of attraction $\basatt(m)$.
\item Since $V_{\min}<V(m)$, the quantity
\[
\int_{-\infty}^{+\infty} \Bigl(\frac{1}{2}\abs{ u_{0,L}'(x)}_{\ddd} ^2+V\bigl(u_{0,L}(x)\bigr)-V(m)\Bigr)\, dx 
\]
(the energy of the bistable initial condition $u_{0,L}$) goes to $-\infty$ as $L$ goes to $+\infty$. In view of the upper bound \cref{dt_en_loc_real} of \cref{lem:loc_en_almost_decr} on the time derivative of localized energy, this shows that, for $L$ large enough, the asymptotic energy of the solution corresponding to the initial condition $u_{0,L}$ is negative, and as a consequence $u_{0,L}$ is \emph{not} in $\basatt(m)$. 
\item The function $u_{0,L,1}$ is equal to $u_{0,L}$. 
\item The function $[0,1]\to X$, $s\mapsto u_{0,L,s}$ is continuous (for the usual $\HoneulAlone$-norm on $X$). 
\end{itemize}
Let us pick a positive quantity $L$, large enough so that $u_{0,L}$ is not in $\basatt(m)$. Then it follows from the observations above that there must exist a quantity $\paramthres$ in $(0,1]$ such that the function (bistable initial condition) $u_{0,\paramthres}$ is in the set $\partial\basatt(m)$ (the topological border of $\basatt(m)$ in $X$). In particular the set $\partial\basatt(m)\cap\Xbist(m,m)$ is nonempty, which proves conclusion \cref{item:first_conclusion_cor_attractor_border_basin_attraction} of \cref{cor:neg_dyn}. 

Now, since according to \vref{prop:scs_asympt_en} the asymptotic energy of a solution is upper semi-continuous with respect to that solution, every initial condition in $\partial\basatt(m)$ must have a nonnegative asymptotic energy, and more precisely, according to \vref{lem:lag_sig_far_pos}, a positive asymptotic energy. According to conclusion \cref{item:thm_main_nonnegative_asympt_energy_approach_bistable_stationary} of \cref{thm:main}, every solution in $\partial\basatt(m)$ must then approach the set $I\bigl(\PhiZero(\valueOfV)\bigr)$ as time goes to $+\infty$. It follows that this set is not reduced to the point $(m,0)$, or else such a solution would approach $m$ uniformly in space and would thus belong to $\basatt(m)$ and not its border, a contradiction. This proves conclusion \cref{item:second_conclusion_cor_attractor_border_basin_attraction} of \cref{cor:neg_dyn}. \Cref{cor:neg_dyn} is proved.

If moreover the set $\mmm_{V(m)}$ is reduced to the singleton $\{m\}$, then it follows that there exists at least one nonconstant stationary solution that is homoclinic to $m$, and this proves \cref{cor:neg_stat}. 
\subsection{Extensions}
As shown by \cref{cor:nonneg_stat,cor:neg_stat,cor:nonneg_dyn,cor:neg_dyn}, the properties of the semi-flow of system \cref{init_syst} provide an alternate approach to results usually obtained by calculus of variation techniques. The results stated above are nothing but elementary examples, but the same approach might be relevant to more recent results, as for instance the existence of non-minimizing connections proved in \cite{OliverBonafoux_nonMinConnOrbitsMultiWell_2022}. 
\section{Examples} 
\label{sec:examples}
This \namecref{sec:examples} is devoted to a discussion on elementary examples in the scalar case $d$ equals $1$, corresponding to the potentials illustrated on \vref{fig:shape_pot}. In all these examples the value $\valueOfV$ of the potential at the equilibria approached at both ends of $\rr$ by the bistable solutions considered is equal to $0$, and hypotheses \cref{hyp_coerc} and \textup{(\hyperlink{hypOnlyMin}{\hypOnlyMinRef{0}})} and \textup{(\hyperlink{hypDiscStat}{\hypDiscStatRef{0}})} are satisfied. 
\subsection{Allen--Cahn equation}
The equation reads (see example (a) of \cref{fig:shape_pot}): 
\[
u_t=u-u^3+u_{xx}=-V'(u)+u_{xx}
\quad\text{where}\quad
V(u)=1/4-u^2/2+u^4/4 \,.  
\]
In this example the set $\mmm_{0}$ is made of the two points $-1$ and $1$, and the set $\PhiZero(0)$ consists of:
\begin{itemize}
\item the ``kink'' solution $x\mapsto \tanh (x/\sqrt{2})$, 
\item and the ``antikink'' solution $x\mapsto -\tanh (x/\sqrt{2})$
\end{itemize}
(and their translates with respect to $x$). According to \cref{thm:main}, for every initial condition $u_0$ in $\Xbist(\pm1,\pm1)$, the solution $S_t u_0$ approaches, as time goes to $+\infty$, a standing terrace involving a finite number of alternatively kink and antikink solutions, getting slowly away from one another. 

Since the long-range interaction between two consecutive kink and antikink solutions is attractive, the following more precise result actually holds. In the sentences below, ``approaches'' means ``approaches as time goes to $+\infty$, uniformly with respect to $x$ in $\rr$''. 
\begin{itemize}
\item If $u_0$ is in $\Xbist(-1,-1)$, then $S_t u_0$ approaches $-1$.
\item If $u_0$ is in $\Xbist(+1,+1)$), then $S_t u_0$ approaches $+1$. 
\item If $u_0$ is in $\Xbist(-1,+1)$, then there exists $x_0\in\rr$ such that $S_t u_0$ approaches the single kink $x\mapsto\tanh\bigl((x-x_0)/\sqrt{2}\bigr)$.
\item If $u_0$ is in $\Xbist(+1,-1)$, then there exists $x_0\in\rr$ such that $S_t u_0$ approaches the single kink $x\mapsto\tanh\bigl((x_0-x)/\sqrt{2}\bigr)$. 
\end{itemize}
This result is implicit in many papers since this Allen--Cahn model is the simplest exhibiting this kind of long-range interaction, and consequently has been the most studied, see for instance \cite{Ei_motionPulses_2002,BethuelSmets_motionLawFrontsScalarRDEquEqualDepthMultWellPot_2017} (where other references can be found). 
\subsection{Over-damped sine--Gordon equation}
The equation reads (see example (b) of \cref{fig:shape_pot}): 
\[
u_t=-\sin u+u_{xx}=-V'(u)+u_{xx} 
\quad\text{where}\quad
V(u)=-\cos u+1 \,.  
\]
In this example the set $\mmm_0$ is $2\pi\zz$. Stationary solutions connecting  equilibria in this set are: a ``kink'' connecting $0$ to $2\pi$, an ``antikink'' connecting $2\pi$ to $0$, their translates with respect to $x$, and their $2\pi\zz$-translates with respect to $u$. 

According to the maximum principle, for every ordered pair $(q_-,q_+)$ in $\zz^2$ and every initial condition $u_0$ in $\Xbist(2\pi q_-,2\pi q_+)$, the corresponding solution is bounded, and therefore all conclusions of \cref{thm:main} hold (the potential can be changed without changing the solution in order hypothesis \cref{hyp_coerc} to be satisfied). According to these conclusions, the solution converges, as $t\to+\infty$, towards a standing terrace involving a finite number of kinks and antikinks, getting slowly away from one another. 

Again, since the long-range interaction between two consecutive kink and antikink solutions is attractive, this standing terrace actually involves either $q_+-q_-$ kinks (if $q_+$ is larger than $q_-$), or $q_--q_+$ antikinks (if $q_-$ is larger than $q_+$), or is reduced to the homogeneous equilibrium $q_+$ if $q_+$ and $p_-$ are equal \cite{BethuelSmets_motionLawFrontsScalarRDEquEqualDepthMultWellPot_2017}. 
\subsection{Nagumo equation}
The equation reads (see example (c) of~\cref{fig:shape_pot}): 
\[
u_t=-u(u-a)(u-1)+u_{xx}=-V'(u)+u_{xx}
\]
where 
\[
V(u)=a\frac{u^2}{2} - (a+1)\frac{u^3}{3} + \frac{u^4}{4}
\quad\text{and}\quad
0<a<1/2
\,.
\]
In this case the set $\mmm_0$ is reduced to the minimum point $0$, the bistable potential $V$ reaches its global minimum at $1$ (thus $V(1)$ is negative), 
and the set $\PhiZero(0)$ is reduced to a single stationary solution (``ground state'') $x\mapsto\phiGround(x)$ homoclinic to $0$ (and its translates with respect to $x$). A Sturm--Liouville argument shows that this solution has one dimension of instability. 

According to \vref{cor:neg_dyn}, the set $\partial\basatt(0)\cap \Xbist(0,0)$ is nonempty, and, for every initial condition $u_0$ in this set, the asymptotic energy $\eeeAsympt[u_0]$ is positive. Thus, all conclusions of \cref{thm:main} hold for this initial condition: the corresponding solution approaches a standing terrace involving a finite (nonzero) number of translates of $\phiGround$, getting slowly away from one another, as time goes to $+\infty$. 

Once again, the long-range interaction between two consecutive translates of $\phiGround$ is attractive (\cite{Ei_motionPulses_2002,MielkeZelik_multiPulseEvolutionAndSpaceTimeChaosDissipativeSystems_2009,BethuelSmets_motionLawFrontsScalarRDEquEqualDepthMultWellPot_2017}), therefore there should actually be only one translate of $\phiGround$ in the standing terrace. Thus, there should exist $x_0$ in $\rr$ such that this solution approaches the translate $x\mapsto \phiGround(x-x_0)$ of $\phiGround$, uniformly with respect to $x$, as time goes to $+\infty$. This conclusion should follow from a combination of the arguments of \cite{MielkeZelik_multiPulseEvolutionAndSpaceTimeChaosDissipativeSystems_2009,BethuelSmets_motionLawFrontsScalarRDEquEqualDepthMultWellPot_2017}, but to the knowledge of the author a detailed proof along such lines is still missing. However, this conclusion has actually been recently proved by Matano and Poláčik (\cite[Theorem~2.5]{MatanoPolacik_dynNonnegSolOneDimRDI_2016} and \cite[Theorem~2.5]{MatanoPolacik_dynNonnegSolOneDimRDII_2020}) by a completely different approach based on the zero number of the solution. Note that in this example the stable manifold of the stationary solution $\phiGround$ is the border of the basin of attraction of the ``metastable'' homogeneous equilibrium $0$ (this has been stated by many authors for a long time, see for instance \cite{Flores_stabManStandWave_1989,Risler_borderBasinsAtt_2000}). 

Similar conclusions can be drawn about the \emph{over-damped sine--Gordon equation with constant forcing} (see example (d) of \cref{fig:shape_pot}): 
\[
u_t=-\sin u+\Omega+u_{xx} 
\quad\text{with}\quad
0<\Omega<1 \,. 
\]
\subsection{``Subcritical'' Allen--Cahn equation}
The equation reads (see example (e) of \cref{fig:shape_pot}): 
\[
u_t=-u+u^3-\epsilon u^5 +u_{xx}=-V'(u)+u_{xx} 
\quad\text{where}\quad
V(u) = \frac{u^2}{2} - \frac{u^4}{4} + \varepsilon \frac{u^6}{6} 
\,,
\]
and where $\varepsilon$ is a small positive quantity, the last term of the potential being there just to ensure coercivity. In this example the set $\mmm_0$ is reduced to the minimum point $\{0\}$, and the set $\PhiZero(0)$ is made of two stationary solutions homoclinic to $0$, say $h_+$ (taking positive valuers) and $h_-$ (taking negative values), and their translates with respect to $x$. 

For every initial condition $u_0$ in $\partial\basatt(0)\cap \Xbist(0,0)$ such that the corresponding solution is bounded (uniformly in $x$ and $t$), the asymptotic energy $\eeeAsympt[u_0]$ is positive and all conclusions of \cref{thm:main} hold, that is the solution converges towards a standing terrace involving a finite (nonzero) number of translates of $h_+$ and $h_-$, getting slowly away from one another. 

Once again, the long-range interaction between two consecutive translates of $h_+$ or two consecutive translates of $h_-$ is attractive \cite{Ei_motionPulses_2002,MielkeZelik_multiPulseEvolutionAndSpaceTimeChaosDissipativeSystems_2009,BethuelSmets_motionLawFrontsScalarRDEquEqualDepthMultWellPot_2017}, and therefore such two consecutive translates of the same stationary solution should not take place in the asymptotic terrace. Again in this case, this conclusion should follow from a combination of the arguments of \cite{MielkeZelik_multiPulseEvolutionAndSpaceTimeChaosDissipativeSystems_2009,BethuelSmets_motionLawFrontsScalarRDEquEqualDepthMultWellPot_2017} or from the zero number argument of \cite{MatanoPolacik_dynNonnegSolOneDimRDI_2016,MatanoPolacik_dynNonnegSolOneDimRDII_2020} (see the proof of Theorem~2.5 in each of these two references); to the best knowledge of the author however, a detailed rigorous proof of this conclusion is still missing. 
\section{Attracting ball for the semi-flow}
\label{sec:att_ball}
This \namecref{sec:att_ball} presents strong similarities with \cite[Appendix~A.1]{Risler_globCVTravFronts_2008} and \cite[Section~2]{GallayJoly_globStabDampedWaveBistable_2009}, although the hypotheses and presentation are slightly different. Note also that if the diffusion matrix $\ddd$ equals identity then the square norm of a solution obeys a maximum principle, leading to a simpler proof for global existence of solutions and existence of an attracting ball for the $\Linfty$-norm, see \cite[\InvasionRelaxationPropAttBall]{Risler_noInvasionCaseHigherSpace_2020} (however this argument does not seem to work if $\ddd$ is not the identity matrix). 
\subsection{Attracting ball in \texorpdfstring{$X$}{X}}
Recall that $X$ denotes the space $\Honeul$ (see \vref{subsec:funct_fram}).
\subsubsection{Statement}
\begin{proposition}[global existence of solutions and attracting ball in $X$]
\label{prop:glob_exist_sol_att_ball_X} 
Assume that hypothesis \cref{hyp_coerc} holds for the potential $V$. Then, for every function $u_0$ in $X$, the solution $t\mapsto S_t u_0$ of system~\cref{init_syst} with initial condition $u_0$ is defined up to $+\infty$ in time. In addition, there exist
\begin{itemize}
\item a positive quantity $\RattX$ (``radius of attracting ball for the $X$-norm''), 
\item and a positive quantity $\RmaxX[u_0]$ (``radius of maximal excursion for the $\HoneulAlone$-norm''),
\item and a positive quantity $\Tatt[u_0]$ (``delay to enter attracting ball''),
\end{itemize}
such that
\[
\begin{aligned}
\sup_{t\ge0} \norm{x\mapsto(S_t u_0)(x)}_X &\le \RmaxX[u_0] \\
\text{and}\qquad
\sup_{t\ge \Tatt[u_0]}\norm{x\mapsto(S_t u_0)(x)}_X &\le \RattX
\,.
\end{aligned}
\]
The quantity $\RattX$ depends only on $V$ and $\ddd$, whereas $\RmaxX[u_0]$ and $\Tatt[u_0]$ depend also on $\norm{u_0}_X$.
\end{proposition}
\subsubsection{Assumptions and notation for the coercivity at infinity}
According to hypothesis \cref{hyp_coerc}, there exist positive quantities $\epsCoerc$ and $\Kcoerc$ such that, for all $u$ in $\rr^d$, 
\begin{equation}
\label{coercivity_at_infty}
u\cdot \nabla V(u)\ge \epsCoerc u^2-\Kcoerc
\,.
\end{equation}
\subsubsection{Attracting ball in \texorpdfstring{$\LTwoul$}{L2ul(R,Rd)}}
First let us make an observation, besides of the proof: with the notation of \cref{subsec:1rst_ord}, expression~\vref{ddt_loc_L2} (time derivative of a localized $L^2$ function) yields, for every (nonnegative) weight function $x\mapsto\psi(x)$ in $W^{2,1}\bigl(\rr,[0,+\infty)\bigr)$, 
\begin{equation}
\label{dt_l2_att_ball}
\frac{d}{dt}\int_{\rr} \frac{\psi}{2}u^2 \, dx \le \int_{\rr} \Bigl[ \psi  \bigl( -\epsCoerc u^2 + \Kcoerc \bigr) +  \frac{\psi''}{2}\abs{u}_{\ddd}^2 \Bigr] \, dx
\,.
\end{equation}
Thus, if the weight function $\psi$ is such that $\eigDmax\psi''$ is less than or equal to $\epsCoerc\psi$, for instance:
\[
\psi(x) = \exp\Bigl(-\sqrt{\frac{\epsCoerc}{\eigDmax}}\abs{x-x_0}\Bigl)
\,,
\]
then inequality~\cref{dt_l2_att_ball} abode yields
\[
\frac{d}{dt} \int_{\rr} \frac{\psi}{2}u^2 \, dx \le -\frac{\epsCoerc}{2} \int_{\rr} \psi \, u^2 \, dx + \Kcoerc \int_{\rr} \psi  \, dx 
\,.
\]
Provided that the semi-flow is global, this inequality ensures the existence of an attracting ball in in the uniformly local Sobolev space $\LTwoul$. The proof that the semi-flow is indeed global and that there is an attracting ball in $X$ will by contrast require a combination of both localized energy and $L^2$-norm. 
\subsubsection{Proof}
\paragraph{Set-up.}
Hypothesis \cref{hyp_coerc} guarantees that $V$ is bounded from below on $\rr^d$; let us write, for all $u$ in $\rr^d$, 
\[
V_0(u)=V(u)-\min_{v\in\rr^d}V(v)
\,;
\quad\text{thus,}\quad
\min_{u\in\rr^d}V_0(u)=0
\,.
\]
Take $u_0$ in $X$ and let 
\[
u:\rr^d\times[0,T_{\max}),\quad (x,t)\mapsto u(x,t)=(S_t u_0)(x)
\]
denote the (maximal) solution of system~\cref{init_syst} with initial condition $u_0$, where $T_{\max}$ in $(0,+\infty]$ denotes the upper bound of the (maximal) time interval where this solution is defined.
\paragraph{Functionals.}
The quantity $\kappa_0$ and functions $\psi_0$ and $\fff_0$ defined below will play similar roles as the quantity $\kappa$ and the functions $\psi$ and $\fff$ that were defined in \cref{subsec:firewall}. Since the definitions below slightly differ from those of \cref{subsec:firewall}, the subscript ``$0$'' is added to avoid confusion and to recall that these new objects are related to the ``normalized'' potential $V_0$. Let $\kappa_0$ be a positive quantity, small enough so that 
\begin{equation}
\label{conditions_on_kappaZero}
\kappa_0^2 \, \frac{\eigDmax}{2}\le\frac{\epsCoerc}{2}
\quad\text{and}\quad
\kappa_0^2 \, \eigDmax\le 2
\end{equation}
(those are the conditions that yield inequality~\cref{dt_fire_att_b} below); it may, for instance, be chosen as
\[
\kappa_0 = \sqrt{\frac{\min(2,\epsCoerc)}{\eigDmax}}
\,.
\]
Let us introduce the  weight function $\psi_0$ defined as
\[
\psi_0(x) = \exp(-\kappa_0\abs{x})
\,,
\]
and, for all $t$ in $[0,T_{\max})$ and $\widebar{x}$ in $\rr$, let 
\[
\begin{aligned}
\fff_0(\widebar{x},t) &= \int_{\rr} T_{\widebar{x}}\psi_0(x)\Bigl( \frac{1}{2}\abs{u_x(x,t)}_{\ddd} ^2 + V_0\bigl(u(x,t)\bigr) + \frac{1}{2}u(x,t)^2 \Bigr) \, dx \,, \\
\mathcal{Q}(\widebar{x},t) &= \int_{\rr} T_{\widebar{x}}\psi_0(x)\Bigl( \frac{1}{2}\abs{u_x(x,t)}_{\ddd} ^2 + \frac{1}{2}u(x,t)^2 \Bigr) \, dx \,,
\end{aligned}
\]
where $T_{\widebar{x}}\psi_0(x)$ is defined as in \cref{subsec:firewall}.
The definition of $V_0$ ensures that
\[
\mathcal{Q}(\widebar{x},t) \le \fff_0(\widebar{x},t)
\,.
\]
\paragraph{Decrease of \texorpdfstring{$\fff_0$}{F0} where \texorpdfstring{$\qqq$}{Q} is large.}
According to the generic expressions~\cref{ddt_loc_en} and~\cref{ddt_loc_L2} of \cref{subsec:1rst_ord}, the function $\fff_0$ is expected to decrease with time, at least --- because of the coercivity hypothesis \cref{hyp_coerc} --- where $u(x,t)$ is large (this decrease will be used to control the function $\mathcal{Q}$). This is formalized by the next lemma. 
\begin{lemma}[$\fff_0$ decreases where $\qqq$ is large]
\label{lem:decrease_of_fff_zero_where_Q_is_large}
There exists a (positive) quantity $\Qfiredecr$, depending only on $V$ and $\ddd$, such that, for all $t$ in $[0,T_{\max})$ and $\widebar{x}$ in $\rr$, 
\begin{equation}
\label{equ:decrease_of_fff_zero_where_Q_is_large}
\Qfiredecr \le \mathcal{Q}(\widebar{x},t)   \implies \partial_t \fff_0(\widebar{x},t) \le -1
\,.
\end{equation}
\end{lemma}
\begin{proof}
According to expressions\vref{ddt_loc_en,ddt_loc_L2} (time derivatives of localized energy and $L^2$ functionals), for all $t$ in $[0,T_{\max})$ and $\widebar{x}$ in $\rr$, 
\[
\begin{aligned}
\partial_t \fff_0(\widebar{x},t) \le {} & \int_{\rr} T_{\widebar{x}}\psi_0(x) \Bigl( -u_t^2 + \kappa_0\abs{\ddd u_x \cdot u_t} - \epsCoerc u^2 + \Kcoerc - \abs{u_x}_{\ddd}^2 + \frac{\kappa_0^2}{2} \abs{u}_{\ddd}^2 \Bigr) \, dx \\
\le {} & \Kcoerc \int_{\rr} \psi_0 (x)\, dx  + \int_{\rr} T_{\widebar{x}}\psi_0(x) \Bigl(-\frac{1}{2}\epsCoerc u^2 - \frac{1}{2}\abs{u_x}_{\ddd}^2\Bigr) \, dx \\
& + \int_{\rr} T_{\widebar{x}}\psi_0(x) \Bigl( -u_t^2 + \kappa_0\abs{\ddd u_x \cdot u_t} - \frac{1}{2}\abs{u_x}_{\ddd}^2 \Bigr) \, dx \\
& + \int_{\rr} T_{\widebar{x}}\psi_0(x) \Bigl(- \frac{\epsCoerc}{2}u^2 + \frac{\kappa_0^2}{2} \abs{u}_{\ddd}^2 \Bigr) \, dx \,. \\
\end{aligned}
\]
According to the conditions \cref{conditions_on_kappaZero} on $\kappa_0$, the two last integrals are negative, thus
\begin{equation}
\label{dt_fire_att_b}
\partial_t \fff_0(\widebar{x},t) \le -\min(\epsCoerc,1)\ \mathcal{Q}(\widebar{x},t)  + \frac{2\Kcoerc}{\kappa_0}
\,,
\end{equation}
and introducing the positive quantity
\[
\Qfiredecr = \frac{1}{\min(\epsCoerc,1)}\Bigl(1+ \frac{2\Kcoerc}{\kappa_0}\Bigr)
\,,
\]
inequality \cref{equ:decrease_of_fff_zero_where_Q_is_large} follows from \cref{dt_fire_att_b}. \Cref{lem:decrease_of_fff_zero_where_Q_is_large} is proved. 
\end{proof}
\paragraph{If \texorpdfstring{$\fff_0$}{F0} is large somewhere its supremum over space decreases.}
There is still a difficulty to overcome, since the functional on the left-hand side of this implication is $\mathcal{Q}(\widebar{x},t)$ --- it would be even better if it was $\fff_0(\widebar{x},t)$. And unfortunately, the fact that the quantity $\fff_0(\widebar{x},t)$ is large does not automatically ensure that $\mathcal{Q}(\widebar{x},t)$ itself is large; indeed the reason why $\fff_0(\widebar{x},t)$ is large could be that the term $V\bigl(u(x,t)\bigr)$ takes very large values (much more than $\abs{u(x,t)}^2$) far away in space from $\widebar{x}$, thus far from the bulk of the weight function $T_{\widebar{x}}\psi_0$ (see \cref{fig:pf_att_ball}). 
\begin{figure}[!htbp]
\centering
\includegraphics[width=\textwidth]{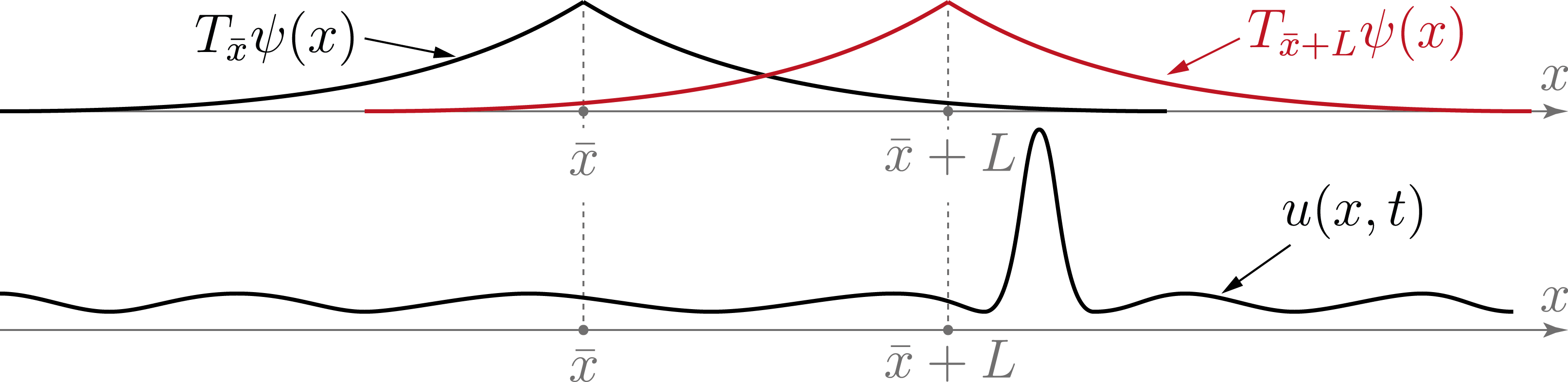}
\caption{Illustration of \cref{lem:sup_fire_higher}. If the quantity $\fff_0(\bar{x},t)$ is very large whereas the quantity $\mathcal{Q}(\bar{x},t)$ is not, this means there must be a high contribution of the potential term due to a large excursion of $u(x,t)$ far from $\bar{x}$ (to the right of $\bar{x}$ on the figure), and as a consequence $\fff_0(\cdot,t)$ reaches a higher value at $\bar{x}+L$ than at $\bar{x}$.}
\label{fig:pf_att_ball}
\end{figure}
In this case, the term $\abs{u(x,t)}^2$ in $\mathcal{Q}(\widebar{x},t)$ could count for nothing if it takes large values only far away from $\widebar{x}$. 

Hopefully, if $\fff_0(\widebar{x},t)$ is very large while $\mathcal{Q}(\widebar{x},t)$ remains below the quantity $\Qfiredecr$, this probably means that $\fff_0(\widebar{x},t)$ is (much) smaller than its supremum over all possible values of $\widebar{x}$. As a consequence, if $\fff_0(\widebar{x},t)$ is large \emph{and} close to its supremum, then the inconvenience above should not occur and $\mathcal{Q}(\widebar{x},t)$ should be large, and thus $\partial_t \fff_0(\widebar{x},t)$ should be negative. These considerations are formalized by the next lemma. 

For $t$ in $[0,T_{\max})$ let
\[
\fffZeroSup(t) = \sup_{\widebar{x}\in\rr} \fff_0(\widebar{x},t)
\]
(since the function $x\mapsto u(x,t)$ is in $X$, this quantity is finite). 
\begin{lemma}[$\qqq$ small and $\fff_0$ large means supremum of $\fff_0$ attained elsewhere]
\label{lem:sup_fire_higher}
There exists a positive quantity $\Fsuphigher$, depending (only) on $V$ and $\ddd$, such that, for all $\widebar{x}$ in $\rr$ and $t$ in $[0,T_{\max})$, 
\[
\Bigl(\mathcal{Q}(\widebar{x},t)\le \Qfiredecr \quad\text{and}\quad \fff_0(\widebar{x},t)\ge \Fsuphigher \Bigr) \implies \fffZeroSup(t) \ge \fff_0(\widebar{x},t) +1
\,.
\]
\end{lemma}
This lemma is illustrated by \cref{fig:pf_att_ball}.
\begin{proof}[Proof of \cref{lem:sup_fire_higher}]
Let $L$ be a positive quantity, large enough so that 
\[
\exp(-\kappa_0 L)\le \frac{1}{3}
\,, \quad \text{namely} \quad L = \frac{\log(3)}{\kappa_0}
\,.
\]
There exists a quantity $\Fboundloc$, depending (only) on $V$ and $\ddd$, such that, for all $\widebar{x}$ in $\rr$ and $t$ in $[0,T_{\max})$,
\[
\mathcal{Q}(\widebar{x},t) \le \Qfiredecr
\implies \int_{\widebar{x}-L}^{\widebar{x}+L}T_{\widebar{x}}\psi_0(x)\Bigl( \frac{1}{2}\abs{u_x(x,t)}_{\ddd} ^2 + V_0\bigl(u(x,t)\bigr) + \frac{1}{2}u(x,t)^2 \Bigr) \, dx \le \Fboundloc
\,.
\]
Thus, if $\mathcal{Q}(\widebar{x},t) \le \Qfiredecr$, then according to the definition of $\fff_0$ at least one of the following inequalities holds:
\begin{equation}
\label{ineq_fire_le_ri}
\begin{aligned}
\text{either} \quad & \int_{-\infty}^{\widebar{x}-L}T_{\widebar{x}}\psi_0(x)\Bigl( \frac{1}{2}\abs{u_x(x,t)}_{\ddd} ^2 + V_0\bigl(u(x,t)\bigr) + \frac{1}{2}u(x,t)^2 \Bigr) \, dx \ge \frac{1}{2}\bigl(\fff_0(\widebar{x},t)-\Fboundloc\bigr) \,, \\
\text{or} \quad &  \int_{\widebar{x}+L}^{+\infty}T_{\widebar{x}}\psi_0(x)\Bigl( \frac{1}{2}\abs{u_x(x,t)}_{\ddd} ^2 + V_0\bigl(u(x,t)\bigr) + \frac{1}{2}u(x,t)^2 \Bigr) \, dx \ge \frac{1}{2}\bigl(\fff_0(\widebar{x},t)-\Fboundloc\bigr)
\,.
\end{aligned}
\end{equation}
Take and fix $\widebar{x}$ in $\rr$ and $t$ in $[0,T_{\max})$ such that $\mathcal{Q}(\widebar{x},t) \le \Qfiredecr$, and assume for instance that the first of the two inequalities~\cref{ineq_fire_le_ri} above holds. 
Observe moreover that, according to the choice of $L$, for all $x$ in $(-\infty,\widebar{x}-L]$, 
\[
T_{\widebar{x}-L}\psi_0(x)=\exp(\kappa_0 L)T_{\widebar{x}}\psi_0(x)\ge 3\, T_{\widebar{x}}\psi_0(x)
\,,
\]
thus, since the integrand in $\fff_0(\cdot,\cdot)$ is nonnegative, the first of the two inequalities~\cref{ineq_fire_le_ri} above yields
\[
\fff_0(\widebar{x}-L,t) \ge \frac{3}{2}\bigl(\fff_0(\widebar{x},t)-\Fboundloc\bigr)
\,,
\]
or equivalently
\[
\fff_0(\widebar{x}-L,t) \ge \fff_0(\widebar{x},t) + \frac{1}{2}\bigl(\fff_0(\widebar{x},t)-3\Fboundloc \bigr)
\,,
\]
and this shows that the lemma holds for the following choice of $\Fsuphigher$:
\[
\Fsuphigher=3\Fboundloc+2
\,.
\]
\end{proof}
\paragraph{End of the proof.}
\begin{proof}[Proof of \cref{prop:glob_exist_sol_att_ball_X}]
It follows from \cref{lem:decrease_of_fff_zero_where_Q_is_large,lem:sup_fire_higher} that, for all $t$ in $[0,T_{\max})$,
\[
\fffZeroSup(t)\le \max\bigl(\Fsuphigher,\fffZeroSup(0)-t\bigr)
\,,
\]
thus
\[
\sup_{\widebar{x}\in\rr}\mathcal{Q}(\widebar{x},t)\le \max\bigl(\Fsuphigher,\fffZeroSup(0)-t\bigr)
\,,
\]
and these estimates hold whatever the initial condition $u_0$ in $X$. On the other hand, it follows from the definition of $\mathcal{Q}$ that, for every $\widebar{x}$ in $\rr$, 
\[
\begin{aligned}
\int_{\widebar{x}}^{\widebar{x}+1}\bigl(u(x,t)^2+u_x(x,t)^2\bigr)\, dx &\le \int_{\widebar{x}}^{\widebar{x}+1}\bigl(u(x,t)^2+\frac{1}{\eigDmin}\abs{u_x(x,t)}_\ddd^2\bigr)\, dx \\
&\le \frac{1}{\min(1,\eigDmin)} \int_{\widebar{x}}^{\widebar{x}+1}\bigl(u(x,t)^2 + \abs{u_x(x,t)}_\ddd^2 \bigr)\, dx \\
&\le \frac{2e^{\kappa_0}}{\min(1,\eigDmin)} \qqq(\widebar{x},t)
\,,
\end{aligned}
\]
thus
\[
\begin{aligned}
\norm{x\mapsto u(x,t)}_X^2 &\le \frac{2e^{\kappa_0}}{\min(1,\eigDmin)} \sup_{\widebar{x}\in\rr}\qqq(\widebar{x},t) \\
&\le \frac{2e^{\kappa_0}}{\min(1,\eigDmin)}\max\bigl(\Fsuphigher,\fffZeroSup(0)-t\bigr)
\,,
\end{aligned}
\]
and this last inequality provides the desired outcome: the semi-flow is globally defined and admits an attracting ball in $X$. \Cref{prop:glob_exist_sol_att_ball_X} is proved.
\end{proof}
\subsection{Attracting ball in \texorpdfstring{$\Linfty$}{Linfty(R,Rd)}}
\begin{lemma}[embedding of $\HoneLoc$ into $\Linfty$]
\label{lem:embedding_HoneLoc_into_Linfty}
For every function $u:x\mapsto u(x)$ in $\HoneLoc$, 
\begin{equation}
\label{embedding_HoneLoc_into_Linfty}
\abs{u(0)}\le\sqrt{2 \int_0^1 \bigl(u(x)^2 + u'(x)^2\bigr)\, dx}
\,.
\end{equation}
\end{lemma}
\begin{proof}
For every function $u:x\mapsto u(x)$ in $\HoneLoc$, 
\[
u(0) 
= -\int_0^1 \frac{d}{dx}\bigl((1-x)u(x)\bigr)\, dx 
= \int_0^1 \bigl(u(x)+(x-1)u'(x)\bigr)\, dx
\,,
\]
thus
\[
\begin{aligned}
\abs{u(0)} &\le \int_0^1 \bigl(\abs{u(x)}+\abs{u'(x)}\bigr)\, dx 
\le \sqrt{\int_0^1 \bigl(\abs{u(x)}+\abs{u'(x)}\bigr)^2\, dx} \\
&\le \sqrt{2\int_0^1 \bigl(u(x)^2 + u'(x)^2\bigr)^2\, dx} 
\,.
\end{aligned}
\]
\end{proof}
The following corollary follows from the previous lemma and from the definition \vref{def_Honeul_norm} of the $\Honeul$-norm. 
\begin{corollary}[embedding of $X$ into $\Linfty$]
\label{cor:Linfty_norm_bounded_from_above_by_Honeul}
For every function $u:x\mapsto u(x)$ in $X$, 
\begin{equation}
\label{Linfty_norm_bounded_from_above_by_Honeul}
\norm{u}_{\Linfty} \le \sqrt{2} \norm{u}_{X}
\,.
\end{equation}
\end{corollary}
\Vref{prop:glob_exist_sol_att_ball} (global existence of solutions and attracting ball for the $\Linfty$-norm) follows from \cref{prop:glob_exist_sol_att_ball_X,cor:Linfty_norm_bounded_from_above_by_Honeul}.
\section{Some properties of the profiles of stationary solutions}
\label{sec:prop_ham_app}
This \namecref{sec:prop_ham_app} is devoted to some properties of solutions of the Hamiltonian system \vref{ham_ord_2} governing stationary solutions of system~\cref{init_syst}: 
\begin{equation}
\label{ham_ord_2_bis}
\ddd u'' = \nabla V(u)
\,.
\end{equation}
As everywhere else, let us consider a function $V$ in $\ccc^2(\rr^d,\rr)$ satisfying the coercivity hypothesis \cref{hyp_coerc}. 
\subsection{Asymptotic behaviour in the neighbourhood of a minimum point}
\label{subsec:ham_equ_hyp}
\begin{lemma}[asymptotics of stationary solutions in the neighbourhood of a minimum point]
\label{lem:ham_equ_hyp}
Let $m$ be a point of $\mmm$, and let $\xi\mapsto\phi(\xi)$ be a global solution of the differential system \cref{ham_ord_2_bis} satisfying
\begin{equation}
\label{hyp_lem_ham_equ_hyp}
\abs{\phi(\xi)-m}_{\ddd}\le \dEsc(m)
\quad\text{for every $\xi$ in $[0,+\infty)$}
\quad\text{and}\quad
\phi(\cdot)\not\equiv m
\,.
\end{equation}
Then following assertions hold.
\begin{enumerate}
\item The ordered pair $\bigl(\phi(\xi),\phi'(\xi)\bigr)$ goes to $(m,0)$ (at an exponential rate) as $\xi$ goes to $+\infty$.
\label{item:cv_spatial_asymptotics_sw}
\item For all $\xi$ in $[0,+\infty)$, the scalar product $\bigl\langle\phi(\xi)-m,\phi'(\xi)\bigr\rangle_{\ddd}$ is negative. 
\label{item:transv_spatial_asymptotics_sw}
\item For all $\xi$ in $(0,+\infty)$, the distance $\abs{\phi(\xi)-m}_\ddd$ is smaller than $\dEsc(m)$. 
\label{item:closer_spatial_asymptotics_sw}
\item The supremum $\sup_{\xi\in\rr}\abs{\phi(\xi)-m}_\ddd$ is larger than $\dEsc(m)$. 
\label{item:escape_spatial_asymptotics_sw}
\end{enumerate}
\end{lemma}
\begin{proof}
See \cite[\GlobalBehaviourLemTWApproachCriticalPoints]{Risler_globalBehaviour_2016}.
\end{proof}
\subsection{Normalized Lagrangian integral of stationary solutions with almost zero normalized Hamiltonian}
\label{subsec:infin_lag}
\begin{notation}
Let $\valueOfV$ denote a real quantity, and let us assume that, in addition to hypothesis \cref{hyp_coerc}, the potential $V$ also satisfies hypothesis \textup{(\hyperlink{hypOnlyMin}{\hypOnlyMinRef{\valueOfV}})}. Let us consider (as in definitions \vref{def_V_ddag_asymptotic_energy,def_V_ddag_relaxation}) the ``normalized potential'' $V^\ddag$ and (as in definition \vref{def_normalized_hamiltonian}) the ``normalized Hamiltonian'' $H^\ddag$ and (as in definition \vref{def_normalized_lagrangian}) the ``normalized (with respect to the level $\valueOfV$) Lagrangian'' $L^\ddag$:
\[
V^\ddag(u) = V(u)-h
\quad\text{and}\quad
H^\ddag(u,v) = \frac{1}{2}\abs{v}_{\ddd}^2 - V^\ddag(u)
\quad\text{and}\quad
L^\ddag(u,v) = \frac{1}{2}\abs{v}_{\ddd}^2 + V^\ddag(u)
\,.
\]
\end{notation}
\begin{definition}[normalized Lagrangian integral of a stationary solution]
If $\xi\mapsto u(\xi)$ is a global solution of system~\cref{ham_ord_2_bis}, let us call \emph{normalized Lagrangian integral} of this solution the quantity
\begin{equation}
\label{def_normalized_lagrangian_integral}
\mathcal{L}^\ddag[\xi\mapsto u(\xi)] = \int_{\rr} L^\ddag\bigl(u(\xi),u'(\xi)\bigr) \, d\xi
\,,
\end{equation}
provided that this integral can be unambiguously defined, that is: provided that the integral is convergent, or that it diverges to $+\infty$ at both ends of $\rr$, or that it diverges to $-\infty$ at both ends of $\rr$. 
\end{definition}
The aim of this \namecref{subsec:infin_lag} is to prove the following proposition. Recall (see \vref{definition_PhiZero_of_h}) that $\PhiZero(\valueOfV)$ denotes the set of solutions $\xi\mapsto u(\xi)$ of system~\cref{ham_ord_2_bis} that are homoclinic or heteroclinic to points of $\mmm_{\valueOfV}$. 
\begin{proposition}[stationary solutions having an almost zero normalized Hamiltonian and a finite normalized Lagrangian integral are bistable]
\label{prop:infin_lag}
There exists a positive quantity $\deltaHam$ such that, for every global solution of system~\cref{ham_ord_2_bis}, if
\begin{itemize}
\item the normalized Hamiltonian $H^\ddag$ of this solution is between $-\deltaHam$ and $+\deltaHam$,
\item and this solution does not belong to the set $\PhiZero(\valueOfV)$, 
\end{itemize}
then the normalized Lagrangian integral \cref{def_normalized_lagrangian_integral} of this solution is equal to $+\infty$.
\end{proposition}
Hypothesis \textup{(\hyperlink{hypOnlyMin}{\hypOnlyMinRef{\valueOfV}})} (more precisely, inequality \vref{nonnegative_pot_around_loc_min} stating that the normalized potential $V^\ddag$ takes only nonnegative values around every critical point in the level set $V^{-1}(\{\valueOfV\})$) plays an essential role in the proof of this proposition (which is false if the converse holds). 
\begin{proof}
\renewcommand{\qedsymbol}{}
If $\xi\mapsto u(\xi)$ is a global solution of system~\cref{ham_ord_2_bis}, let
\[
\SigmaEsc[\xi\mapsto u(\xi)] = \SigmaEsc[u(\cdot)] = \{\xi\in\rr : 
\text{ for all } m \text{ in } \mmm_{\valueOfV},  \abs{u(\xi)-m}_{\ddd}>\dEsc(m) \}
\]
(observe the analogy with the notation $\SigmaEsc(t)$ in \cref{subsec:firewall}). 

It follows from inequality \vref{nonnegative_pot_around_loc_min} that, if $\xi\mapsto u(\xi)$ is a global solution of system~\cref{ham_ord_2_bis}, then
\begin{equation}
\label{lag_sig_close}
L^\ddag\bigl(u(\xi),u'(\xi)\bigr) \ge 0 
\quad\text{ for all } \xi \text{ in } \rr\setminus \SigmaEsc[u(\cdot)]\,.
\end{equation}
The proof will follow from the next two lemmas. 
\end{proof} 
\begin{lemma}[non bistable stationary solutions never stop to ``Escape'']
\label{lem:lag_sig_far_infin}
For every global solution $\xi\mapsto u(\xi)$ of system~\cref{ham_ord_2_bis} that is not in $\PhiZero(\valueOfV)$, the set $\SigmaEsc[u(\cdot)]$ is unbounded.
\end{lemma}
\begin{proof}[Proof of \cref{lem:lag_sig_far_infin}]
This lemma is an immediate consequence of \cref{lem:ham_equ_hyp} of the previous \cref{subsec:ham_equ_hyp}.
\end{proof}
\begin{lemma}[almost zero normalized Hamiltonian yields positive normalized Lagrangian at each ``Escape'']
\label{lem:lag_sig_far_pos}
There exist positive quantities $\deltaHam$ and $\deltaLag$ such that, for every global solution $\xi\mapsto u(\xi)$ of system~\cref{ham_ord_2_bis}, if the normalized Hamiltonian $H^\ddag$ of this solution is between $-\deltaHam$ and $+\deltaHam$, then, for every $\widebar{\xi}$ in $\rr$, the following holds:
\[
[\widebar{\xi},\widebar{\xi}+1]\cap \SigmaEsc[u(\cdot)]\not=\emptyset
\implies
\int_{\widebar{\xi}}^{\widebar{\xi}+1} L^\ddag\bigl(u(\xi),u'(\xi)\bigr) \, d\xi \ge \deltaLag
\,.
\]
\end{lemma}
\begin{proof}[Proof of \cref{lem:lag_sig_far_pos}]
Let us proceed by contradiction and assume that, for every positive integer $n$, there exists a global solution $\xi\mapsto u_n(\xi)$ of system~\cref{ham_ord_2_bis} such that the Hamiltonian $H^\ddag$ of this solution is between $-1/p$ and $+1/p$, and such that there exists $\xi_n$ in $\rr$ such that
\[
[\xi_n,\xi_{n+1}]\cap \SigmaEsc[u_n(\cdot)]\not=\emptyset
\quad\text{and}\quad
\int_{\xi_n}^{\xi_{n+1}} L^\ddag\bigl(u_n(\xi),u_n'(\xi)\bigr) \, d\xi \le \frac{1}{n}
\,.
\]
A compactness argument will lead to the sought contradiction. 

For notational convenience, let us assume without loss of generality (up to replacing $\xi\mapsto u_n(\xi)$ by $\xi\mapsto u_n(\xi-\xi_n)$) that $\xi_n$ equals $0$. Then the last estimate reads:
\[
\int_0^1 \Bigl(\frac{1}{2}\abs{u_n'(\xi)}_\ddd^2 + V^\ddag\bigl( u_n(\xi)\bigr) \Bigr) \, d\xi \le \frac{1}{n}
\,,
\]
and since the Hamiltonian of the solution is between $-1/p$ and $+1/p$ the following estimate holds:
\[
\int_0^1 \Bigl(\frac{1}{2}\abs{u_n'(\xi)}_\ddd^2 - V^\ddag\bigl( u_n(\xi)\bigr) \Bigr) \, d\xi 
= \int_0^1 H^\ddag\bigl( u_n(\xi),u_n'(\xi) \bigr) \, d\xi 
\le \frac{1}{n}
\,.
\]
Summing up these two inequalities yields
\begin{equation}
\label{posit_Lag_escape_bound_uprime}
\int_0^1 \abs{u_n'(\xi)}_{\ddd}^2\, d\xi \le \frac{2}{n}
\,,
\end{equation}
and dropping the square term of the integrands in the same two inequalities yields
\begin{equation}
\label{posit_Lag_escape_bound_potential}
-\frac{1}{p} \le \int_0^1 V^\ddag\bigl(u_n(\xi)\bigr)\, d\xi \le \frac{1}{n}
\,.
\end{equation}
According to inequality \cref{posit_Lag_escape_bound_uprime} $u_n(\cdot)$ varies by less than $\sqrt{2/n}$ on $[0,1]$, and according to the inequalities \cref{posit_Lag_escape_bound_potential} $u_n(0)$ is bounded independently of $n$ (indeed according to the coercivity hypothesis \cref{hyp_coerc}, the quantity $V(v)$ goes to $+\infty$ as $\abs{v}$ goes to $+\infty$). 

Thus, up to extracting a subsequence, it may be assumed that the sequence of functions $\xi\mapsto u_n(\xi)$ converges, uniformly on $[0,1]$, towards an equilibrium $u_{\infty}$ of system \cref{ham_ord_2_bis} satisfying:
\[
V^\ddag(u_{\infty})=0
\quad\text{and}\quad
\abs{u_{\infty}-m}_{\ddd}\ge\dEsc(m)
\text{ for all } m \text{ in } \mmm_{\valueOfV}
\,,
\]
a contradiction with the definition of $\mmm_{\valueOfV}$ and hypothesis \textup{(\hyperlink{hypOnlyMin}{\hypOnlyMinRef{\valueOfV}})}.
\end{proof}
\begin{proof}[Proof of \cref{prop:infin_lag}]
Let $\xi\mapsto u(\xi)$ be a global solution of system~\cref{ham_ord_2_bis} such that:
\begin{enumerate}
\item the normalized Hamiltonian $H^\ddag$ of this solution is between $-\deltaHam$ and $+\deltaHam$,
\item and this solution is \emph{not} in $\PhiZero(\valueOfV)$.
\end{enumerate}
Then, for every positive quantity $\widebar{\xi}$ (say larger than $1$),
\[
\int_{0}^{\widebar{\xi}} L^\ddag\bigl(u(\xi),u'(\xi)\bigr) \, d\xi = 
\sum_{i=0}^{\text{int}(\widebar{\xi})-1} \int_{i}^{i+1} L^\ddag\bigl(u(\xi),u'(\xi)\bigr) \, d\xi
+ 
\int_{\text{frac}(\widebar{\xi})}^{\widebar{\xi}} L^\ddag\bigl(u(\xi),u'(\xi)\bigr) \, d\xi
\,, 
\]
and the $i$-th term under the sum of the right-hand side of this equality is:
\begin{itemize}
\item nonnegative if the intersection
\[
\bigl[i,i+1\bigr]\cap \SigmaEsc[u(\cdot)]
\]
is empty (according to assertion \cref{lag_sig_close}), 
\item greater than or equal to $\deltaLag$ if this intersection is nonempty (in view of \cref{lem:lag_sig_far_pos} about the non-negativity of $L^\ddag\bigl(u(\cdot),u'(\cdot)\bigr)$), 
\end{itemize}
and according to \cref{lem:lag_sig_far_infin} the second of these two alternatives occurs for an unbounded number of values of $i$ as $\widebar{\xi}$ goes to $+\infty$. In addition, the remaining term of the right-hand side of this inequality is bounded from below (as is $V(\cdot)$ and thus as is $V^\ddag(\cdot)$). 
As a consequence (applying the symmetric argument at the left of $0$), both quantities  
\[
\int_{0}^{\widebar{\xi}} L^\ddag\bigl(u(\xi),u'(\xi)\bigr) \, d\xi
\quad\text{and}\quad
\int_{-\widebar{\xi}}^0 L^\ddag\bigl(u(\xi),u'(\xi)\bigr) \, d\xi
\]
go to $+\infty$ as $\widebar{\xi}$ goes to $+\infty$. \Cref{prop:infin_lag} is proved. 
\end{proof}
\section{The space of asymptotic patterns}
\label{sec:sp_as_patt}
The aim of this \namecref{sec:sp_as_patt} is to make a few (rather abstract) remarks concerning the regularity (more precisely, the upper semi-continuity) of the correspondence between an initial condition and the distribution of energy in the standing terrace provided by conclusion \cref{item:thm_main_approach_standing_terrace_and_value_asymptotic_energy} of \cref{thm:main} when the asymptotic energy of the corresponding solution is not equal to $-\infty$. 

Let us assume that the potential $V$ satisfies hypothesis \cref{hyp_coerc}. Let $\valueOfV$ be a real quantity, and let us assume that hypotheses \textup{(\hyperlink{hypOnlyMin}{\hypOnlyMinRef{\valueOfV}})} and \textup{(\hyperlink{hypDiscStat}{\hypDiscStatRef{\valueOfV}})} hold. For every ordered pair $(m_-,m_+)$ of points of $\mmm_{\valueOfV}$, let us introduce the space
\[
\XbistNoInv(m_-,m_+) = \Xbist(m_-,m_+) \cap \eeeAsympt^{-1}\bigl( [0,+\infty) \bigr) 
\,.
\]
In this notation, the additional subscript ``no-inv'' refers to the fact, that, for those initial conditions, the stable equilibria at both ends of space are not ``invaded'' by travelling fronts. Indeed, \cite[\GlobalBehaviourPropNonInvasionImpliesRelaxation]{Risler_globalBehaviour_2016} states (under the additional hypothesis that the diffusion matrix $\ddd$ is the identity matrix) that solutions in $\Xbist(m_-,m_+)$ having an asymptotic energy equal to $-\infty$ are exactly those for which the equilibria at both ends of space are invaded by bistable travelling fronts. 

For every $u_0$ in $\XbistNoInv(m_-,m_+)$, let us denote by $\qAsympt[\phi_0]$ the ``number of items in the standing terrace'' approached by the corresponding solution. This defines a map
\begin{equation}
\label{def_map_numb_con}
\qAsympt:\XbistNoInv(m_-,m_+)\to\nn
\,.
\end{equation}
As an example of use of this notation, observe that, for every point $m$ in $\mmm_{\valueOfV}$, 
\[
\basatt(m) = \XbistNoInv(m,m) \cap \qAsympt^{-1} \bigl(\{0\}\bigr)
\,.
\]
The following proposition is a consequence of \vref{cor:neg_dyn}.
\begin{proposition}[the number of items in the standing terrace is not lower semi-continuous with respect to the initial condition]
Assume that hypotheses \cref{hyp_coerc} and \textup{(\hyperlink{hypOnlyMin}{\hypOnlyMinRef{\valueOfV}})} and \textup{(\hyperlink{hypDiscStat}{\hypDiscStatRef{\valueOfV}})} hold, and assume in addition that the global minimum value of $V$ is less than $\valueOfV$. Then the number of items in the asymptotic standing terrace approached by the solution is \emph{not} lower semi-continuous with respect to the initial condition. In more formal terms, the map $\qAsympt[\cdot]$ defined in~\cref{def_map_numb_con} is \emph{not} lower semi-continuous.
\end{proposition}
\begin{proof}
According to \cref{cor:neg_dyn}, for every $m$ in $\mmm_{\valueOfV}$, the set $\partial \basatt(m)$ is nonempty, and for every initial condition $u_0$, in this set, the integer $\qAsympt[u_0]$ is nonzero. On the other hand, by definition of the topological border, $u_0$ is arbitrarily close to initial conditions in $\basatt(m)$, and for those initial condition the integer $\qAsympt[\cdot]$ is zero.
\end{proof}
It is likely that this map $\qAsympt[\cdot]$ is \emph{not} upper semi-continuous in general (thus neither lower nor upper semi-continuous, in general). It would be interesting however to build an explicit example of a potential $V$ for which $\qAsympt[\cdot]$ is not upper semi-continuous (say, for which an unstable pulse may split into two repulsive ``smaller'' pulses). The conclusion that can be drawn from this observation is that the definition~\cref{def_map_numb_con} of the map $\qAsympt[\cdot]$ is ``irrelevant'' (let us say: ``bad''), in the sense that it does not ensure upper semi-continuity. By contrast, any ``good'' definition of an asymptotic feature of a solution should display some form of upper semi-continuity. In this sense, the asymptotic energy defined in \cref{subsubsec:asympt_en} is a ``good'' feature. 

Unfortunately, the following definitions will turn to be naively ``bad''. Thus the sole interest of the next lines is to raise the question of what would be the ``good'' definitions to choose in place of these ``bad'' ones. 

Let us introduce the following spaces (``bad'' space of asymptotic profiles and ``bad'' space of asymptotic energy distributions):
\[
\pppBad = \rr^d \cup \bigsqcup_{q\in\nn^*} \bigl( \ccc^3(\rr,\rr^d) \cap H^1(\rr,\rr^d) \bigr)^q
\quad\text{and}\quad
\eeeBad = \{0\} \cup \bigsqcup_{q\in\nn^*}\rr_+^q
\,.
\]
Conclusion \cref{item:thm_main_approach_standing_terrace_and_value_asymptotic_energy} of \cref{thm:main} leads us to define the following map, which sends an initial condition to the profiles of the standing terrace approached by the solution (let us denote by $\phi_1,\dots, \phi_{\qAsympt[u_0]}$ these profiles if $\qAsympt[u_0]$ is positive):
\[
\ppp_\infty: \XbistNoInv(m_-,m_+) \to \pppBad \,,
\quad
u_0 \mapsto
\left\{
\begin{aligned}
m_+
&\quad\text{if}\quad \qAsympt[u_0] = 0 \,, \\
(\phi_1,\dots,\phi_{\qAsympt[u_0]})
&\quad\text{if}\quad \qAsympt[u_0] > 0 \,,
\end{aligned}
\right.
\]
and the following map, that sends an ``asymptotic pattern'' to the corresponding ``distribution of asymptotic energies'':
\[
\eee: \pppBad \to \eeeBad \,,
\quad
m_+ \mapsto 0 \,,
\quad
(\phi_1,\dots,\phi_q) \mapsto \bigl(\eee[\phi_1], \dots, \eee[\phi_q]\bigr)
\,,
\]
and the following map, that does nothing more than summing up the components of a ``distribution of asymptotic energies'':
\[
\Sigma : \eeeBad \to [0,+\infty) \,,
\quad
0 \mapsto 0 \,,
\quad
(E_1,\dots,E_q) \mapsto \sum_{i=1}^q E_i
\,,
\]
and the following map, that simply counts the number of items in the asymptotic pattern:
\[
\text{card} : \eeeBad \to \nn \,,
\quad
0 \mapsto 0 \,
\quad
(E_1,\dots,E_q) \mapsto q
\,.
\]
As already mentioned, it is likely that the map
\[
\qAsympt = \text{card} \circ \eee\circ\ppp_\infty
\]
is not upper semi-continuous, whereas by contrast \cref{prop:scs_asympt_en} states that the map
\[
\eeeAsympt = \Sigma \circ \eee\circ\ppp_\infty
\]
is upper semi-continuous. 

Unfortunately, there is no hope that, with the definitions above, the map $\eee\circ\ppp_\infty$ may display any kind of upper semi-continuity. The sole goodness of the spaces $\pppBad$ and $\eeeBad$ is that they bear a partial order that is relevant (only in space dimension one) with respect to the phenomenon under consideration, but this is far from being sufficient to ensure the desired upper semi-continuity. The problem of finding proper definitions for these two spaces so that the map $\eee\circ\ppp_\infty$ (together with the map ``counting the number of items in the standing terrace'') be upper semi-continuous is beyond the scope of this paper.

The results of \cite{Risler_globalBehaviour_2016} (global behaviour of all bistable solutions under generic assumptions on the potential) raise the same kind of questions about the topological structure of the asymptotic pattern of every bistable solutions (and not only those of the set $\XbistNoInv(m_-,m_+)$), including the travelling fronts involved in this asymptotic pattern and their speeds. 
\subsubsection*{Acknowledgements}
I am indebted to Thierry Gallay and Romain Joly for their help and interest through numerous fruitful discussions. 
\emergencystretch=1em
\printbibliography 
\bigskip
\mySignature
\end{document}